\documentclass[10pt]{amsart}
\usepackage{bbm}

\usepackage{amsmath, amsthm, amssymb}
\usepackage{enumitem}
\usepackage{mathrsfs}
\usepackage{calc}
\usepackage{graphicx}
\usepackage[all,cmtip]{xy}
\usepackage{csquotes}
\usepackage{amscd}
\usepackage{lmodern}
\usepackage{amsfonts}
\usepackage{picinpar}
\usepackage{tikz-cd}
\usepackage{pgf,tikz,pgfplots}
\pgfplotsset{compat=1.14}
\usepackage{mathrsfs}
\usetikzlibrary{arrows}
\usepackage{mathptmx}
\usepackage{stmaryrd}
\usepackage{array}
\usepackage{extarrows}
\usepackage{hyperref} 
\usepackage{longtable}
\usepackage{verbatim}
\usepackage{esint}
\usepackage[T1]{fontenc}

\newcolumntype{P}[1]{>{\raggedright\let\newline\\\arraybackslash\hspace{0pt}}p{#1}}

\numberwithin{equation}{section}

\makeatletter
\def\paragraph{\@startsection{paragraph}{4}%
 \z@\z@{-\fontdimen2\font}%
 {\normalfont\bfseries}}
\makeatother

\makeatletter
\newcommand{\proofstep}[1]{%
  \par
  \addvspace{\medskipamount}
  \textit{#1\@addpunct{.}}\enspace\ignorespaces
}
\makeatother

\DeclareMathOperator{\GL}{GL}
\DeclareMathOperator{\PGL}{PGL}

\DeclareMathOperator{\PU}{PU}
\DeclareMathOperator{\PSL}{PSL}

\DeclareMathOperator{\Bir}{Bir}
\DeclareMathOperator{\Aut}{Aut}
\DeclareMathOperator{\Aff}{Aff}

\DeclareMathOperator{\Id}{Id}

\DeclareMathOperator{\Iso}{Isom}
\DeclareMathOperator{\Jonq}{Jonq}
\DeclareMathOperator{\Cent}{Cent}

\begin{document}

\theoremstyle{plain}
\newtheorem{theorem}{Theorem}
\numberwithin{theorem}{subsection}
\newtheorem{exercise}{Exercise}
\newtheorem{corollary}[theorem]{Corollary}
\newtheorem{claim}[theorem]{Claim}
\newtheorem{lemma}[theorem]{Lemma}
\newtheorem{proposition}[theorem]{Proposition}
\newtheorem{conjecture}[theorem]{Conjecture}
\newtheorem{maintheorem}{Theorem}
\newtheorem{maincor}[maintheorem]{Corollary}
\newtheorem{mainproposition}[maintheorem]{Proposition}
\renewcommand{\themaintheorem}{\Alph{maintheorem}}

\theoremstyle{definition}
\newtheorem{fact}[theorem]{Fact}
\newtheorem{definition}[theorem]{Definition}
\newtheorem{remark}[theorem]{Remark}
\newtheorem{question}[theorem]{Question}
\newtheorem{example}[theorem]{Example}

\def\bA{{\mathbb{A}}}
\def\bB{{\mathbb{B}}}
\def\bC{{\mathbb{C}}}
\def\bD{{\mathbb{D}}}
\def\bR{{\mathbb{R}}}
\def\bS{{\mathbb{S}}}
\def\bO{{\mathbb{O}}}
\def\bE{{\mathbb{E}}}
\def\bF{{\mathbb{F}}}
\def\bH{{\mathbb{H}}}
\def\bI{{\mathbb{I}}}
\def\bT{{\mathbb{T}}}
\def\bZ{{\mathbb{Z}}}
\def\bX{{\mathbb{X}}}
\def\bP{{\mathbb{P}}}
\def\bN{{\mathbb{N}}}
\def\bQ{{\mathbb{Q}}}
\def\bK{{\mathbb{K}}}
\def\bG{{\mathbb{G}}}

\def\nrj{{\mathcal{E}}}
\def\cA{{\mathscr{A}}}
\def\cB{{\mathscr{B}}}
\def\cC{{\mathscr{C}}}
\def\cD{{\mathscr{D}}}
\def\cE{{\mathscr{E}}}
\def\cF{{\mathscr{F}}}
\def\cB{{\mathscr{G}}}
\def\cH{{\mathscr{H}}}
\def\cI{{\mathscr{I}}}
\def\cJ{{\mathscr{J}}}
\def\cK{{\mathscr{K}}}
\def\Layer{{\rm Layer}}
\def\cM{{\mathscr{M}}}
\def\cN{{\mathscr{N}}}
\def\cO{{\mathscr{O}}}
\def\cP{{\mathscr{P}}}
\def\cQ{{\mathscr{Q}}}
\def\cR{{\mathscr{R}}}
\def\cS{{\mathscr{S}}}
\def\Up{{\rm Up}}
\def\cU{{\mathscr{U}}}
\def\cV{{\mathscr{V}}}
\def\cW{{\mathscr{W}}}
\def\cX{{\mathscr{X}}}
\def\cY{{\mathscr{Y}}}
\def\cZ{{\mathscr{Z}}}

\def\CC{{\mathbf{C}}}
\def\RR{{\mathbf{R}}}
\def\NN{{\mathbf{N}}}
\def\QQ{{\mathbf{Q}}}
\def\ZZ{{\mathbf{Z}}}
\def\fk{{\mathbf{k}}}

\newcommand{\Ort}[2]{\operatorname{O}_{#1}(#2)}

\newcommand{\OL}{\mathcal{O}_{\mathbf{L}}}
\newcommand{\PL}{\mathbf{P}_{\mathbf{L}}}
\newcommand{\PP}{\mathbb{P}}
\newcommand{\Fol}{\mathcal{F}}
\newcommand{\PPo}{\mathbb{P}^1}
\newcommand{\PPt}{\mathbb{P}^2}
\newcommand{\PPPP}{\mathbb{P}^1\times\mathbb{P}^1}
\newcommand{\SSS}{\mathbb{S}}
\newcommand{\HH}{\mathbb{H}}
\newcommand{\Hir}{\mathscr{H}}
\newcommand{\OO}{\mathcal{O}}
\newcommand{\EE}{\mathcal{E}}
\newcommand{\PicMt}{\mathscr{Z}}
\newcommand{\HHabc}{\mathbb{H}^{m} (m\in \mathbf{N}\cup \{\infty\})}
\newcommand{\KK}{\mathscr{K}}
\newcommand{\Fp}{\mathbf{F}_p}
\newcommand{\CRC}{\operatorname{Bir}(\mathbb{P}^2)}
\newcommand{\Kah}{K\"{a}hler}
\newcommand{\Mob}{M\"{o}bius}
\newcommand{\Jonqui}{Jonqui\`eres }
\newcommand{\pisigma}{\pi_1^{orb}(\Sigma)}
\newcommand{\boundh}[1]{\partial _h#1}
\newcommand{\Dev}{\operatorname{Dev}}
\newcommand{\Hol}{\operatorname{Hol}}
\newcommand{\cad}{c.-\`a-d.}

\newcommand{\llpar}{(\!(}
\newcommand{\rrpar}{)\!)}
\newcommand{\GLa}[2]{\operatorname{GL}_{#1}(\mathbf{#2})}
\newcommand{\PGLa}[2]{\operatorname{PGL}_{#1}(\mathbf{#2})}
\newcommand{\PGLthrK}{\operatorname{PGL}_{3}(\mathbf{K})}
\newcommand{\AUK}{\operatorname{Aut}(\mathbb{P}_{\mathbf{K}}^2)}
\newcommand{\CRK}{\operatorname{Cr}_{2}(\mathbf{K})}
\newcommand{\CRn}{\operatorname{Cr}_{n}(\mathbf{C})}
\newcommand{\charK}{\operatorname{char}(\mathbf{K})}

\newcommand{\JonqK}{\operatorname{Jonq}(\mathbf{K})}
\newcommand{\Jonqz}{\operatorname{Jonq}_0(\mathbf{K})}
\newcommand{\PGLtK}{\operatorname{PGL}_{2}(\mathbf{K})}
\newcommand{\Sym}{\mathscr{S}}
\newcommand{\Fpb}{\overline{\mathbf{F}_p}}
\newcommand{\dashmapsto}{\mapstochar\dashrightarrow}

\title{Birational Kleinian groups}

\author[S. Zhao]{Shengyuan Zhao}
\address[S. Zhao]{Universit\'e Toulouse III Paul Sabatier\\
Institut de Math\'ematiques de Toulouse\\
118, route de Narbonne\\
F-31062 Toulouse Cedex 9, France}
\email{shengyuan.zhao@math.univ-toulouse.fr}

\begin{abstract}
In this paper we study birational Kleinian groups, i.e.\ groups of birational transformations of complex projective varieties acting in a free, properly discontinuous and cocompact way on an open subset of the variety with respect to the usual topology. We obtain a classification in dimension two.
\end{abstract}

\maketitle

\setcounter{tocdepth}{1}
\tableofcontents

\section{Introduction}

\subsection{Birational Kleinian groups}

\subsubsection{Definitions}\label{definitionssection}
According to Klein's 1872 Erlangen program a geometry is a space with a group of transformations. One of the implementations of this idea is Ehresmann's notion of geometric structures (\cite{Ehr36}) which models a space locally on a geometry in the sense of Klein. In his Erlangen program, Klein conceived a geometry modeled upon birational transformations:
\begin{displayquote}[{\cite{KleinG}, \cite{KleinE}{8.1}}]
Of such a geometry of rational transformations as must result on the basis of the transformations of the first kind, only a beginning has so far been made...
\end{displayquote}
In this paper we study the group of birational transformations of a variety from Klein and Ehresmann's point of view.

Let $Y$ be a smooth complex projective variety and $U\subset Y$ a connected open set in the usual topology of $Y$ considered as a smooth manifold. Let $\Gamma\subset \Bir (Y)$ be an \emph{infinite} group of birational transformations. We impose the following conditions on $\Gamma$:
\begin{enumerate}
	\item the points of indeterminacy of $\Gamma$ are disjoint from $U$ and $\Gamma$ preserves $U$, in other words each element of $\Gamma$ induces a biholomorphism of $U$;
	\item the action of $\Gamma$ on $U$ is free, properly discontinuous and cocompact.
\end{enumerate}
Therefore $U$ is identified with an infinite Galois covering of the quotient compact complex manifold $X=U/\Gamma$ and $\Gamma$ is identified with the group of deck transformations. We call the data $(Y,\Gamma,U,X)$ a \emph{birational kleinian group}. When the context is clear, we will just call $\Gamma$ a birational kleinian group. 

We say that two birational kleinian groups $(Y,\Gamma,U,X)$ and $(Y',\Gamma',U',X')$ are \emph{geometrically conjugate} if there exists a birational map $\varphi: Y\dashrightarrow Y'$ such that
\begin{enumerate}
	\item $\phi$ is regular on $U$ and sends $U$ biholomorphically to $U'$;
	\item $\Gamma'=\varphi \Gamma  \varphi^{-1}$.
\end{enumerate}
The corresponding quotient manifolds $X$ and $X'$ are isomorphic. If $(Y,\Gamma,U,X)$ is a birational kleinian group and if $\Gamma'$ is a finite index subgroup of $\Gamma$, then $(Y,\Gamma',U,X')$ is also a birational kleinian group where $X'$ is a finite unramified covering of $X$.

Two well-known families of examples are given by complex tori and locally symmetric varieties:
\begin{enumerate}
	\item $Y=\PP^n$, $U=\CC^n$, $\Gamma$ is a lattice in $\CC^n$ and $X$ is a complex torus.
	\item $U$ is a hermitian symmetric space of noncompact type, $Y$ is the associated hermitian symmetric space of compact type and $\Gamma$ is a cocompact lattice in the corresponding Lie group.
\end{enumerate}

The problem of classification of birational kleinian groups was posed by Shing-Tung Yau, Fedor Bogomolov and Serge Cantat. 

\subsubsection{Comparison with classical Kleinian groups}\label{classicalkleinian}
Nowadays a \emph{Kleinian group} means an arbitrary discrete subgroup of $\PGL(2,\CC)$. In this text we will call it \emph{classical Kleinian groups} to avoid confusion. If $\Gamma\subset\PGL(2,\CC)$ is a classical Kleinian group, then there is a unique maximal $\Gamma$-invariant open (not necessarily connected) subset $\Omega_{\Gamma}\subset\PP^1$ on which $\Gamma$ acts discontinuously; it is called the \emph{domain of discontinuity} of $\Gamma$. The complement $\PP^1\backslash \Omega_{\Gamma}$ is called the \emph{limit set} of $\Gamma$. 

If $\Gamma$ is a torsion free finitely generated classical Kleinian group then Ahlfors finiteness theorem asserts that $\Omega_{\Gamma}/\Gamma$ is a finite union of Riemann surfaces of finite type, i.e.\ smooth quasiprojective curves (see \cite{Ahl64}, \cite{Sul85}). If $\Gamma$ is a torsion free finitely generated classical Kleinian group such that $\Omega_{\Gamma}$ has an invariant connected component $\Omega_{\Gamma}^0$, then $\Omega_{\Gamma}^0/\Gamma$ is a connected smooth quasiprojective curve. Our definition of birational Kleinian group is a direct generalization of such classical Kleinian groups.

For more general spaces the notion of domain of discontinuity and that of limit set are not well defined (see \cite{Kulk78}). In dimension $\geq 2$ a group of automorphisms of a projective variety may act in a free, properly discontinuous and cocompact way on different invariant open subsets. For example let a Fuchsian surface group act diagonally on $\PP^1\times \PP^1$, then both $\bD\times \PP^1$ and $\PP^1\times \bD$ can serve as domain of discontinuity (see \cite{CNS}, \cite{KLP18} for more examples). This is why we define a birational Kleinian group as a quadruple $(Y,\Gamma,U,X)$: $Y,\Gamma,U$ determine $X$ but $U$ is not always uniquely determined by $\Gamma$.

\subsubsection{Complex projective Kleinian groups and Anosov subgroups}
If a birational Kleinian group $(Y,\Gamma,U,X)$ satisfies that $Y=\PP^n$ and $\Gamma\subset \PGL_{n+1}(\CC)$, then we call it a \emph{complex projective Kleinian group}. Complex projective Kleinian groups have been classified in dimension two by Kobayashi-Ochiai, Mok-Yeung, Klingler and Cano-Seade (see \cite{KobOch80}, \cite{MY93}, \cite{Kli98}, \cite{CanoSeade14}). If the quotient manifold is projective then they have been almost classified in any dimension $\geq 2$ (see \cite{JahRad15}). A reference on the subject is the book \cite{CNS} where the authors use the terminology ``complex Kleinian group''. 

Another class of birational Kleinian groups that has been studied is the so called \emph{Anosov subgroups}. Guichard-Wienhard \cite{GW12} and Kapovich-Leeb-Porti \cite{KLP18} constructed many birational Kleinian groups $(Y,\Gamma,U,X)$ such that $Y$ is a flag variety $G/P$, the quotient of a complex linear group $G$ by a parabolic subgroup $P$, and $\Gamma$ is an Anosov subgroup of $G$. Note that in many known examples of complex projective Kleinian groups (see \cite{CNS}) or Anosov subgroups (see \cite{DS20}), the quotient compact complex manifolds are not \Kah.

For complex projective Kleinian groups or Anosov subgroups, transformations in $\Gamma$ are in the connected component of identity of the automorphism group of the variety (which is a Lie group). Such an automorphism acts trivially on the cohomology of the variety. An automorphism acting non trivially on cohomology has more complicated dynamical properties (see \cite{Canicm}). In general a group of birational transformations acting on a variety does not necessarily give rise to a corresponding action on the cohomology of the variety. When $n\geq 2$ the group of birational transformations of $\PP^n$ is much larger than $\PGL_n(\CC)$ and has complicated topological structure (see \cite{BF13}). In particular one needs an infinite number of parameters to describe it. We emphasize also that the action of a group of birational transformations is not really a set-theoretic group action on the whole variety: a birational transformation could fail to be well defined at some points, or fail to be locally a diffeomorphism at some points where it is defined. These are some of the reasons why birational Kleinian groups seem to be much more complicated. Beyond complex projective Kleinian groups and Anosov subgroups, even if we restrict to automorphisms, there is essentially nothing known.

\subsubsection{Uniformization and geometric structures}
Every simply connected Riemann surface is biholomorphic to $\bP^1$, $\CC$ or the unit disk $\bD$ by Koebe-Poincar\'e's uniformization theorem which plays an omnipresent role in the study of classical Kleinian groups. Universal covers of projective varieties of dimension $\geq 2$ can be very complicated, not to mention non-projective manifolds. A reasonable way to study them is by adding additional hypotheses. The three elements of the uniformization are: the fundamental group, the universal cover and the action of the fundamental group on it. There are various ways to specify these data. 
From this point of view, our paper is similar in theme to Claudon-H\"oring-Koll\'ar's work \cite{CHK13}. In \cite{CHK13} it is proved under the abundance conjecture that the only projective varieties with quasi-projective universal covers are bundles over abelian varieties or their finite quotients. Koll\'ar-Pardon investigated in \cite{KP12} the situation where we replace quasi-projective universal covers with semi-algebraic universal covers, and abelian varieties with quotients of noncompact hermitian symmetric spaces. A similar statement was proved with additional assumptions. Recall that abelian varieties and locally symmetric varieties are two basic examples of birational Kleinian groups. 

The results of \cite{CHK13} and \cite{KP12} suggest that requiring universal covers to be in some sense algebraic is very restrictive. Our study amounts to assume in the problem of uniformization an algebraic hypothesis on the action of deck transformations instead of on the universal cover itself. We should think of it as uniformization with geometric models. More precisely it is a uniformization with Ehresmann's geometric structures, i.e.\ the so called $(G,X)$-structures. A $(G,X)$-structure on a manifold is an atlas of charts with values in $X$ and locally constant changes of coordinates in $G$. Classically $G$ is a Lie group acting transitively on some model manifold $X$. Here we consider the more general notion of $(\Bir(Y),Y)$-structures. For example a $(\Bir(\bP^n),\bP^n)$-structure is an atlas of charts such that the changes of coordinates can be written as rational fractions. Quotient manifolds of birational Kleinian groups can be considered as the simplest $(\Bir(Y),Y)$-manifolds; they have injective developing maps. We refer to \cite{Dlo16} and \cite{Zhaobirinoue} for $(\Bir(Y),Y)$-structures. A similar notion called bi-algebraic structure has been introduced in \cite{KUY16}.

There are very few known examples of domains in algebraic varieties of dimension $>1$ that are infinite covers of projective varieties. We expect such domains to be holomorphically convex (see \cite{EKPR12}). Let $m\geq 2$. Wong \cite{Wong77} and Rosay \cite{Ros79} proved that if a bounded domain in $\CC^m$ covers a projective variety and has smooth boundary then it is biholomorphic to the unit ball. Frankel \cite{Frankel89} proved that the only convex bounded domains in $\CC^m$ that cover projective varieties are bounded symmetric domains. These results suggest that domains covering projective varieties are hard to describe if they are not bounded symmetric domains. The only other examples of domains that the author knows are universal covers of Kodaira fibrations and their variants. A Kodaira fibration is a projective surface with a submersive non-isotrivial fibration onto a hyperbolic Riemann surface which is a differentiable fiber bundle (see \cite{Kod67}). It follows from Bers' work on Teichm\"uller theory that the universal cover of a Kodaira fibration is biholomorphic to a bounded domain in $\CC^2$. We will see that there is no birational Kleinian group $(Y,\Gamma,U,X)$ such that $X$ is a Kodaira fibration. We also mention that Griffths proved in \cite{Gri71}, with an idea similar to that of Kodaira surfaces, that any point in a projective manifold has a Zariski neiborhood covered by a bounded domain.

\subsubsection{Non compact quotients}
In our definition of a birational kleinian group $(Y,\Gamma,U,X)$, we assumed that the quotient $X$ is compact. We are going to see in this paper that the cocompactness assumption makes birational kleinian groups quite rigid, at least in dimension two, in the sense that there are not many possible examples. If we impose that the action of $\Gamma$ on $U$ is properly discontinuous without being cocompact, then there are some examples with interesting dynamics related to Teichm\"uller theory, in dimension $2$ and also in higher dimension (see \cite{Min13}, \cite{MPT15}, \cite{HTZ18}, \cite{RR21}). However I do not know any interesting example for which $X$ is quasi-projective.

\subsection{Main results}\label{mainresultssection}
First we observe that birational Kleinian groups only exist on those varieties whose groups of birational transformations are large enough.
\begin{maintheorem}\label{nonnegthm}
There are no birational Kleinian groups acting on projective varieties of Kodaira dimension $\geq 0$.
\end{maintheorem}

The core of our paper concerns only dimension two. We obtain an almost complete classification of birational Kleinian groups in dimension two. Our result is new even for groups of automorphisms.

\begin{maintheorem}\label{verylongthmintro}
Let $(Y,U,\Gamma,X)$ be a birational Kleinian group in dimension $2$. Suppose that if $\Gamma$ is virtually cyclic or if $X$ is a class VII surface then $\Gamma$ contains no loxodromic elements. Then up to geometric conjugation and up to taking a finite index subgroup, we are in one of the cases in the table in Section \ref{longlistsection}.
\end{maintheorem}

Here we state a simplified version of our classification in the case of \Kah{} quotients and simply connected domains. There are two hermitian symmetric spaces of noncompact type in dimension two: the Euclidean ball $\bB^2$ in $\bP^2$ and the bidisk $\bD\times\bD$ in $\bP^1\times\bP^1$. Roughly speaking our classification says that apart from these two cases, all other birational Kleinian groups having \Kah{} quotients come from one-dimensional classical Kleinian groups. 
\begin{maintheorem}\label{themainthm}
Let $(Y,U,\Gamma,X)$ be a birational Kleinian group in dimension two such that $U$ is simply connected and $X$ is \Kah{}. Then up to geometric conjugation and up to taking a finite index subgroup, we are in one of the following six situations:
\begin{enumerate}
 \item $Y=\bP^2$, $U=\bC^2$, $\Gamma$ is a group of translations and $X$ is a complex torus.
 \item $Y=\PP^1\times \PP^1$, $U$ is the bidisk $\bD\times\bD$ and $ \Gamma$ is a torsion free irreducible cocompact lattice in $\PSL(2,\RR)\times\PSL(2,\RR)\subset \Aut(\PP^1)\times\Aut(\PP^1)$.
 \item $Y=\PP^2$, $U$ is the Euclidean ball $\bB^2$, $ \Gamma$ is a torsion free cocompact lattice in $\PU(1,2)\subset \PGL(3,\CC)=\Aut(\PP^2)$.
  \item $\Gamma=\Gamma_1\times\Gamma_2$ is a product of two classical Kleinian groups (one possibly trivial), $U$ is a product $D_1\times D_2$ embedded in $Y=\PP^1\times\PP^1$, where $D_i\subset \PP^1$ is a simply connected invariant component of the domain of discontinuity of $\Gamma_i$.
	\item $Y=\PP^1\times\PP^1$, $U=D_1\times \CC$ where $D_1$ is a simply connected invariant component of the domain of discontinuity of a classical Kleinian group $\Gamma_1$. The projection of $\PP^1\times\PP^1$ onto the first factor is equivariant under a surjective homomorphism $\Gamma\rightarrow \Gamma_1$. The quotient surface $X$ is a principal elliptic bundle over $D_1/\Gamma_1$. Any element of $\Gamma$ has the form $(x,y)\dashmapsto (\gamma_1(x),y+R(x))$ where $\gamma_1\in\Gamma_1$ and $R\in \CC(x)$ has no poles in $D_1$. 	
	\item $Y=\PP^1\times\PP^1$, $U=D_1\times \PP^1$ where $D_1$ is a simply connected invariant component of the domain of discontinuity of a classical Kleinian group $\Gamma_1$. The projection of $\PP^1\times\PP^1$ onto the first factor is equivariant under an isomorphism $\Gamma\rightarrow \Gamma_1$. The quotient surface $X$ is a geometrically ruled surface over $D_1/\Gamma_1$. 
\end{enumerate}
\end{maintheorem}

Classical Kleinian groups that appear in Theorem \ref{themainthm} have simply connected invariant components in their domains of discontinuity. Examples are fuchsian groups, quasi-fuchsian groups and their limits (\cite{Bers70}, \cite{Mas88}). Any such group is a limit of quasi-fuchsian groups (see \cite{BCM12}).

\subsection{Strategy and plan}

\subsubsection{Strategy}

Let $(Y,U,\Gamma,X)$ be a birational Kleinian group. For simplicity and to explain the strategy in the most important case, we assume that $Y$ is a rational surface, $X$ is \Kah{} and $U$ is simply connected. Thus $\Gamma$ can be identified with the fundamental group $\pi_1(X)$. We give here a sketch of the main steps of the proof of Theorem \ref{themainthm}, which is a mixture of three topics: dynamics of birational transformations, fundamental groups of compact \Kah{} manifolds (a.k.a \Kah{} groups) and holomorphic foliations. 

Manin \cite{Manin86} introduced a faithful action of the group of birational transformations $\Bir(Y)$ by isometries on an infinite dimensional hyperbolic space $\HH_Y$ (see \cite{Can11}). Thus we have an action of $\Gamma=\pi_1(X)$ on $\HH_Y$. Very strong restrictions on isometric actions of \Kah{} groups on finite dimensional hyperbolic spaces are found by Carleson-Toledo \cite{CT} and they are generalized by Delzant-Py \cite{DelPy} to infinite dimensional hyperbolic spaces; their works are based on the theory of harmonic mappings on \Kah{} manifolds. We combine Delzant-Py's results with the Kleinian property of $\Gamma$ to prove that the action of $\Gamma$ on $\HH_Y$ has fixed points in $\HH_Y$ or on the boundary of $\HH_Y$. Cantat proved in \cite{Can11} that the actions on $Y$ of finitely generated groups of birational transformations with fixed points in $\HH_Y\cup \partial \HH_Y$ have highly constrained geometric features. We deduce from Cantat's result that either $\Gamma$ is conjugate to a group of automorphisms, or to a group preserving a rational fibration. 

In the first case, it is possible to further reduce the situation to the case of complex projective Kleinian groups. We can thus apply the classification obtained by Kobayashi-Ochiai \cite{KobOch80}, Mok-Yeung \cite{MY93}, Klingler \cite{Kli98} and Cano-Seade \cite{CanoSeade14}. 

Consider now the case where $\Gamma$ preserves a rational fibration. We observe that the rational fibration induces a regular holomorphic foliation on $X$. Brunella did an almost complete classification of regular holomorphic foliations on compact complex surfaces in \cite{Bru97}. The foliations appearing in our study have additional structures: they support foliated $(\PGL(2,\CC),\CC)$-structures along the leaves. Recently in \cite{DG} Deroin-Guillot proved a property of foliated $(\PGL(2,\CC),\CC)$-structures. Combining Deroin-Guillot's theorem with Brunella's classification we have a detailed list of possible foliations appearing in our study. Then we examine whether and how each foliation is realized, by using some properties of birational transformations preservring rational fibrations proved in our previous paper \cite{Zhaocent}. This analysis will occupy half of the paper.

\subsubsection{Plan} 
In Section \ref{nonnegkoddimchapter} we prove Theorem \ref{nonnegthm}. In the rest of the paper we only work in dimension two. In Section \ref{complexprojectivekleiniangroups} we present the known classification of complex projective Kleinian groups. Preliminaries on groups of birational transformations of projective surfaces are introduced in Section \ref{cremonagroups}. In Section \ref{delzantpychapter} we prove that birational Kleinian groups have fixed points on $\HH_Y\cup \partial \HH_Y$. Section \ref{foliationintro} is a glossary of holomorphic foliations on surfaces. In Sections \ref{rationalfibrationsectionone}, \ref{rationalfibrationsectiontwo} and \ref{rationalfibrationsectionthree} we study birational Kleinian groups preserving a rational fibration. 
At last Section \ref{synthesis} consists of our classification of birational Kleinian groups in dimension two. Results from previous sections are assembled. It contains the long version of Theorem \ref{themainthm} and serves as an outline of the materials in previous sections.

\section{Varieties of non-negative Kodaira dimension}\label{nonnegkoddimchapter}

\subsection{Invariant measure}
Let $Y$ be a smooth projec tive variety of dimension $n$.A measure which is absolutely continuous with respect to the Lebesgue measure associated with some Riemannian metric on $Y$ does not charge algebraic sets of codimension $\geq 1$. Therefore the push-forward of such a measure by a birational transformation of $Y$ is well-defined and it makes sense to talk about $\Gamma$-invariant measure. More precisely for $\nu$ such a measure, $\gamma\in\Bir(Y)$ a birational transformation and $E\subset Y$ a subset, $\gamma_*\nu(E)=\nu(\gamma^{-1}(E\backslash \mathrm{Ind}(\gamma^{-1})))$ where $\mathrm{Ind}$ means the set of indeterminacy points. The objective of this section is to prove:
\begin{lemma}\label{noinvariantmeasure}
Let $Y$ be a smooth projective variety of non-negative Kodaira dimension. Thee is a $\Bir(Y)$-invariant probability measure of support equal to $Y$ which is absolutely continuous with respect to the Lebesgue measure.
\end{lemma}

We show first that Theorem \ref{nonnegthm} follows from Lemma \ref{noinvariantmeasure}.
\begin{proof}[Proof of Theorem \ref{nonnegthm}]
Let $(Y,\Gamma,U,X)$ be a birational Kleinian group such that $Y$ has non-negative Kodaira dimension. By discontinuity of the action there exists a non-empty open subset $W$ of $Y$ such that $\gamma(W)\cap W=\emptyset$ for all $\gamma\neq e$. Let $\mu$ be a $\Gamma$-invariant probability measure of total support which is absolutely continous with respect to the Lebesgue measure as in Lemma \ref{noinvariantmeasure}. Then $\mu(U)\geq \sum_{\gamma\in \Gamma}\mu(\gamma(W))=+\infty$, contradiction.
\end{proof}

Corollary \ref{nowandering} below is another consequence of Lemma \ref{noinvariantmeasure}. In its statement $f$-invariant means that $f$ and $f^{-1}$ are regular on $V$ and $f(V)=V$; recurrent means that for any open subset $W\subset V$, for almost all $x\in W$ (with respect to the Lebesgue measure associated with a Riemannian metric), there are infinitely many $n\in\NN$ such that $f^n(x)\in W$. An open subset is wandering if its iterates by $f$ are disjoint from each other. 
\begin{corollary}\label{nowandering}
Let $Y$ be a smooth projective variety of non-negative Kodaira dimension. Let $f\in\Bir(Y)$ be a birational transformation of infinite order. Then for any non-empty $f$-invariant Borel open subset $V$ of $Y$, the action of $f$ on $V$ is recurrent. In particular there is no wandering open subset.
\end{corollary}
\begin{proof}
Let $V$ be a non-empty $f$-invariant Borel subset. Let $V'=V-\cup_n f^{-n}(Ind(f))$ be the complement of indeterminacy points of all iterates of $f$ in $V$. Let $\mu$ be a $\Gamma$-invariant probability measure of total support which is absolutely continous with respect to the Lebesgue measure. Then $V'$ is $f$-invariant and $\mu(V)=\mu(V')$. Poincaré recurrence theorem says that a measure preserving dynamical system on a space of finite measure is always recurrent.
\end{proof}


Let us explain the construction of the measure in Lemma \ref{noinvariantmeasure}. When $Y$ is an abelian variety or a Calabi-Yau variety, the canonical bundle is trivial and we can use the measure associated with the natural non-vanishing volume form defined up to a multiplicative scalar. In the general case we take a non-zero pluricanonical section $\alpha\in \mathrm{H}^0(Y,mK_Y)$ which in local coordinates can be written as $\alpha(z_1,\dots,z_n)(dz_1\wedge \cdots \wedge dz_n)^m$ where $n$ is the dimension of $Y$. Then the differential $2n$-form
\begin{equation}\label{eq:singularvolumeform}
i^{n^2}\big(\alpha\wedge \overline{\alpha}\big)^{\frac{1}{m}}=i^{n^2}\vert \alpha(z)\vert^{\frac{2}{m}}dz_1\wedge \cdots \wedge dz_n\wedge d\overline{z_1}\wedge \cdots \wedge d\overline{z_n}
\end{equation} 
is a singular volume form which is positive on the complement of the divisor $\{\alpha=0\}$. The finite measure defined by this singular volume form has full support and is absolutely continuous with respect to the Lebesgue measure associated with any Riemannian metric on $Y$. Hence the existence of a $\Bir(Y)$-invariant pluricanonical form implies Lemma \ref{noinvariantmeasure}. This is an immediate consequence of a theorem of Nakamura-Ueno (\cite{Ueno75} Chapter VI) saying that the linear representation of $\Bir(Y)$ on the space of pluricanonical forms has finite image. We give here an elementary argument without using the full strength of Nakamura-Ueno's theorem.

For a pluricanonical form $\alpha\in \mathrm{H}^0(Y,mK_Y)$ and a birational transformation $f\in\Bir(Y)$, we can pull back $\alpha$ by $f$ outside the indeterminacy locus of $f$ and then extend it in a unique way by Hartogs' Principle. In this way we obtain a linear representation 
\[
\rho:\Bir(Y)\rightarrow \GL\big(\mathrm{H}^0(Y,mK_Y)\big).
\] 
If we pull back the measure associated with a pluricanonical form, then we get the measure associated with the pulled-back pluricanonical form. Lemma \ref{noinvariantmeasure} is a consequence of the following statement:
\begin{lemma}
Let $Y$ be a smooth projective variety of non-negative Kodaira dimension. There is a non identically zero singular volume form $\Xi$ such that for any $\gamma\in\Bir(Y)$ we have $\gamma^*\Xi=\Xi$.
\end{lemma}

\begin{proof}
By the assumption on Kodaira dimension, there exists $m\in\NN^*$ such that $\mathrm{H}^0(Y,K_Y^m)$ is not trivial.
Consider the following function $\lVert\cdot \rVert$(similar to a norm) on the vector space $\mathrm{H}^0(Y,K_Y^m)$:
\[
\lVert\alpha\rVert^2=\int_Y i^{n^2}\big(\alpha\wedge \overline{\alpha}\big)^{\frac{1}{m}}.
\]
For $f\in\Bir(Y)$, a change of variables gives
\[
\lVert\alpha\rVert^2=\int_Y i^{n^2}\big(\alpha\wedge \overline{\alpha}\big)^{\frac{1}{m}}=\int_Y i^{n^2}\big(f^*\alpha\wedge \overline{f^*\alpha}\big)^{\frac{1}{m}}=\lVert f^*\alpha\rVert^2.
\]
This means that the image of the representation $\rho:\Bir(Y)\rightarrow \GL(\mathrm{H}^0(Y,K_Y^m))$ is contained in the compact group preserving the function $\lVert\cdot \rVert$. Denote by $G$ the Zariski closure of $\rho(\Bir(Y))$ in $\GL(\mathrm{H}^0(Y,K_Y^m))$. The group $G$ is compact and admits a left and right invariant Haar measure $\nu$. Let $\alpha$ be a non-zero element of $\mathrm{H}^0(Y,K_Y^m)$. Then
\[
\Xi=\int_G g^*\big(i^{n^2}(\alpha\wedge \overline{\alpha})^{\frac{1}{m}}\big)d\nu(g)
\]
is a desired singular volume form.
\end{proof}

\subsection{Genus one fibrations.}
If $Y$ is a surface with Kodaira dimension one then $Y$ has a genus one fibration $\pi:Y\rightarrow C$ induced by a pluricanonical linear system (see \cite{Iit70}). Every birational transformation of $Y$ preserves $\pi$. Note that the induced action of $\Bir(Y)$ on $C$ is finite, which is a particular case of Nakamura-Ueno Theorem (\cite{Ueno75} Chapter VI) that we used above. In this concrete case we present an elementary topological argument to illustrate Theorem \ref{nonnegthm}. We will also need this argument later while dealing with Halphen twists on rational surfaces.

\begin{lemma}\label{ellipticnondiscrete}
 Let $\Lambda$ be an infinite group of automorphisms of an elliptic curve $E$. Then the closure of the orbit of a point $x\in E$ is a finite union of real subtori.
\end{lemma}
\begin{proof}
The automorphism group of $E$ is the extension of a finite group by $E$. Some finite index subgroup $\Lambda'\subset \Lambda$ is thus a subgroup of the Lie group $E$. The closure $\overline{\Lambda'}$ in $E$ is a Lie group. It is thus a finite union of connected real subtori of positive dimension. For any $x\in E$ the closure of the orbit $\Lambda'\cdot x$ is also a union of real subtori of positive dimension. 
\end{proof}

\begin{lemma}\label{ellipticfibrationlem}
Let $\varphi:Y\rightarrow C$ be a genus one fibration on a surface $Y$. Let $U$ be an open set of $Y$ which is preserved by an infinite subgroup $\Gamma$ of $\Bir(Y)$. Suppose that $\Gamma$ preserves $\varphi$ and that the induced action on $C$ is finite. Then the action of $\Gamma$ on $U$ is not discontinuous. 
\end{lemma}
\begin{proof}
By hypothesis, a finite index subgroup $\Gamma_0$ of $\Gamma$ preserves each fiber of $\varphi$. Since this group is infinite, it contains a countable infinite subgroup $\Gamma_1$. We claim that there is a fiber $E=\varphi^{-1}(c)$ of genus $1$ which intersects $U$ and on which $\Gamma_1$ induces an infinite group of automorphisms. Indeed, $U$ is open, so $\varphi(U)$ is uncountable, and if $\gamma\in \Gamma_1$ is not the identity there are at most countably many points $z\in C$ for which $\gamma_{\vert \varphi^{-1}(z)}$ is the identity. Then, the result follows from Lemma \ref{ellipticnondiscrete}, applied to $E$ and $\Gamma_{1\vert E}$.
\end{proof}
Note that Lemma \ref{ellipticfibrationlem} implies Theorem \ref{nonnegthm} for surfaces of Kodaira dimension one.

\section{Complex projective Kleinian groups}\label{complexprojectivekleiniangroups}
\subsection{Classification}
This section is a presentation of the known classification of complex projective Kleinian groups in dimension two. We first give the list, then describe its items.
\begin{theorem}[Kobayashi-Ochiai \cite{KobOch80}, Mok-Yeung \cite{MY93}, Klingler \cite{Kli98}, Cano-Seade \cite{CanoSeade14}]\label{classificationofcplxprojkleinian}
If $(\PP^2,U,\Gamma,X)$ is a birational Kleinian group such that $\Gamma\subset \PGL(3,\CC)$, then up to taking finite index subgroup and up to geometric conjugation inside $\PGL(3,\CC)$, we are in one of the following situations:
\begin{enumerate}
	\item $U=\CC\times \CC, \CC\times\CC^*$ or $\CC^*\times\CC^*$ and $X$ is a complex torus; here $\CC\times \CC, \CC\times\CC^*, \CC^*\times\CC^*$ are the standard Zariski open subsets of $\PP^2$ and $\Gamma$ is a lattice in $U$ viewed as a Lie group.
	\item $U$ is the Euclidean ball $\mathbb{B}^2$ embedded in the standard way in $\PP^2$, i.e.\ in an affine chart $\CC^2\subset \PP^2$, and $X$ is a ball quotient.
	\item $U=\CC^2\backslash\{0\}$ where $\CC^2$ is embedded in the standard way in $\PP^2$ as a Zariski open set and $X$ is a Hopf surface.
	\item $U$ is the standard Zariski open subset $\CC^2\subset \PP^2$ and $X$ is a primary Kodaira surface.
	\item $U=\HH\times \CC$ is embedded in the standard way in an affine chart $\CC^2$, $X$ is an Inoue surface.
	\item $U$ is biholomorphic to $\HH\times \CC^*$ and $X$ is an affine-elliptic bundle.
\end{enumerate}
\end{theorem}

Remark that the cases of complex tori and ball quotients are the only cases where $X$ is \Kah{}. In higher dimension, if we want $X$ to be projective, then there is only one more possibility: $X$ can be the total space of a family of abelian varieties over a Shimura curve (Jahnke-Radloff \cite{JahRad15}).

The quotient surface is equipped with a $(\PGL(3,\CC),\PP^2)$-structure. The proof of Theorem \ref{classificationofcplxprojkleinian} has two main steps. First one finds the surfaces that admit $(\PGL(3,\CC),\PP^2)$-structures and the secondly one finds all $(\PGL(3,\CC),\PP^2)$-structures on those candidates. The first problem was studied by Gunning \cite{Gun78} and Kobayashi-Ochiai \cite{KobOch80}; it is completed by Klingler \cite{Kli98}. In \cite{Kli98}, the second problem is also investigated and is almost completed, except the case of ball quotients. The fact that the natural complex projective structure on a ball quotient is the only one is proved by Mok-Yeung in \cite{MY93}; later Klingler gave a different proof in \cite{Kli01}. The version of the classification that we state here can be found in the paper of Cano-Seade \cite{CanoSeade14} which contains more precise information, in particular when the group has torsion.

Though not every $(\PGL(3,\CC),\PP^2)$-structure arises from a complex projective Kleinian group, it was easier to study first the geometry of surfaces with complex projective structures than to look directly at the dynamics of subgroups of $\PGL(3,\CC)$. However this is not the strategy that we will use for general birational Kleinian groups: our method is based on a good understanding of dynamics of groups of birational transformations while general $(\Bir(Y),Y)$-structures seem to be more difficult to study.

\subsection{Examples}\label{examplesofprojstructures}
Here we describe all cases in Theorem \ref{classificationofcplxprojkleinian}. Some of them exist in any dimension.

\subsubsection{Complex tori}
A complex torus is a quotient of the commutative Lie group $\CC^2$ by a lattice, thus is naturally realized as the quotient of a complex affine Kleinian group. The affine varieties $\CC\times \CC, \CC\times\CC^*, \CC^*\times\CC^*$ are also infinite covers of complex tori and the corresponding groups of deck transformations are also birational.

\subsubsection{Ball quotients}
Let $\mathbb{B}^2$ be the unit ball $\{[z_0;\cdots;z_2]; \vert z_0\vert^2-\vert z_1 \vert^2-\vert z_2 \vert^2>0\}\subset \PP^2$. The group of biholomorphisms of $\mathbb{B}^2$ is $\PU(1,2)$. If $\Gamma$ is a torsion-free cocompact lattice of $\PU(1,2)$, then $\mathbb{B}^2/\Gamma$ is a projective manifold of general type equipped with a $(\PGL(3,\CC),\PP^2)$-structure.

\subsubsection{Hopf surfaces}
Let $\gamma\in\GL_2(\CC)$ be a contracting linear automorphism of $\CC^2$ such that $\gamma^n(x)\rightarrow 0$ for any $x\in\CC^2\backslash\{0\}$. Then the quotient of $\CC^2\backslash\{0\}$ by the action of $\gamma$ is a compact surface with a complex affine structure. Such surfaces are among the so called \emph{Hopf surfaces} and are not \Kah{}.

\subsubsection{Primary Kodaira surfaces}\label{primarykodaira}
Consider the subgroup $G$ of the group of affine transformations $\Aff(2,\CC)$ generated by the following four elements
\begin{align*}
z\mapsto z+\begin{pmatrix}0\\1\end{pmatrix},\quad & z\mapsto \begin{pmatrix}1&0\\c&1\end{pmatrix}z+\begin{pmatrix}1\\0\end{pmatrix}\\
z\mapsto z+\begin{pmatrix}0\\a\end{pmatrix},\quad & z\mapsto \begin{pmatrix}1&0\\d&1\end{pmatrix}z+\begin{pmatrix}b\\0\end{pmatrix}
\end{align*}
where $a,b\notin \RR$ and $c,d$ are not both equal to $0$.
The quotient $\CC^2/G$ is a primary Kodaira surface, i.e.\ a non \Kah{} principal bundle of elliptic curves over an elliptic curve; its Kodaira dimension is zero (see \cite{BHPV}). Note that the group $G$ is not unique up to conjugation for a fixed primary Kodaira surface, it has one parameter of freedom (see \cite{Kli98}).

\subsubsection{Inoue surfaces}\label{Inouesurfaces}
Inoue surfaces were discovered in \cite{Ino74} where they were constructed as quotients of complex affine Kleinian groups. There are three types of Inoue surfaces. Their universal covers are all $\HH\times \CC$ where $\HH$ is the Poincar\'e upper half plane. 

An \emph{Inoue surface of type $S^0$} is a quotient $\HH\times\CC/G$ where the group $G$ is a subgroup of $\Aff(2,\CC)$ generated by (we refer to \cite{Ino74} for the precise conditions that the parameters in the formulas have to satisfy)
\begin{align*}
g_0:&(x,y)\mapsto (\alpha x,\beta y)\\
g_i:&(x,y)\mapsto (x+a_i,y+b_i)\quad \text{for} \ i=1,2,3.
\end{align*}
An \emph{Inoue surface of type $S^+$} is a quotient $\HH\times\CC/G$ where $G$ is generated by \begin{align*}
g_0:(x,y)&\mapsto (\alpha x,y+t)\\
g_i:(x,y)&\mapsto (x+a_i,y+b_i x+c_i)\quad i=1,2\\
g_3:(x,y)&\mapsto (x,y+\frac{b_1a_2-b_2a_1}{n})
\end{align*}
 An Inoue surface of type $S^{-}$ has a double cover which is an Inoue surface of type $S^{+}$; in the above formulas it suffices to replace $g_0:(x,y)\mapsto (\alpha x,y+t)$ with $(x,y)\mapsto (\alpha x,-y)$.

\subsubsection{Affine-elliptic bundles}\label{affineellipticbundles}
This family of examples is constructed in \cite{Kli98}. Let $S=\HH/\pi_1(S)$ be a compact hyperbolic Riemann surface. Let $\bar{\rho}:\pi_1(S)\rightarrow \PGL(2,\CC)$ and $\phi:\HH\rightarrow \PP^1$ be the holonomy and developing map of some $(\PGL(2,\CC),\PP^1)$-structure on $S$. By \cite{Gun67}, $\bar{\rho}$ always lifts to a representation $\rho:\pi_1(S)\rightarrow \GL_2(\CC)$. Consider $\CC^2\backslash\{0\}$ as a $\CC^*$-bundle over $\PP^1$ and denote it by $W$. Consider the $\CC^*$-bundle $\phi^*W$ over $\HH$ obtained by pulling back $W$ via $\phi$. The natural complex affine structure on $W$ as an open set of $\CC^2$ induces a complex affine structure on $\phi^*W$. The group $\pi_1(S)$ acts on $\phi^*W$ via its natural action on $\HH$ and the representation $\rho$. This action preserves the affine structure on $\phi^*W$. Multiplication in the fibers also preserves this affine structure. Hence the quotient of $\phi^*W$ by the group generated by $\pi_1(S)$ and a multiplication in the fibers is a compact elliptic surface with a complex affine structure. We call them \emph{affine-elliptic bundles} as in \cite{Kli98}. If the $(\PGL(2,\CC),\PP^1)$-structure on $S$ is the quotient of the domain of discontinuity of a classical Kleinian group then this construction gives rise to a complex affine Kleinian group.

\section{Preliminaries on groups of birational transformations of surfaces}\label{cremonagroups}
This section is a glossary of notions and theorems on groups of birational transformations of surfaces. Most contents presented in this section can be found in the survey \cite{Can18}. We refer to \cite{Beauvillebook} and \cite{BHPV} for background on birational geometry of surfaces. 

\subsection{Two subgroups of the Cremona group}
The group $\Bir(\PP^2)$ is called the \emph{Cremona group}. In homogeneous coordinates a \emph{Cremona transformation} can be written as
\[
[x_0;x_1;x_2]\dashmapsto[P_0;P_1;P_2]
\]
where $P_1, P_2, P_3$ are homogeneous polynomials in $x_0,x_1,x_2$ and it has an inverse map of the same form. Here are two subgroups of $\Bir(\PP^2)$ that we will encounter several times in this text:

\subsubsection{The toric subgroup}
Fix a system of homogeneous coordinates on the projective plane $\PP^2$. The complement of three coordinates axes is isomorphic to $\CC^*\times\CC^*$. 
A matrix $\begin{pmatrix}	a&b\\c&d \end{pmatrix}\in \GL(2,\ZZ)$ acts by automorphism on $\CC^*\times\CC^*$ as follows: $(x,y)\mapsto (x^a y^b, x^c y^d)$. The group $\CC^*\times\CC^*$ acts on itself by multiplication. We obtain in this way an embedding of the semi-direct product $\CC^*\times\CC^*\rtimes \GL(2,\ZZ)$ into the Cremona group $\CRC$. We call its image \emph{the toric subgroup} of the Cremona group. The toric subgroup depends on a choice of homogeneous coordinates in $\PP^2$. Different choices of coordinates yield conjugate subgroups of the Cremona group and the conjugation is realized by an element of $\PGL(3,\CC)$. 
   
\subsubsection{The \Jonqui group.}

Fix an affine chart of $\PP^2$ with coordinates $(x,y)$.
\emph{The \Jonqui group} $\Jonq$ is the subgroup of the Cremona group of all transformations of the form
\[
(x,y)\mapsto \left ( \frac{ax+b}{cx+d},\frac{A(x)y+B(x)}{C(x)y+D(x)} \right ),\quad 
\begin{pmatrix}
	a&b\\c&d
\end{pmatrix}\in \PGL(2,\CC),\quad 
\begin{pmatrix}
	A&B\\C&D
\end{pmatrix}\in \operatorname{PGL}(2,\CC(x)).
\]
In other words, $\Jonq$ is the group of all birational transformations of $\PP^1\times\PP^1$ permuting the fibers of the projection onto the first factor; it is isomorphic to the semi-direct product $\PGL(2,\CC)\ltimes \operatorname{PGL}(2,\CC(x))$. The \Jonqui group is defined only if a system of coordinates is chosen; a different choice of the affine chart yields a conjugation by an element of $\PGL(3,\CC)$.

\subsection{Hyperbolic geometry and birational transformations}
Let $Y$ be a smooth complex projective surface. Manin \cite{Manin86} constructed an infinite dimensional hyperbolic space $\HH_Y$ and an injective homomorphism from $\Bir(Y)$ to the isometry group $\Iso(\HH_Y)$. For $f\in \Bir(Y)$ we denote by $f_*$ the corresponding element of $\Iso(\HH_Y)$. There is a dictionnary between the dynamics of $f_*$ on $\HH_Y$ and the dynamics of $f$ on $Y$. 

Let $f\in \Bir(Y)$. Let $\kappa\in \operatorname{H}^{1,1}(Y,\RR)$ be an ample class of self-intersection $1$. The \emph{degree} of $f$ with respect to $\kappa$, denoted by $\operatorname{deg}_{\kappa}(f)$, is the intersection number $f_*\kappa \cdot \kappa$. The \emph{translation length} of $f_*$ on $\HH_Y$ is $L(f_*)=\operatorname{inf}_{x\in\HH_Y}d(x,f_*(x))$. The sequence $\left(\operatorname{deg}_{\kappa}(f^n)\right )^{\frac{1}{n}}$ converges to a real number $\lambda(f)\geq 1$, called the dynamical degree of $f$ (see \cite{DF01}). Its logarithm $\log (\lambda(f))$ is the translation length $L(f_*)$ (see \cite{Can11}).

Elements of $\Bir(Y)$ are classified into four types (see \cite{DF01} and previous works \cite{Giz80}, \cite{Can01}, \cite{Can01b}):  
\begin{enumerate}
	\item The sequence $\operatorname{deg}_{\kappa}(f^n)$ is bounded, $f$ is birationally conjugate to an automorphism of a smooth projective surface birational to $Y$ and a positive iterate of $f$ lies in the connected component of identity of the automorphism group of that surface. The isometry $f_*$ is an elliptic isometry of $\HH_Y$. We call $f$ an \emph{elliptic} element.
	\item The sequence $\operatorname{deg}_{\kappa}(f^n)$ grows linearly, $f$ preserves a unique pencil of rational curves and $f$ is not conjugate to an automorphism of any projective surface. The isometry $f_*$ is a parabolic isometry of $\HH_Y$. We call $f$ a \emph{\Jonqui twist}.
	\item The sequence $\operatorname{deg}_{\kappa}(f^n)$ grows quadratically, $f$ is conjugate to an automorphism of a projective surface preserving a unique genus one fibration. The isometry $f_*$ is a parabolic isometry of $\HH_Y$. We call $f$ a \emph{Halphen twist}.
	\item The sequence $\operatorname{deg}_{\kappa}(f^n)$ grows exponentially, i.e.\ $\lambda(f)>1$. The isometry $f_*$ is a loxodromic isometry of $\HH_Y$. We call $f$ a \emph{loxodromic} element. 
\end{enumerate}

\subsection{Tits alternative}\label{sectiontits}
A group $G$ is said to satisfy \emph{the Tits alternative} if any subgroup of $G$ contains either a solvable subgroup of finite index, or a non-abelian free subgroup. Tits \cite{Tits72} proved that any linear algebraic group over a field of characteristic zero satisfies the Tits alternative. Lamy \cite{Lam01} proved that the group $\Aut(\CC^2)$ of polynomial automorphisms of the affine plane $\CC^2$ satisfies the Tits alternative. For $\Bir(\PP^2)$, the Tits alternative has been proved by Cantat \cite{Can11} for finitely generated subgroups and in general by Urech \cite{Ure21}. Furthermore for $\Bir(\PP^2)$ the Tits alternative has been improved to the following classification of infinite subgroups; one proves the Tits alternative by proving it for each case in the classification. The bulk of the classification is done by Cantat \cite{Can11} and it is finished by Urech \cite{Ure21}. Besides \cite{Can11} and \cite{Ure21}, the works of Weil \cite{Weil55}, Gizatullin \cite{Giz80}, Blanc-Cantat \cite{BC16} and D\'eserti \cite{Des15} are respectively the main input for case 1), 4), 5), 6) in the classification. 

\begin{theorem}[Strong Tits alternative for $\Bir(Y)$]\label{strongTitsalternative}
Let $Y$ be a smooth complex projective surface birational to $\bP^2$ and let $\Gamma$ be an infinite subgroup of $\Bir(Y)$. Then up to taking a finite index subgroup $\Gamma$ fits in one of the following mutually exclusive situations:
  \begin{enumerate}
  \item There is a projective surface $Y'$ and a birational map $\phi:Y\dashrightarrow Y'$ such that $\Gamma'=\phi\Gamma\phi^{-1}$ is contained in $\Aut^0(Y')$. 
  \item There is a projective surface $Y'$ and a birational map $\phi:Y\dashrightarrow Y'$ such that $\Gamma'=\phi\Gamma\phi^{-1}$ preserves a rational fibration on $Y'$. All elements of $\Gamma$ are elliptic but their degrees are not uniformly bounded and the group $\Gamma$ is not conjugate to a group of automorphisms of any projective surface.
	\item There is a projective surface $Y'$ and a birational map $\phi:Y\dashrightarrow Y'$ such that $\Gamma'=\phi\Gamma\phi^{-1}$ preserves a rational fibration on $Y'$. At least one element of $\Gamma$ is a \Jonqui twist.
  \item There is a projective surface $Y'$ and a birational map $\phi:Y\dashrightarrow Y'$ such that $\Gamma'=\phi\Gamma\phi^{-1}$ acts by automorphisms and preserves a genus one fibration. The group $\Gamma$ is virtually a free abelian group of rank $\leq 8$ formed by Halphen twists.
  \item $\Gamma$ is virtually a cyclic group generated by a loxodromic element. 
	\item $\Gamma$ is solvable but not virtually abelian. There is a birational map $\phi:Y\dashrightarrow \CC^*\times \CC^*$ such that $\Gamma'=\phi\Gamma\phi^{-1}$ is virtually contained in the toric subgroup $\Aut(\CC^*\times \CC^*)=(\CC^*\times\CC^*)\rtimes \GL_2(\ZZ)$. The loxodromic elements of $\Gamma$ form an infinite cyclic group. 
  \item $\Gamma$ contains a non-abelian free group all of whose non-trivial elements are loxodromic. 
	\item $\Gamma$ is a torsion group, is not finitely generated, and is not conjugate to a group of automorphisms of any projective surface.
 \end{enumerate}
\end{theorem}

A subgroup is called \emph{elementary} if it fixes a point in $\HH_Y\cup \partial \HH_Y$ or preserves a geodesic in $\HH_Y$; otherwise it is called \emph{non-elementary}. In Theorem \ref{strongTitsalternative} $\Gamma$ is non-elementary only in the seventh case.

In the first case of Theorem \ref{strongTitsalternative} $\Gamma$ is called an \emph{elliptic subgroup} of $\Bir(Y)$ because it fixes a point in $\HH_Y$; in the third or the fourth case, $\Gamma$ is called a \emph{parabolic subgroup} because it fixes a point in $\partial \HH_Y$.

\begin{remark}
It turns out that all birational Kleinian groups of interest appear either in Case 1) or Case 3) (see Theorem \ref{verylongthm}).
\end{remark}

\subsection{Centralizers of birational transformations preserving a rational fibration}
Here we collect a series of technical results that we will use in various concrete situations. They can all be found with proofs in \cite{Zhaocent}. Let $Y$ be a smooth rational surface and $r:Y\rightarrow \PP^1$ be a rational fibration. The subgroup of $\Bir(Y)$ that preserves the fibration $r$ will be identified with the \Jonqui group. For an element $f\in\Jonq$ (resp. a subgroup $\Gamma\subset \Jonq$) we denote by $f_B$ (resp. $\Gamma_B$) the element (resp. subgroup) of $\Aut(\PP^1)$ induced by the action on the base of the fibration. The centralizer in $\Bir(Y)$ of an element $f\in\Jonq$ is denoted by $\Cent(f)$. If $f$ is a \Jonqui twist then $\Cent(f)$ is necessarily a subgroup of $\Jonq$ and we denote by $\Cent_0(f)$ the normal subgroup of $\Cent(f)$ that preserves fiberwise the fibration $r$. If $f$ is a \Jonqui twist then we have an exact sequence 
\[
\{1\}\rightarrow \Cent_0(f)\rightarrow \Cent(f) \rightarrow \Cent_B(f)\rightarrow \{1\}.
\]

\begin{theorem}[Blanc-D\'eserti \cite{BD15}]\label{bdcentralizer}
Let $f\in\CRC$ be an elliptic element of infinite order. There exists an affine chart with affine coordinates $(x,y)$ on which $f$ acts by automorphism of the following form:
\begin{enumerate}
	\item $(x,y)\mapsto (\alpha x,\beta y)$ where $\alpha,\beta \in \CC^*$ are such that the kernel of the group homomorphism $\ZZ^2\rightarrow \CC^*,(i,j)\mapsto\alpha^i\beta^j$ is generated by $(k,0)$ for some $k\in\ZZ$;
	\item $(x,y)\mapsto (\alpha x,y+1)$ where $\alpha\in \CC^*$.
\end{enumerate}
The centralizer $\Cent(f)$ of $f$ in $\Bir(Y)$ is respectively described as follows:
\begin{enumerate}
	\item $\Cent(f)=\{(x,y)\dashmapsto(\eta(x),yR(x^k))\vert R\in \CC(x),\eta\in \PGL(2,\CC), \eta(\alpha x)=\alpha \eta(x)\}$.	If $\alpha$ is not a root of unity, i.e.\ if $k=0$, then $R$ must be constant.
	\item $\Cent(f)=\{(x,y)\dashmapsto(\eta(x),y+R(x))\vert \eta\in \PGL(2,\CC), \eta(\alpha x)=\alpha \eta(x), R\in \CC(x),\\ R(\alpha x)=R(x)\}$.	If $\alpha$ is not a root of unity, then $R$ is a constant and for some $\beta\in\CC^*$ we have $\eta(x)=\beta x$.
\end{enumerate}
\end{theorem}

\begin{theorem}[\cite{Zhaocent}]\label{zhaocentralizer1}
Let $f\in\Jonq$ be a \Jonqui twist such that $f_B$ has infinite order. Then $\Cent_B(f)$ is isomorphic to the product of a finite cyclic group with $\ZZ$. The infinite cyclic subgroup generated by $f_B$ has finite index in $\Cent_B(f)$.
\end{theorem}

We will also use the above two theorems in the following form:
\begin{theorem}[\cite{Zhaocent}]\label{zhaocentralizer}
Let $\Gamma\subset \Jonq$. Suppose that $\Gamma$ is isomorphic to $\ZZ^n$ with $n\geq 2$ and that $\Gamma_B$ is infinite. Then there is a subgroup $\Gamma'$ of finite index of $\Gamma$ with generators $\gamma_1,\cdots,\gamma_n$ such that one of the following assertions holds up to conjugation in $\Jonq$:
\begin{enumerate}
	\item For any $i$, $\gamma_i$ has the form $(x,y)\mapsto (a_ix,b_iy)$ with $a_i,b_i\in\CC^*$.
	\item For any $i$, $\gamma_i$ has the form $(x,y)\mapsto (x+a_i,b_iy)$ with $a_i\in\CC$ and $b_i\in\CC^*$.
	\item For any $i$, $\gamma_i$ has the form $(x,y)\mapsto (a_ix,y+b_i)$ with $a_i\in \CC^*$ and $b_i\in\CC$.
	\item For any $i$, $\gamma_i$ has the form $(x,y)\mapsto (x+a_i,y+b_i)$ with $a_i,b_i\in\CC$.
	\item $\gamma_1$ is a \Jonqui twist of the form $(x,y)\dashmapsto(\eta(x),yR(x))$ where $\eta\in\PGL(2,\CC)$ and $R\in\CC(x)^*$. For any $i>1$, $\gamma_i$ has the form $(x,y)\mapsto (x,b_iy)$ with $b_i\in\CC^*$.
	\item $\gamma_1$ is a \Jonqui twist of the form $(x,y)\dashmapsto(\eta(x),y+R(x))$ where $\eta\in\PGL(2,\CC)$ and $R\in\CC(x)$. For any $i>1$, $\gamma_i$ has the form $(x,y)\mapsto (x,y+b_i)$ with $b_i\in\CC$.
\end{enumerate}
\end{theorem}

\section{Birational Kleinian groups are elementary}\label{delzantpychapter}
The goal of this section is to show
\begin{theorem}\label{kleinianareelementary}
Let $(Y,U,\Gamma,X)$ be a birational Kleinian group in dimension two. If $X$ is not a surface of class VII then $\Gamma$ is an elementary subgroup of $\Bir(Y)$. 
\end{theorem}
Recall that a non-elementary subgroup of $\Bir(Y)$ is a subgroup that fixes a point or preserves a geodesic in $\HH_Y$ (cf. Section \ref{sectiontits}). A surface of class VII is a compact complex surface with Kodaira dimension $-\infty$ and first Betti number $1$; it is non-\Kah{}.

\subsection{Isometric actions of \Kah{} groups on hyperbolic spaces}
 In the next two subsections we will deduce Theorem \ref{kleinianareelementary} from the following theorem of Delzant-Py (the terminology \emph{factors through a fibration} is explained below):
\begin{theorem}[Delzant-Py \cite{DelPy}]\label{DPthm}
Let $Y$ be a rational surface and $X$ be a compact \Kah{} manifold. Let $\rho:\pi_1(X)\rightarrow \Bir(Y)$ be a non-elementary representation, then one of the following two cases occurs:
\begin{enumerate}
\item There is a birational map $\varphi:Y\dashrightarrow \CC^{*}\times \CC^*$ such that $\varphi\rho\varphi^{-1}$ is a representation into the toric group $\CC^*\times\CC^*\rtimes \GL_2(\ZZ)$.
\item There is a finite index subroup $\Gamma'$ of $\Gamma$ corresponding to a finite \'etale cover $X'$ of $X$ such that
the induced representation $\pi_1(X')\rightarrow\Gamma'\rightarrow \Bir(\PP^2)$ factors through a fibration $X'\rightarrow \Sigma$ onto a hyperbolic orbicurve $\Sigma$.
\end{enumerate}
\end{theorem}

In this article a \textit{hyperbolic orbicurve} $\Sigma$ will be thought as a Riemann surface with finitely many marked points with multiplicities, obtained as a quotient of the Poincar\'e half-plane $\HH$ by the action of a cocompact lattice in $\PSL(2,\RR)$; marked points are images of the fixed points of the action and multiplicities are orders of the stabilizers (see \cite{Thurston} for orbifolds). The cocompact lattice is isomorphic to the orbifold fundamental group $\pisigma$ of $\Sigma$. A continuous map from a complex manifold $X$ to $\Sigma$ is holomorphic if it lifts to a holomorphic map from the universal cover $\tilde{X}$ of $X$ to the half-plane, i.e.\ if there exists a holomorphic map from $\tilde{X}$ to $\HH$ such that the following diagram is commutative:
\begin{equation*}
\begin{tikzcd}
\tilde{X} \arrow{r}{} \arrow{d}{} & \HH  \arrow{d}{}\\
X \arrow{r}{} & \Sigma 
\end{tikzcd}
\end{equation*}
A \emph{fibration} of $X$ onto $\Sigma$ is a holomorphic surjective map $f:X\rightarrow \Sigma$ with connected fibers. A fibration $f:X\rightarrow \Sigma$ induces a homomorphism $f_*:\pi_1(X)\rightarrow \pisigma$. We say that a group representation $\rho: \pi_1(X) \rightarrow G$ \emph{factors through the fibration $f$} if there exists a homomorphism $\hat{\rho}: \pisigma\rightarrow G$ such that $\rho=\hat{\rho}\circ f_*$.

Non-elementary isometric actions of \Kah{} groups on hyperbolic spaces are known to factor through fibrations since the work of Carlson-Toledo \cite{CT}. Delzant-Py's Theorem is obtained by generalizing Carlson-Toledo's work to infinite dimension.

\subsection{Conjugation}\label{prep1}
The goal of this subsection is to prove a lemma which will be used in the proof of Theorem \ref{kleinianareelementary} and also in the sequel of our paper. Roughly speaking the lemma says that a conjugation of birational Kleinian groups is always a geometric conjugation. Let $(Y,\Gamma,U,X)$ be a birational kleinian group.
\begin{remark}\label{indetdisjfromU}
Since $\Gamma$ acts regularly on $U$, the indeterminacy points and contracted curves of elements of $\Gamma$ are disjoint from $U$.
\end{remark}

Recall that the definition of "geometrically conjugate" is given in Section \ref{definitionssection}.

\begin{lemma}\label{birconjtoaut}
Suppose that there exists a birational map $\phi: Y\dashrightarrow Y'$ to a second surface $Y'$ and a non-empty Zariski open set $Z'\subset Y'$ such that $\Gamma'=\phi\circ\Gamma\circ \phi^{-1}$ is contained in $\Aut(Z')$. Then there exist a third surface $Y''$, a non-empty Zariski open set $Z''\subset Y''$ and a birational kleinian group $(Y'',\Gamma'',U'',X'')$ geometrically conjugate to $(Y,\Gamma,U,X)$ such that $\Gamma'' \in \Aut(Z'')$. 

The Zariski open set $Z''$ can be obtained by blowing up $Z'$. If $Y'=Z'$, then $Y''=Z''$. We have $\Gamma''\subset\Aut^0(Y'')$ if $Y'=Z'$ and $\Gamma'\subset\Aut^0(Y')$.
\end{lemma}
We will use the following terminology in the proof of Lemma \ref{birconjtoaut}: if $f$ is a birational map and if $x\in \operatorname{Ind}(f)$ is an indeterminacy point then the \emph{image} $f(q)$ is $\bigcap_W f(W\backslash \{q\})$ where $W$ runs throuh all analytic neiborhoods of $q$, i.e.\ $f(q)$ is the connected curve contracted by $f^{-1}$ onto $q$.
\begin{proof}
We first claim that the image of $U\cap \operatorname{Ind}(\phi)$ is disjoint from $Z'$. Suppose by way of contradiction that $u\in U$ is an indeterminacy point of $\phi$ and $\phi(u)\cap Z'\neq \emptyset$. Pick a non trivial element $\gamma\in \Gamma$. Denote $\phi\circ \gamma \circ \phi^{-1}$ by $\gamma'$. As $\gamma$ acts freely by diffeomorphism on $U$, $\gamma(u)$ is a point different from $u$. And as $\gamma'$ is an automorphism of $Z'$, the strict transform $\gamma'(f(u))$ is still a curve which intersects $Z'$. Thus $\gamma'(f(u))$ is contracted by $\phi^{-1}$ onto $\gamma(u)$. This implies that $\{\gamma(u)\vert \gamma\in \Gamma\}$ is an infinite set of indeterminacy points of $\phi$, which is impossible. 

Replacing $Y$ with the blow up at the indeterminacy points of $\phi$ outside $U$, we can assume that $\operatorname{Ind}(\phi)\subset U$. By the previous paragraph the image of any indeterminacy point of $\phi$ is disjoint from $Z'$. Again up to replacing $Y$ with a blow up, we can also assume that no (-1)-curves contracted by $\phi$ onto a point of $Z'$ lie outside $U$. 

Consider now a point $q\in Z'\cap\operatorname{Ind}(\phi^{-1})$. We assert that $\phi^{-1}(q)\cap U\neq \emptyset$. The proof goes as follows. If $\phi^{-1}(q)$ contains an indeterminacy point of $\phi$ then this indeterminacy point is in $U$ by the assumption that we made in the previous paragraph. Assume that $\phi^{-1}(q)\cap \operatorname{Ind}(\phi)=\emptyset$. Then $\phi^{-1}(q)$ is a connected chain of rational curves containing at least one component which is a $(-1)$-curve. This $(-1)$-curve intersects $U$ because of the assumption in the previous paragraph.

Pick $u'\in Z'$ an indeterminacy point of $\phi^{-1}$. Let $C$ be the connected set of curves on $Y$ contracted onto $u'$. Then $C\cap U\neq \emptyset$ and for $\gamma \in \Gamma$ the strict transform $\gamma(C)=\overline{\gamma(C\backslash \{\text{ind points of}\ \gamma\})}$ is still a set of curves intersecting $U$. As $\gamma'$ acts by automorphism on $Z'$, the image $\gamma'(u')$ is still a point in $Z'$ and $\gamma(C)$ is contracted onto it. Therefore $\Gamma'$ permutes the indeterminacy points of $\phi^{-1}$ in $Z'$.  

We blow up each point in $\operatorname{Ind}(\phi^{-1})\cap Z'$ once to obtain a surface $Y_1$ and denote by $\phi_1$ the contraction morphism $Y_1\xlongrightarrow{\phi_1}Y'$. Let $Z_1$ be the preimage $\phi_1^{-1}(Z')\subset Y_1$. Since $\Gamma'$ permutes the indeterminacy points of $\phi^{-1}$ in $Z'$, the group $\Gamma_1=\phi_1^{-1}\circ\Gamma'\circ\phi_1$ acts by automorphisms on $Z_1$. Note that $\Gamma_1\subset\Aut^0(Y_1)$ if $Y'=Z'$ and $\Gamma'\subset \Aut^0(Y')$. Then we continue the process for $Y_1,Z_1$ and $\Gamma_1$. After finitely many times, say $m$ times, we get $Y_m, Z_m, \Gamma_m$ such that the birational map $\varphi:Y_m\dashrightarrow Y$ induced by $\phi^{-1}$ has no indeterminacy points in $Z_m$. Then our assumption in the second paragraph implies that $\varphi$ restricted to $Z_m$ is an isomorphism onto image. Hence $\Gamma$ acts by automorphisms on the Zariski open set $\varphi(Z_m)\subset Y$. The proof is finished.
\end{proof}

\subsection{From Delzant-Py's theorem to birational Kleinian groups}\label{fromDPtoKleinian}
This subsection is a proof of Theorem \ref{kleinianareelementary}. Throughout the section $(Y,\Gamma,U,X)$ will be a birational kleinian group in dimension two. We want to show that $\Gamma$ is elementary. By Theorem \ref{nonnegthm} the Kodaira dimension of $Y$ is necessarily $-\infty$. If $Y$ is not rational, then it is a ruled surface (see \cite{Beauvillebook}) and $\Bir(Y)$ preserves the unique ruling. This means that $\Bir(Y)$ does not have any non-elementary subgroup. Thus we can and will assume that $Y$ is a \emph{rational surface}. 

The open set $U$ is an infinite intermediate Galois covering of $X$ and $\Gamma$, being the deck transformation group of the covering, is a quotient of $\pi_1(X)$. Thus we have a representation $\rho:\pi_1(X)\rightarrow \Bir(Y)$, composition of the quotient homomorphism $\pi_1(X)\twoheadrightarrow \Gamma$ and the inclusion $\Gamma \hookrightarrow \Bir(Y)$. 

The proof of Theorem \ref{kleinianareelementary} is divided into three parts. In the first part we will show that $\Gamma$ is not conjugate to a non-elementary subgroup of the toric group. In the second part we show that the representation $\rho$ does not factor through a curve. In the third part we investigate the situation where the quotient $X$ is neither \Kah{} nor class VII.

If $X$ is \Kah{} then we can apply Theorem \ref{DPthm} to the representation $\rho$. If $\rho$ is non-elementary then there are two possible cases : either it factors through an orbicurve or it is conjugate to a subgroup of the toric subgroup. Thus for \Kah{} $X$ the first and the second part of this subsection will be sufficient for the proof of Theorem \ref{kleinianareelementary}.

\subsubsection{Toric subgroup} 
The function field $\CC(Y)$ is isomorphic to $\CC(x,y)$ as $Y$ is rational. A valuation on $\CC(Y)$ is a $\ZZ$-valued function $v$ on $\CC(Y)^*$ such that 
\begin{enumerate}
	\item $v(a)=0$ for any $a\in\CC^*$;
	\item $v(PQ)=v(P)+v(Q)$ and $v(P+Q)\geq \min (v(P),v(Q))$ for any $P,Q\in \CC(Y)^*$;
	\item $v(\CC(Y)^*)=\ZZ$.
\end{enumerate}
The elements of $\CC(Y)^*$ with valuations $\geq 0$ together with $0$ form the valuation ring $A_v$; those with valuations $>0$ form a maximal ideal $M_v$. If the residue field $A_v/M_v$ has transcendence degree $1$ over $\CC$, then $v$ is called a \emph{divisorial valuation}. The set of divisorial valuations is in bijection with the set of irreducible hypersurfaces in all birational models of $Y$ (cf. \cite{SZ75}). If $v_E$ is the divisorial valuation associated with an irreducible hypersurface $E$ in some birational model of $Y$, then for any $P\in\CC(Y)^*$ the value $v_E(P)$ is the vanishing order of $P$ along $E$. We can identify $\Bir(Y)$ with the group of automorphisms of $\CC(Y)$. Thus $\Bir(Y)$ acts on the set of valuations by precomposition. It preserves the subset of divisorial valuations.

\begin{lemma}\label{toricloxodromic}
Suppose that $Y$ is a smooth compactification of $\CC^*\times\CC^*$. Let $\gamma$ be a loxodromic element in the toric subgroup. Then every irreducible component of $Y\backslash (\CC^*\times\CC^*)$ is contracted by a power of $\gamma$.
\end{lemma}
\begin{proof}
The complement $Y\backslash (\CC^*\times\CC^*)$ has only finitely many irreducible components. Let $v_E$ be a divisorial valuation with $E$ outside of $\CC^*\times\CC^*$. It is sufficient to prove that $v_E$ has an infinite orbit under $\gamma$. Let $(x,y)$ be the standard coordinates on $\CC^*\times\CC^*$. As $E$ is outside $\CC^*\times\CC^*$, at least one of $x,y$ has non-zero valuation with respect to $v_E$. 

The tranformation $\gamma$ can be written as $(x,y)\mapsto (\alpha x^ay^b,\beta x^cy^d)$ with $\alpha,\beta \in \CC^*$ and $\begin{pmatrix}a&b\\c&d\end{pmatrix}\in \GL_2(\ZZ)$. The matrix $\begin{pmatrix}a&b\\c&d\end{pmatrix}$ is hyperbolic because $\gamma$ is loxodromic. We have for any $n\in\ZZ$
\[
\begin{pmatrix}(\gamma^n\cdot v_E)(x)\\(\gamma^n\cdot v_E)(y)\end{pmatrix}=\begin{pmatrix}a&b\\c&d\end{pmatrix}^n\begin{pmatrix}v_E(x)\\v_E(y)\end{pmatrix}.
\]
Since $\begin{pmatrix}v_E(x)\\v_E(y)\end{pmatrix}$ is a non-zero integer vector and $\begin{pmatrix}a&b\\c&d\end{pmatrix}$ is a hyperbolic matrix, the orbit of $\begin{pmatrix}v_E(x)\\v_E(y)\end{pmatrix}$ under $\begin{pmatrix}a&b\\c&d\end{pmatrix}$ is infinite. This implies that the orbit of $v_E$ under $\gamma$ is infinite.
\end{proof}

\begin{lemma}\label{toricconjugation}
Suppose that there is a birational map $\phi:Y\dashrightarrow \CC^*\times\CC^*$ such that $\phi\Gamma\phi^{-1}$ is contained in the toric subgroup. Suppose that $\Gamma$ is not virtually a cyclic group. Then $(Y,\Gamma,U,X)$ is geometrically conjugate to a birational Kleinian group $(Y',\Gamma',U',X')$ such that $Y'$ is a compactification of $\CC^*\times\CC^*$ and $\Gamma'\subset \Aut(\CC^*\times\CC^*)$.
\end{lemma}
\begin{proof}
By Lemma \ref{birconjtoaut} we can and will assume that there exists a blow up $Z$ of $\CC^*\times \CC^*$ such that $Y$ is a compactification of $Z$ and $\Gamma\subset \Aut(Z)$. Consider the blow down map $g:Z\rightarrow \CC^*\times \CC^*$ which is equivariant under the action of $\Gamma$. 

Suppose that $C$ is a connected curve on $Z$ such that $g(C)$ is a point. Up to replacing $\Gamma$ by a finite index subgroup, we can and will assume that each irreducible component of a curve contracted by $g$ is invariant under $\Gamma$. Thus $g(C)$ is a fixed point of $\Gamma$ in $\CC^*\times \CC^*$. Without loss of generality we can assume that $g(C)$ is the point $(1,1)$. The stabilizer of $(1,1)$ in the toric group is $\GL(2,\ZZ)$. Thus $\Gamma\subset\GL(2,\ZZ)$. Let $Z_1$ be the blow up of $\CC^*\times \CC^*$ at $(1,1)$ and let $E_1\subset Z_1$ be the exceptional curve. The induced action of $\GL(2,\ZZ)$ on $E_1$ is conjugate to the usual action of $\GL(2,\ZZ)$ on $\PP^1$ by Möbius maps. Note that $g$ factors through $Z_1$. If $C\neq E_1$ then $\Gamma$ must have a fixed point in $E_1$ to be blown up further. However any subgroup of $\GL(2,\ZZ)$ fixing a point in $\PP^1$ is cyclic. This would contract our hypothesis. Therefore we deduce that $C=E_1$.

If $C=E_1$ does not intersect $U$, then the contraction of $C$ would be a geometric conjugation of our birational Kleinian group and the proof is finished. Suppose by way of contradiction that $U_1=U\cap C$ is not empty. Then $U_1/\Gamma$ is a smooth compact curve on $X$. In other words $U_1$ is an invariant union of components of the discontinuity domain of a classical Kleinian group isomorphic to $\Gamma$. However this classical Kleinian group is contained in the Fuchsian group $\GL(2,\ZZ)$ which has a free subgroup of finite index and its domain of discontinuity contains the upper and lower half planes. If it is a non-elementary Fuchsian subgroup of $\GL(2,\ZZ)$ then either it contains a parabolic element or it has a purely loxodromic Schottky Fuchsian subgroup of finite index. In both cases the quotient $U_1/\Gamma$ cannot be compact. Then it has to be an elementary subgroup of $\GL(2,\ZZ)$ which is necessarily cyclic, contradiction to our hypothesis.

\end{proof}

\begin{proposition}\label{notorickleinian}
Suppose that there is a birational map $\phi:Y\dashrightarrow \CC^*\times\CC^*$ such that $\phi\Gamma\phi^{-1}$ is contained in the toric subgroup. Suppose that any subgroup of $\Gamma$ generated by a loxodromic element has infinite index. Then $\Gamma$ contains no loxodromic element.
\end{proposition}
\begin{proof}
By Lemma \ref{toricconjugation} we can suppose that $Y$ is a compactification of $\CC^*\times \CC^*$ and $\Gamma\subset \CC^*\times\CC^*\rtimes \GL_2(\ZZ)$. Suppose by way of contradiction that $\Gamma$ contains a loxodromic element $\eta$. Then by Lemma \ref{toricloxodromic} any irreducible component of $Y\backslash(\CC^*\times \CC^*)$ is contracted by a power of $\eta$. This implies that the open set $U$ on which $\eta$ acts by biholomorphism is a subset of $\CC^*\times\CC^*$.

Consider the exponential map $\pi$ from $\CC^2$ to $\CC^*\times \CC^*$; it is a covering map. The exponential also gives rise to a homomorphism $\rho:\Aff(2,\CC)=\CC^2\rtimes \GL(2,\CC)\rightarrow \ Aut(\CC^*\times \CC^*)=\CC^*\times\CC^*\rtimes \GL(2,\CC)$. Then $\pi$ is $\rho$-equivariant. Consider a connected component $\bar{U}$ of $\pi^{-1}(U)$ and the subgroup $\bar{\Gamma}$ of $\rho^{-1}(\Gamma)$ that preserves $\bar{U}$. Then $\bar{\Gamma}$ is a complex affine Kleinian group. We can apply Theorem \ref{classificationofcplxprojkleinian} to $\bar{\Gamma}$. There are five possibilities for $X$: it is a torus, a Hopf surface, an Inoue surface, a primary Kodaira surface or an affine-elliptic bundle. From the descriptions that we give in Section \ref{examplesofprojstructures}, we see that, for the linear part of an element of $\bar{\Gamma}$ to be a hyperbolic matrix, the only possibilities are Inoue surfaces and affine-elliptic bundles. But in these two cases the group is not contained in $\GL(2,\ZZ)$. For Inoue surfaces of type $S^0$, the parameter $\beta$ (see Section \ref{Inouesurfaces}) is not a real number while for Inoue surfaces of type $S^{\pm}$ the parameter $\alpha$ is a degree two algebraic integer (see \cite{Ino74}). For affine-elliptic bundles it is because the Fuchsian group $\GL(2,\ZZ)$ has no subgroup with compact hyperbolic Riemann surface as quotient (see the argument at the end of the proof of Lemma \ref{toricconjugation}). Hence the assumption that $\Gamma$ contains a loxodromic element is absurd.
\end{proof}

\subsubsection{Factorization through curves}

\begin{lemma}\label{liftprop}
If $\Gamma$ contains a loxodromic element of $\Bir(Y)$, then for any compact curve $C$ on $X$ the image of the composition $\pi_1(C)\rightarrow \pi_1(X) \twoheadrightarrow \Gamma$ is infinite.
\end{lemma}
\begin{proof}
By Theorem \ref{nonnegthm} we can assume that $Y$ is a birationally ruled surface. All birational transformations of a non rational ruled surface preserve the ruling, thus are not loxodromic. We can and will further assume that $Y$ is rational.

Suppose by way of contradiction that $C$ is a compact curve on $X$ such that $\pi_1(C)\rightarrow \Gamma$ has finite image. Denote by $D$ the normalization of $C$. Then there is a finite unramified cover $\bar{D}\rightarrow D$ such that the composition $\pi_1(\bar{D})\rightarrow\pi_1(D)\rightarrow \pi_1(X) \twoheadrightarrow \Gamma$ is trivial. This implies the existence of a map $\imath: \bar{D}\rightarrow U$ lifting the composition map $\bar{D}\rightarrow D\rightarrow C\rightarrow X$. Only a finite subgroup of $\Gamma$ preserves $\imath(\bar{D})$ because an infinite group can never act in a discrete way on the compact space $\imath(\bar{D})$. Therefore $\{\gamma(\imath(\bar{D}))\}_{\gamma\in\Gamma}$ form an infinite family of disjoint smooth compact curves in $U\subset Y$. Every element of $\Gamma$ permutes these curves. Denote by $\alpha_j, j\in\NN^*$ the classes of these curves in the N\'eron-Severi group of $Y$. Note that the Picard group is isomorphic to the N\'eron-Severi group because $Y$ is rational. The intersection number $\alpha_i\cdot \alpha_j$ is zero for $i\neq j$ since the corresponding curves are disjoint. As the N\'eron-Severi group has finite rank, we can suppose that for some $r\in\NN^*$, the classes $\alpha_1,\cdots,\alpha_r$ are linearly independent and for any $n>r$ the class $\alpha_n$ is equal to a linear combination of the $\alpha_j, j\leq r$. Among the $\alpha_j, j\leq r$, at most one has zero self-intersection because otherwise there would exist a two-dimensional totally isotropic subspace of the N\'eron-Severi group which contradicts the Hodge index theorem. When we write $\alpha_n$ as a linear combination of the the $\alpha_j, j\leq r$, if the coefficient before $\alpha_j$ is non-zero then $\alpha_j$ has zero self-intersection because $\alpha_n\cdot \alpha_j=0$. This implies that all but finitely many of the $\alpha_j$ are equal to a class $\beta$ of zero self-intersection. Thus the linear system associated with $\beta$ has dimension $\geq 1$. Since the group $\Gamma$ permutes the curves $\gamma(\imath(\bar{D}))$, the class $\beta$ is $\Gamma$-invariant. Hence we obtain a pencil of curves invariant under $\Gamma$. However a loxodromic birational transformation preserves no pencils of curves (see \cite{Can11}).
\end{proof}


\begin{proposition}\label{nonelemneveroccurs}
Suppose that $\Gamma$ is a non-elementary subgroup. Then the representation $\rho:\pi_1(X)\twoheadrightarrow \Gamma \hookrightarrow \Bir(Y)$ does not factor through a hyperbolic orbicurve. 
\end{proposition}
\begin{proof}
Suppose by way of contradiction that $\rho$ is non-elementary and factors through a hyperbolic orbicurve $\Sigma$. Let $F$ be a general fibre of the fibration $X\rightarrow \Sigma$. The image of $\phi: \pi_1(F)\rightarrow \pi_1(X)$ is in the kernel of $\pi_1(X)\rightarrow \pisigma$, thus in the kernel of $\rho$. This contradicts Lemma \ref{liftprop}. 
\end{proof}

\subsubsection{Elliptic surfaces}
According to Kodaira's classification of surfaces, a non-\Kah{} surface which is not of class VII fits in one of the two following possibilities: either $X$ is a primary or secondary Kodaira surface, or $X$ is an elliptic surface with Kodaira dimension $1$. The fundamental group of a Kodaira surface is solvable and has no non-elementary representations into $\Bir(Y)$. To finish the proof of Theorem \ref{kleinianareelementary}, it remains to consider the case where $X$ is an elliptic surface of Kodaira dimension $1$. 

\begin{proposition}
If $X$ is an elliptic surface of Kodaira dimension $1$, then $\Gamma$ is an elementary subgroup of $\Bir(Y)$.
\end{proposition}
\begin{proof}
Let $X\rightarrow \Sigma$ be a genus one fibration and assume that $X$ has Kodaira dimension one.
We have an exact sequence 
\[1\rightarrow H \rightarrow \pi_1(X)\xrightarrow{\varphi} \pisigma \rightarrow 1\]
where $H$ is the image of the fundamental group of a regular fiber in $\pi_1(X)$ (see \cite{GurSha85} Theorem 1). 

Suppose by way of contradiction that the representation $\rho:\pi_1(X)\rightarrow \Gamma\hookrightarrow \Bir(Y)$ is non-elementary. By Proposition \ref{liftprop} $\rho(H)$ is not finite. As a regular fiber is an elliptic curve, its fundamental group is isomorphic to $\ZZ^2$. Therefore $\rho(H)$ is an infinite abelian group. Since $\Gamma$ is non-elementary, there exists at least one element $a\in\pi_1(X)$ such that $\varphi(a)$ has infinite order and $\rho(a)$ is a loxodromic element. Consider $<\varphi(a)>$ the infinite cyclic subgroup of $\pisigma$ generated by $\varphi(a)$. Denote by $G$ the subgroup $\varphi^{-1}(<\varphi(a)>)$ of $\pi_1(X)$. Then $G$ is an extension of $<\varphi(a)>$ by $H$; in particular it is solvable. The group $\rho(G)$ is solvable but not virtually cyclic; it contains a loxodromic element $\rho(a)$. By Theorem \ref{strongTitsalternative} we infer that a finite index subgroup of a conjugate of $\rho(G)$ is contained in the toric subgroup and that $\rho(H)$ is an infinite elliptic subgroup. The whole group $\Gamma$ normalizes $\rho(H)$. By a theorem of Cantat on normalizers (see Appendix of \cite{DelPy}) we infer that $\Gamma$ is up to conjugation contained in the toric subgroup. This contradicts Proposition \ref{notorickleinian}.
\end{proof}

\subsection{Surfaces of class VII}
It would be nice to drop the hypothesis that $X$ is not a class VII surface in Theorem \ref{kleinianareelementary}. Note that no known surface of class VII has non-solvable fundamental group. Thus an analogue of Theorem \ref{kleinianareelementary} for class VII surfaces could be a vacuous statement. On the other hand in \cite{CT97} Carleson-Toledo applied the same method of \cite{CT} to study isometric actions of fundamental groups of class VII surfaces on finite dimensional hyperbolic spaces, by replacing harmonic mappings with hermitian harmonic mappings. So one could hope for an analogue of Delzant-Py's theorem for class VII surfaces.

\section{Foliated surfaces}\label{foliationintro}
 
This section is a glossary of holomorphic foliations on surfaces. We present Brunella's classification in \ref{Brunellathmsection}. In \ref{foliatedcomplexprojectivestructuressection} we state Deroin-Guillot's theorem on foliated $(\PGL(2,\CC),\PP^1)$-structures. In this paper we only consider holomorphic foliations without singularities, i.e.\ regular holomorphic foliations.

\subsection{Brunella's Theorem}\label{Brunellathmsection}
The vocabulary in the following theorem will be explained below.
\begin{theorem}[Brunella \cite{Bru97}]\label{Brunellathm}
Let $X$ be a compact complex surface and $\Fol$ a regular holomorphic foliation on $X$. Then one of the following situations holds:
\begin{enumerate}
	\item $\Fol$ comes from a fibration of $X$ onto a curve;
	\item $X$ is a complex torus and $\Fol$ is an irrational linear foliation;
	\item $\Fol$ is an obvious foliation on a non-elliptic Hopf surface;
	\item $\Fol$ is an obvious foliation on an Inoue surface;
	\item $\Fol$ is an infinite suspension of $\PP^1$ or of an elliptic curve over a compact Riemann surface;
	\item $\Fol$ is a turbulent foliation with at least one invariant fiber;
	\item $\Fol$ is a transversely hyperbolic foliation with dense leaves whose universal cover is a fibration of disks over a disk.
\end{enumerate}
\end{theorem}

The possible overlaps in the above list are: case 2 with case 5, case 3 with case 6, case 5 with case 6. When a complex torus admits an elliptic fibration, it is straightforward to see that an irrational linear foliation is also a suspension of an elliptic curve. For overlaps between case 3 and case 6, or between case 5 and case 6, see Section \ref{examplesturbulent}.

\subsubsection{Fibrations.}
Let $X$ be compact complex surface and let $f:X\rightarrow B$ be a fibration whose singular fibers are all multiples of smooth curves. The fibration equips $X$ with a regular foliation whose leaves are the underlying manifolds of the fibers. Let $mF_b$ be a fiber with multiplicity $m$ lying over a point $b\in B$. Let $w$ be a local coordinate on $B$ which vanishes at $b$, let $h$ be a local equation of $F_b$ in $X$, then $f^*(dw)/h^{m-1}$ is a local differential form defining the foliation.

\subsubsection{Linear foliations on tori}
Let $\Lambda$ be a lattice in $\CC^2$. The quotient $X=\CC^2/\Lambda$ is a two dimensional complex torus. Let $(z,w)$ be the natural coordinates on $\CC^2$. A constant holomorphic differential $adz+bdw$ on $\CC^2$ with $a,b\in\CC$ and $ab\neq 0$ descends onto $X$ and defines a regular foliation $\Fol$ on $X$; this is called a \emph{linear foliation} on the torus $X$. Choosing such a linear foliation amounts to choose a one dimensional $\CC$-linear subspace $W=\ker(adz+bdw)$ of $\CC^2$. If $W\cap \Lambda$ is a lattice in $W$, then the leaves of $\Fol$ are elliptic curves and the foliation is a fibration. If $W\cap \Lambda$ is non empty but not a lattice, then the leaves of $\Fol$ are biholomorphic to $\CC^*$. If $W\cap \Lambda$ is empty, then the leaves of $\Fol$ are biholomorphic to $\CC$ and are dense in $X$. When the leaves of $\Fol$ are not compact, we say the linear foliation $\Fol$ is \emph{irrational}.

\subsubsection{Obvious foliations on Hopf surfaces}\label{Hopfnormalform}
A \emph{Hopf surface} is a compact complex surface covered by $\CC^2\backslash\{0\}$. A \emph{primary Hopf surface} is a surface biholomorphic to the quotient of $\CC^2\backslash\{0\}$ by a transformation of the form 
\[H_{(\alpha,\beta,\gamma,m)}:(z,w)\mapsto (\alpha z+\gamma w^m,\beta w), \quad \alpha,\beta,\gamma \in\CC, m\in\NN^*, 0<\vert \alpha \vert \leq \vert \beta \vert<1\]
with $\alpha=\beta^m$ if $\gamma\neq 0$. We call the model $(\CC^2\backslash\{0\})/<H_{(\alpha,\beta,\gamma,m)}>$ a standard primary Hopf surface. Any Hopf surface has a finite unramified cover which is a primary Hopf surface. See \cite{Kod66} for the above assertion and the following:
\begin{enumerate}
	\item The standard primary Hopf surface has an elliptic fibration if and only if $\gamma=0$ and $\alpha^k=\beta^l$ for some $k,l\in\NN^*$.
	\item If $\gamma=0$ but $\alpha^k\neq\beta^l$ for any $k,l\in\NN^*$, then the two smooth elliptic curves given by $z=0$ and $w=0$ are the only curves on the primary Hopf surface.
	\item If $\gamma\neq 0$ then the smooth elliptic curve given by $w=0$ is the only curve on the primary Hopf surface.
\end{enumerate}

Suppose that $\gamma=0$. The complex lines in $\CC^2$ parallel to $z=0$ and to $w=0$ form two regular foliations on $\CC^2\backslash\{0\}$ that descend to regular foliations on the Hopf surface. A vector field of the form $az\frac{\partial}{\partial z}+bw\frac{\partial}{\partial w}$ with $a,b\in\CC$ and $ab\neq 0$ is invariant under $(z,w)\mapsto (\alpha z,\beta w)$, thus descends to the Hopf surface. It gives rise to a regular foliation whose leaves in $\CC^2\backslash \{0\}$ which are different from the two axes are given by $(e^{az},ce^{bz})$. When $a=b$ this is just the foliation by complex vector lines.

Suppose that $\gamma\neq 0$. The complex lines parallel to $w=0$ give a regular foliation on the Hopf surface. There exist also regular foliations associated with certain vector fields on $\CC^2\backslash\{0\}$ of Poincar\'e-Dulac's form $(mz+aw^m)\frac{\partial}{\partial z}+w\frac{\partial}{\partial w}$.
When $m=1$ the normal form $(z,w)\mapsto (\alpha z+\gamma w,\beta w)$ is a linear transformation and the foliation by complex vector lines in $\CC^2\backslash\{0\}$ is one of these foliations.

All the foliations described above will be called \emph{obvious foliations on Hopf surfaces}.

\subsubsection{Obvious foliations on Inoue surfaces}
See Paragraph \ref{Inouesurfaces} for the formulas defining Inoue surfaces. The vertical and horizontal foliations on $\HH\times \CC$ descend to two regular foliations on an Inoue surface of type $S^0$. Only the vertical foliation on $\HH\times \CC$ descends to a regular foliation on an Inoue surface of type $S^+$ or $S^-$. The foliations described above will be called \emph{obvious foliations on Inoue surfaces}.

\subsubsection{Suspensions}
Let $M$ be a Riemann surface. We denote by $\hat{M}$ a Galois covering of $M$ with deck transformation group $G$. Let $N$ be another Riemann surface. Let $\alpha:G\rightarrow \Aut(N)$ be a homomorphism of groups. The group $G$ acts on $\hat{M}\times N$ in the following way: $g\cdot (m,n)=(g\cdot m,\alpha(g)\cdot n)$. Then the quotient $X=(\hat{M}\times N)/G$ fibers onto $M=\hat{M}/G$ via the first projection; the fibers are all isomorphic to $N$. The foliation by $\{\hat{M}\times \{n\}\}_{n\in N}$ descends to a foliation on $X$ which is everywhere transverse to the fibration. We call such a foliation a \emph{suspension of $N$ over $M$} with monodromy $\alpha:G\rightarrow \Aut(N)$. If $M,N$ are compact, then $X$ is compact; this is the only case we will consider. If moreover $\alpha(G)$ is a finite subgroup of $\Aut(N)$, then the leaves of the foliation are compact and the foliation is in fact a fibration. When $\alpha(G)$ is infinite (since $N$ is compact this is only possible if $N$ is $\PP^1$ or an elliptic curve) the leaves are non-compact and we call such a suspension \emph{infinite}.

Remark that if $X$ is a compact complex surface with an infinite suspension foliation, then $X$ is \Kah{}. This is clear if it is a suspension of $\PP^1$ because then $X$ is ruled. Assume that it is a suspension of an elliptic curve. We use the notations in the previous paragraph. Up to replacing $G$ by a subgroup of finite index we can assume that $\alpha(G)\subset \Aut(N)$ is an abelian group of translations on the elliptic curve $N$. Thus there is an action of $N$ on $X$ by translations in the fibers of the elliptic bundle $X\rightarrow M$. In other words $X\rightarrow M$ is a principal elliptic bundle (see \cite{BHPV} V.5.1). Furthermore the suspension process says exactly that the bundle $X\rightarrow M$ is defined by a locally constant cocycle. By \cite{BHPV} V.5.1 and V.5.3 we infer that the second Betti number of $X$ is even. This implies that $X$ is \Kah{} (see \cite{BHPV} IV.3.1).

\subsubsection{Turbulent foliations} 
Let $X$ be a compact complex surface and $X\rightarrow B$ an elliptic fibration with constant functional invariant (i.e.\ all regular fibers are isomorphic) whose singular fibers are all multiples of smooth elliptic curves. Let $\Fol$ be a regular foliation on $X$. If a finite number of fibers of the elliptic fibration are $\Fol$-invariant and all other fibers are transverse to $\Fol$, then $\Fol$ is called a \emph{turbulent foliation}. The underlying elliptic fibration is locally trivial outside the invariant fibers of $\Fol$; the trivialization is given by the foliation. The invariant fibers are regular or multiples of elliptic curves. Locally around an invariant fiber, let $(z,w)$ be a system of local coordinates such that the fibration is given by $(z,w)\mapsto z^m$ where $m$ is the multiplicity of the fiber. Then the foliation $\Fol$ is locally defined by a local differential form $dz-A(z)dw$ where $A$ is a holomorphic function vanishing at $0$.

\subsubsection{Transversely hyperbolic foliations with dense leaves}

This is the most difficult type of foliations. We refer to \cite{Bru97} and \cite{Tou16} for precise definition and descriptions of transversely hyperbolic foliations. The only known example is as follows. Let $\Gamma$ be a torsion free cocompact irreducible lattice in $\PSL(2,\RR)\times \PSL(2,\RR)$, then the bidisk quotient $\bD\times\bD/\Gamma$ is a general type surface with two foliations induced by the product structure of $\bD\times \bD$. Each leaf is dense in $\bD\times\bD/\Gamma$ because the projection of an irreducible lattice onto a factor $\PSL(2,\RR)$ is dense. It is not known whether there exists other examples of regular foliations with dense leaves on general type surfaces. It is in this sense that Brunella's classification is still incomplete. Thanks to Deroin-Guillot's work that we will see shortly, these mysterious foliations will not appear in our study of birational Kleinian groups.

\subsection{Foliated complex projective structures}\label{foliatedcomplexprojectivestructuressection}
\subsubsection{Deroin-Guillot's Theorem}
Let $\Fol$ be a regular holomorphic foliation on a complex manifold $X$. A \emph{foliated $(\PGL(2,\CC),\PP^1)$-structure} on $\Fol$ is an open covering $\{U_i\}$ of $X$ and submersions $\phi_i:U_i\rightarrow \PP^1$, transverse to $\Fol$, immersive on the leaves, such that the restrictions on a leaf $L$ satisfy $\phi_i\vert _L \circ \phi_j\vert _L^{-1}\in \PGL(2,\CC)$. Obvious foliations on Hopf and Inoue surfaces, turbulent foliations and suspension foliations carry natural foliated $(\PGL(2,\CC),\PP^1)$-structures (see \cite{DG}). 

Though every Riemann surface carries a $(\PGL(2,\CC),\PP^1)$-structure, not every foliation carries a foliated $(\PGL(2,\CC),\PP^1)$-structure. The existence of a \emph{foliated $(\PGL(2,\CC),\PP^1)$-structure} imposes strong restrictions not only on the foliation but also on the topology of the manifold. We state one of the results of \cite{DG} in the generality that we need:
\begin{theorem}[Deroin-Guillot]
If $X$ is a compact complex surface having a regular foliation with a foliated $(\PGL(2,\CC),\PP^1)$-structure, then the Chern classes of $X$ satisfy $c_1(X)^2=2c_2(X)$.
\end{theorem}
The strict inequality $c_1^2>2c_2$ is satisfied by any Kodaira fibration (see \cite{BHPV} V.14) and by any general type surface equipped with a regular foliation with dense leaves which is not a bidisk quotient (see \cite{Bru97} the last paragraph). Therefore the following statement is a consequence of Deroin-Guillot's theorem:
\begin{corollary}\label{DGcorollary}
If $X$ is a general type surface having a regular foliation with a foliated $(\PGL(2,\CC),\PP^1)$-structure, then the universal cover of $X$ is the bidisk. 
\end{corollary}

It's well known that all $(\PGL(2,\CC),\PP^1)$-structures on a given Riemann surface form an affine space directed by the vector space of quadratic differentials. For a foliation $\Fol$, let $T_{\Fol}$ be the tangent bundle of $\Fol$. Consider holomorphic sections of $T_{\Fol}^{*2}$. They are quadratic differentials along the leaves of $\Fol$. If $\Fol$ admits at least one foliated $(\PGL(2,\CC),\PP^1)$-structure, then all foliated $(\PGL(2,\CC),\PP^1)$-structures on $\Fol$ form an affine space directed by the vector space of holomorphic sections of $T_{\Fol}^{*2}$ (see \cite{DG}).

\begin{proposition}\label{fibrationprojstructure}
Let $\Fol$ be a regular foliation on a surface $X$ defined by a fibration $f:X\rightarrow C$. If $\Fol$ has a foliated $(\PGL(2,\CC),\PP^1)$-structure then regular fibers of $f$ are isomorphic. If moreover the fibers of $f$ have genus $\geq 2$, then the $(\PGL(2,\CC),\PP^1)$-structures on the leaves are the same.
\end{proposition}

\begin{proof}
Since $\Fol$ is a foliation without singularities, the singular fibers of $f$ can only be multiples of smooth curves. Such a fibration is not isotrivial if and only if it is a Kodaira fibration. However we just said that there is no foliated $(\PGL(2,\CC),\PP^1)$-structure on a Kodaira fibration.

If fibers of $f$ have genus $\geq 2$, then there exists a finite unramified cover $\hat{C}$ of $C$ such that the fibration $\hat{f}:\hat{X}\rightarrow \hat{C}$ obtained by base change is a product. Consider the induced foliated $(\PGL(2,\CC),\PP^1)$-structure on $\hat{X}$ as a family of $(\PGL(2,\CC),\PP^1)$-structures on a fiber parametrized by $\hat{C}$. The conclusion follows because $(\PGL(2,\CC),\PP^1)$-structures on a given Riemann surface form an affine space and there is no non constant holomorphic map from $\hat{C}$ to an affine space.
\end{proof}

Let $\Gamma$ be a torsion free cocompact irreducible lattice in $\PSL(2,\RR)\times \PSL(2,\RR)$ and consider the bidisk quotient $X=\bD\times\bD/\Gamma$. Both foliations on $X$ are equipped naturally with a foliated $(\PGL(2,\CC),\PP^1)$-structure. All leaves are biholomorphic to the disk and have the same $(\PGL(2,\CC),\PP^1)$-structure: the standard one of a round disk in $\PP^1$.

\begin{proposition}\label{bidiskprojstructure}
If $\Fol$ is one of the natural foliation on the quotient of the bidisk by a torsion free cocompact irreducible lattice, then the natural foliated $(\PGL(2,\CC),\PP^1)$-structure is the unique foliated $(\PGL(2,\CC),\PP^1)$-structure on $\Fol$.
\end{proposition}

\begin{proof}
It suffices to show that $\operatorname{H}^0(X,T_{\Fol}^{*2})=\{0\}$. In fact for any positive integer $n$ the line bundle $T_{\Fol}^{*n}$ has no non-zero sections because $T_{\Fol}^{*}$ has Kodaira dimension $-\infty$ by \cite{Bru15} Chapter 9.5 Example 9.3.
\end{proof}

\section{Invariant rational fibration I}\label{rationalfibrationsectionone}

In this and the next two sections we suppose that $(Y,\Gamma,U,X)$ is a birational kleinian group and there is a $\Gamma$-invariant rational fibration $r:Y\rightarrow B$ over a smooth projective curve $B$. Note that when $B$ is not a rational curve, any birational transformation group of $Y$ preserves automatically $r$ because any curve transverse to the fibration is non rational whereas $\Gamma$ has to preserve rational curves. We have a group homomorphism $\Gamma\rightarrow \Aut(B)$ whose image will be denoted by $\Gamma_B$. The ruling $r$ is equivariant with respect to the action of $\Gamma$ on $Y$ and that of $\Gamma_B$ on $B$. 
If $\Gamma_B$ is infinite, then $B$ is $\PP^1$ or an elliptic curve. We will study in this section the case where $\Gamma_B$ is finite. The case where $\Gamma_B$ is infinite will be studied in Sections \ref{rationalfibrationsectiontwo} and \ref{rationalfibrationsectionthree}.

By contracting $(-1)$-curves which are contained in the fibers of $r$ but are disjoint from $U$, we can and will assume, without loss of generality, that if a fiber of $r$ intersects $U$, then its irreducible components of self-intersection $-1$ all intersect $U$.

\subsection{From invariant fibration to foliation}
We recall a basic fact about non relatively minimal ruled surfaces:
\begin{fact}\label{ruledsingfib}
Let $F$ be a singular fiber of a rational fibration. Then $F$ is a tree of rational curves whose components are of self-intersection $\leq -1$. 
\end{fact}

\begin{lemma}\label{nocontractedcurve}
Let $F$ be a fiber which intersects $U$. Then no element of $\Gamma$ has an indeterminacy point on $F$ and no component of $F$ is contracted by an element of $\Gamma$. 
\end{lemma}
\begin{proof}
Let $\gamma\in\Gamma$ be a non trivial birational transformation. Let us first show that $\gamma$ does not contract any component of $F$. If $F$ is a regular fiber, then it has only one irreducible component which is itself and this component intersects $U$ by hypothesis. Since the action of $\Gamma$ on $U$ is regular, $F$ cannot be contracted. 

Now assume that $F$ is a singular fiber.  Let $Y\xleftarrow{\epsilon} Z\xrightarrow{\delta} Y$ be the minimal resolution of indeterminacy of $\gamma$. Here $Z$ is a smooth projective surface, $\epsilon,\delta$ are birational morphisms and $\gamma=\delta\circ\epsilon^{-1}$. Suppose by way of contradiction that some component of $F$ is contracted by $\gamma$. Then there exists a component $C$ of $F$ such that the strict transform of $C$ in $Z$ has self-intersection $(-1)$ and is contracted by $\delta$. By Fact \ref{ruledsingfib} $C$ has self-intersection $\leq -1$. As the self-intersection number decreases after a blow up, the strict transform of $C$ in $Z$ has self-intersection $(-1)$ if and only if $C$ has self-intersection $(-1)$ and no point in $C$ is blown up by $\epsilon$. But by the hypothesis of our initial setting the $(-1)$-components of $F$ all intersect $U$ and can not be contracted by $\gamma$. This means $\delta$ does not contract $C$, contradiction. 

We remark that the total transform of $F$ by $\gamma$ is also a fiber of $r$ which intersects $U$. The above reasoning, applied to $\gamma^{-1}$, says that $\gamma^{-1}$ does not contract anything onto $F$, i.e.\ $\gamma$ does not have any indeterminacy point on $F$.
\end{proof}
We deduce immediately from Lemma \ref{nocontractedcurve} the following corollary:
\begin{corollary}\label{regularonacylinder}
The group $\Gamma$ acts by biholomorphisms on $r^{-1}(r(U))$. In particular if $r(U)=B$ then $\Gamma\subset \Aut(Y)$. 
\end{corollary}

The following proposition shows how holomorphic foliations arise in our study.
\begin{proposition}\label{withoutsingularpoints}
Let $(Y,\Gamma,U,X)$ be a birational kleinian group.
Suppose $\Gamma\subset \Bir(Y)$ preserves a fibration $r:Y\rightarrow B$. If a fiber $F$ of the invariant fibration intersects $U$, then $F\cap U$ contains no singular point of $F$. Therefore the fibration descends to a regular holomorphic foliation on $X$.
\end{proposition}
\begin{proof}
Suppose by way of contradiction that $F$ is a singular fiber and $p\in F\cap U$ is a singular point of $F$. If $\gamma\in \Gamma$, then
$\gamma(p)$ is a singular point of the singular fiber $\gamma(F)$ since $\Gamma$ preserves the fibration and acts 
by biholomorphisms on $U$. However a fibration has only finitely many singular fibers and each singular fiber has only finitely many singular points. Thus the infinite group $\Gamma$ does not act freely on $U$, contradiction.
\end{proof}
From now on we will denote by $\Fol$ the foliation on $X$ induced by the ruling $r$. 

\begin{proposition}\label{hasafoliatedprojectivestructure}
In the setting of Proposition \ref{withoutsingularpoints}, if moreover $r$ has no singular fibers then the foliation $\Fol$ is equipped with a foliated $(\PGL(2,\CC),\PP^1)$-structure along the leaves.
\end{proposition}
\begin{proof}
The open set $r^{-1}(r(U))$ is a $\PP^1$-bundle over $r(U)$. According to Corollary \ref{regularonacylinder} an element of $\Gamma$ sends a fiber over a point of $r(U)$ to another fiber by a M\"obius transformation. Then the proposition follows directly from the definition of foliated $(\PGL(2,\CC),\PP^1)$-structure.
\end{proof}

\subsection{Finite action on the base}
Recall that a decomposable ruled surface is a relatively minimal ruled surface with two disjoint sections.
\begin{theorem}\label{thmruled1}
Let $(Y,\Gamma,U,X)$ be a birational kleinian group on a surface $Y$ which is ruled over a curve $B$. Assume that $\Gamma$ preserves the ruling, and $\Gamma$ induces a finite action on the base curve $B$. Then up to geometric conjugation and up to taking finite index subgroup, we are in one of the following situations:
\begin{enumerate}
	\item $Y=B\times \PP^1$, $\Gamma\subset \{\Id\}\times \PGL(2,\CC)$ and $U=B\times D$ where $D\subset\PP^1$ is an invariant component of the domain of discontinuity of $\Gamma$ viewed as a classical Kleinian group.
	\item $Y$ is $\PP(\EE)$ where $\EE$ is an extension of $\OO_B$ by $\OO_B$. The extension determines a section $s$ of the ruling. We have $U=Y-s$. The subgroup of $\Aut_B(Y)$ fixing $s$ is isomorphic to $\CC$ in which $\Gamma$ is a lattice. The group $\Gamma$ is isomorphic to $\ZZ^2$ and the surface $X$ is a principal elliptic fiber bundle.
	\item $Y$ is obtained by blowing up a decomposable ruled surface, $U$ is a Zariski open set of $Y$ whose intersection with each fiber is biholomorphic to $\CC^*$. The group $\Gamma$ is cyclic and is generated by an automorphism which acts by multiplication in each fiber. The quotient surface $X$ is an elliptic fibration over $B$ with isomorphic regular fibers and whose only singular fibers are multiples of smooth elliptic curves.
\end{enumerate}
\end{theorem}

The rest of Section \ref{rationalfibrationsectionone} is the proof of Theorem \ref{thmruled1}.

\subsection{Fiberwise dynamics}
We start the proof of Theorem \ref{thmruled1} by making the following reduction: replacing $\Gamma$ with a finite index subgroup, we can and will assume that $\Gamma_B=\{\mathrm{Id}\}$. 

\begin{lemma}\label{ruledsurfacefullimage}
We have $r(U)=B$, $\Gamma \subset \Aut(Y)$. The foliation $\Fol$ on $X$ is a fibration.
\end{lemma}
\begin{proof}
 The set $V=r(U)$ is an open subset of $B$ because a holomorphic submersion is open. Since $\Gamma_B=\{\Id\}$ the fibration $r$ induces a commutative diagram:
\begin{equation*} 
\begin{tikzcd}
U \arrow[r]{}{\pi} \arrow{d}{} & X \arrow{d}{} \\
V \arrow{r}{} & V 
\end{tikzcd} 
\end{equation*}
As $X$ is compact, $V$ should be compact too. This implies that $V$ is the whole curve $B$. By Corollary \ref{regularonacylinder}, we have $\Gamma \subset \Aut(Y)$. The leaves of the foliation on $X$ are the fibers of the map $X\rightarrow V$; they are compact. Thus the foliation is a fibration.
\end{proof}

Since $\Gamma\subset \Aut(Y)$, we can and will assume, up to replacing $\Gamma$ by a finite index subgroup, that any irreducible curve contained in a fiber is $\Gamma$-invariant. In particular every singular point of a singular fiber is fixed by $\Gamma$.

\begin{proposition}\label{geometricruled}
If $F\cong \PP^1$ is a regular fiber of $r:Y\rightarrow B$, then $F\cap U$, as a subset of $\PP^1$, is a union of components of the domain of discontinuity of a classical Kleinian group. If moreover $r:Y\rightarrow B$ has at least one singular fiber, then for any fiber $F$ (singular or non-singular), the intersection $F\cap U$ is biholomorphic to $\CC^*$ and $\Gamma$ is cyclic.
\end{proposition}
\begin{proof}
Let $F_b$ be the (possibly singular) fiber over a point $b\in B$ such that $F_b\cap U\neq \emptyset$. There is a homomorphism $\Gamma \rightarrow \Aut (F_b)$ with image $\Gamma_b$ and $\Gamma_b$ preserves $U\cap F_b$. The group $\Gamma_b$ is isomorphic to $\Gamma$ because the action of $\Gamma$ on $U$ is free. We have a commutative diagram
 \begin{equation}\label{ruleddiagram} \begin{CD}
 \begin{tikzcd}
U \arrow[r]{}{\pi} \arrow{d}{r} & X \arrow{d}{f} \\
B \arrow{r}{} & B
\end{tikzcd}
 \end{CD} \end{equation}
where $f$ is a surjective proper holomorphic map. By \cite{BHPV} III.11 the sheaf $f_* \OO_X$ is locally free and the non-singular fibers of $f$ have the same number of connected components. If $F_b$ is a non-singular fiber then it is isomorphic to the projective line $\PP^1$. The quotient of $U\cap F_b$ by the action of $\Gamma_b$ is a general fiber of $f$, that is, a disjoint union of smooth compact curves. This means that $\Gamma_b$ is a classical Kleinian group and $U\cap F_b$ is a union of connected components of the discontinuity set of $\Gamma_b$.

Suppose that the fiber $F_b$ is singular. Every irreducible component of $F_b$ contains at least one singular point. If an irreducible component of $F_b$ intersects $U$ then it contains at most two singular points because otherwise it would be pointwise fixed by $\Gamma_b$. In particular $\Gamma_b$ restricted to each component of $F_b$ is isomorphic to a solvable subgroup of $\PGL(2,\CC)$. So if there exists at least one singular fiber, then the group $\Gamma$ is solvable. A solvable torsion free classical Kleinian group is elementary, i.e.\ fixes a point in $\HH^3\cup\PP^1$ or preserves a geodesic in $\HH^3$, thus isomorphic to $\ZZ$ or $\ZZ^2$ and its domain of discontinuity is biholomorphic to $\CC$ or $\CC^*$. This implies that the intersection of $U$ with any fiber is biholomorphic to $\CC$ or $\CC^*$. In particular $f:X\rightarrow B$ has connected fibers. Thus for the singular fiber $F_b$ the intersection $F_b\cap U$ is connected and is contained in one component $C_b$ of $F_b$. The component $C_b$ is necessarily a $(-1)$-curve because by our hypothesis made at the beginning of Section \ref{rationalfibrationsectionone} every (-1)-curve in the fiber $F_b$ intersects $U$. 

Under the hypothesis that there is a singular fiber $F_b$, we need to rule out the situation where $\Gamma$ is isomorphic to $\ZZ^2$ and the intersection of $U$ with any fiber is biholomorphic to $\CC$. Suppose to the contrary that $\Gamma$ is isomorphic to $\ZZ^2$. Then the fibration $f:X\rightarrow B$ is a genus one fibration whose singular fibers are multiples of smooth elliptic curves. Non-singular fibers of such a genus one fibration are all isomorphic to each other because otherwise the $j$-function would be a non constant holomorphic function on $B$ (see \cite{BHPV} V.10). Let $\Delta$ be a small disk in $B$ centred at $b$ and let $(x,y)$ be the local coordinates around a point of the singular fiber $F_b$ where $r$ is given by $x$ and $F_b$ is defined by $x=0$. The isotriviality of $f$ implies that, locally over $\Delta\backslash\{0\}$, up to a holomorphic change of coordinates the action of $\Gamma$ is generated by two transformations $\gamma_j:(x,y)\mapsto (x,y+a_j), j=1,2$ where $\ZZ a_1+\ZZ a_2$ is a lattice in $\CC$. In other words there is birational morphism $r^{-1}(\Delta)\rightarrow \Delta\times \PP^1$, composition of blow ups, which sends $F_b$ to $\PP^1$ such that the action of $\Gamma$ on $r^{-1}(\Delta)$ is the lift of the action of $\ZZ^2$ on $\Delta\times \PP^1$ generated by $\gamma_j:(x,y)\mapsto (x,y+a_j), j=1,2$. Thus $r^{-1}(\Delta)$ is necessarily obtained from $\Delta\times \PP^1$ by blowing up the common fixed points of $\gamma_j, j=1,2$ (and a priori other common fixed infinitely near points). However the action of $\gamma_j, j=1,2$ on the exceptional divisors are always trivial. Thus there was no blowup. So $U$ does not intersect any exceptional divisor. Since $U$ must, by assumption, intersect any $(-1)$-curve in a fiber, there are no exceptional curves and $r=\operatorname{id}$. This contradicts that $F_b$ is a singular fiber.
\end{proof}

\subsection{Automorphism groups of geometrically ruled surfaces}\label{maruruled}
Before continuing our\\ study of birational Kleinian groups preserving a rational fibration, we collect some preliminaries on automorphism groups of geometrically ruled surfaces, following Maruyama's paper \cite{Maru}. Our notations are those of \cite{Hart} but are different from those of \cite{Maru}. Let $Y$ be a geometrically ruled surface (i.e.\ relatively minimal). The geometrically ruled surface $Y$ over $B$ is isomorphic to the projective bundle $\PP(\EE)$ where $\EE$ is a rank two vector bundle (not unique) over $B$ such that $\mathrm{H}^0(B,\EE)\neq \{0\}$ and $\mathrm{H}^0(B,\EE\otimes \mathcal{L})=\{0\}$ for every line bundle $\mathcal{L}$ with $\deg(\mathcal{L})<0$. The opposite of the degree of $\bigwedge^2\EE$, denoted by $e$, is an invariant of $Y$; the integer $-e$ is the minimal self-intersection number of a section of $Y\rightarrow B$. A section with self-intersection number $-e$ is called a minimal section; the line sub-bundle of $\EE$ which corresponds to the minimal section is called a maximal line bundle. The geometrically ruled surface $Y$ is called \emph{decomposable} if there exists a decomposition $\EE=\mathcal{L}_1\oplus \mathcal{L}_2$ where $\mathcal{L}_1, \mathcal{L}_2$ are line bundles, otherwise $Y$ is called indecomposable. There exist two sections of $Y\rightarrow B$ which do not intersect each other if and only if $Y$ is decomposable.

The group of automorphisms of $Y$ which preserve fiberwise the ruling is denoted by $\Aut_B(Y)$. 
\begin{theorem}[Maruyama \cite{Maru}]\label{Maruyamathm}
 \begin{enumerate}
  \item If $e<0$, then $\Aut_B(Y)$ is a finite group.
  \item If $e\geqslant 0$ and if $Y$ is indecomposable, then $\Aut_B(Y)\cong\CC^k$ where $\mathcal{L}\subset \EE$ is the unique maximal line bundle and $k=\dim \mathrm{H}^0(B,(\det\EE)^{-1}\otimes \mathcal{L}^2)$.
  \item If $Y\cong\PP^1\times B$, then $\Aut_B(Y)\cong\PGL(2,\CC)$.
  \item If $Y\cong\PP(\EE)$ with $\EE=\mathcal{L}\oplus (\mathcal{L}\otimes \mathcal{N})$, $\mathcal{N}^2= \OO_B$ 
  and $\mathcal{N}\neq \OO_B$, then $\Aut_B(Y)$ is an extension of $\CC^*$ by $\ZZ/2\ZZ$.
  \item If $Y$ is decomposable and does not fit in the previous two cases, then $\Aut_B(Y)$ is isomorphic to the following subgroup of $\GL({k+1},\CC)$:
  \[
  H_k=\Bigg\lbrace 
  \begin{pmatrix}
         \alpha & t_1 & \cdots &\cdots & t_k\\
         0 & 1 & 0&\cdots &0\\
         \vdots & &\ddots & &\vdots\\
         \vdots & & &\ddots &\vdots\\
         0 & &\cdots &0 & 1
  \end{pmatrix}\alpha \in \CC^*, t_1,\cdots,t_k \in \CC
 \Bigg\rbrace
  \]
  where $k=\dim \mathrm{H}^0(B,(\det\EE)^{-1}\otimes \mathcal{L}^2)$ and $\mathcal{L}$ is a maximal line bundle.
 \end{enumerate}
\end{theorem}
In the second and last case, the action of an element of $H_k$ on the fiber $F_b$ over $b\in B$ is basically $x\mapsto \alpha x+t_1l_1(b)+\cdots+t_kl_k(b)$ where $l_1,\cdots,l_k$ form a base of $\mathrm{H}^0(B,(\det\EE)^{-1}\otimes \mathcal{L}^2)$ (see \cite{Maru} for the exact meaning of the formula).

We deduce immediately from Maruyama's theorem the following:
\begin{corollary}\label{nonsolvablethenproduct}
If $\Aut_B(Y)$ is infinite and not solvable, then $Y=B\times \PP^1$.
\end{corollary}

\subsection{Proof of Theorem \ref{thmruled1}}
We turn back to the proof of Theorem \ref{thmruled1} without assuming that $Y$ is geometrically ruled. We still assume that $\Gamma_B=\{\mathrm{Id}\}$ and that any component of any singular fiber of $r$ is fixed by $\Gamma$. By Proposition \ref{geometricruled} we know that if $Y$ is not geometrically ruled then the intersection of $U$ with a regular fiber is biholomorphic to $\CC^*$.

\subsubsection{Non-abelian case}
\begin{proposition}\label{propruledtrivial}
Suppose that $\Gamma$ is not solvable. Then $Y=\PP^1\times B$ and $U=W\times B$ where $W\subset \PP^1$ is an invariant component of $\Gamma$, viewed as a non-elementary classical Kleinian group. The quotient surface $X$ is the product of $B$ with a hyperbolic Riemann surface.
\end{proposition}
\begin{proof}
In particular $\Gamma$ is not cyclic and Proposition \ref{geometricruled} implies that $Y$ is geometrically ruled. By Corollary \ref{nonsolvablethenproduct} $Y= \PP^1\times B$. In this case $\Aut_B(Y)=\PGL(2,\bC)$ and the intersection of $U$ with each fiber is a union of components of some classical Kleinian group. This classical kleinian group does not depend on the fiber. Since $U$ is connected, the intersection of $U$ with each fiber is the same component of the discontinuity set. 
\end{proof}

\subsubsection{Translation in the fibers}
As $\Gamma$ is isomorphic to a classical Kleinian group by Proposition \ref{geometricruled}, if it is solvable then it is isomorphic to $\ZZ^2$ or $\ZZ$, depending on whether the intersection of $U$ with a fiber is $\CC$ or $\CC^*$. 

We first investigate the case where $\Gamma=\ZZ^2$; in this case $\Gamma$ acts on each fiber by translations (translation means parabolic automorphism of $\PP^1$). In case 2 (resp. case 5) of Maruyama's theorem, the action is given, in local coordinate $z$ of the fiber, by
\begin{equation}
 z\mapsto z+t_1 e_1+\cdots+t_k e_k \ (\text{resp. } z\mapsto \alpha z+t_1 e_1+\cdots+t_k e_k)
\end{equation}
where $e_1,\cdots,e_k$ extend to a global base of $\mathrm{H}^0(B,(\det\EE)^{-1}\otimes \mathcal{L}^2)$. Since a line bundle has a nowhere vanishing section if and only if it is trivial, an automorphism can act by non-trivial translation in every fiber only if the line bundle $(\det\EE)^{-1}\otimes \mathcal{L}^2$ is trivial. Writing $0\rightarrow \mathcal{L} \rightarrow \EE \rightarrow \mathcal{L}'\rightarrow 0$, we have $(\det\EE)^{-1}\otimes \mathcal{L}^2=\mathcal{L}\otimes(\mathcal{L}')^{-1}$. Therefore if $(\det\EE)^{-1}\otimes \mathcal{L}^2$ is trivial then $\mathcal{L}=\mathcal{L}'$ and $\mathcal{L}^{-1}\otimes\EE$ is an extension of two trivial line bundles. Since $\PP(\EE)$ is isomorphic to $\PP(\mathcal{L}^{-1}\otimes\EE)$ we obtain:
\begin{proposition}\label{propruledtrivial2}
Suppose that the intersection of $U$ with a fiber is $\CC$. There is a rank two vector bundle $\EE$ on $B$ such that
\begin{enumerate}
	\item $Y$ is isomorphic to $\PP(\EE)$;
	\item We can write $\EE$ as an extension $0\rightarrow \OO_B \rightarrow \EE \rightarrow \OO_B\rightarrow 0$ such that $U=Y-s(B)$ where  $s:B\rightarrow Y$ is the section determined by this extension, i.e.\ $s\mapsto s(b)$ with $s(b)=\PP(\OO_b)\in \PP(\EE_b)$.
\end{enumerate}
As $(\det\EE)^{-1}\otimes \mathcal{L}^2=\OO_B$, the subgroup of $\Aut_B(Y)$ fixing $s(B)$ is isomorphic to $\CC$ by Theorem \ref{Maruyamathm}. We can identify $\Gamma$ as a lattice in this subgroup. The surface $X$ is an elliptic fiber bundle over $B$ with fibers isomorphic to $\CC/\Gamma$.
\end{proposition}

The action of $\CC$ on $U$ descends to an action of $\CC/ \Gamma$ on $X$; this implies that $X$ is a principal elliptic bundle by Lemma V.5.1 of \cite{BHPV}. A principal elliptic bundle is topologically a product of the circle $S^1$ with an $S^1$-bundle (Proposition V.5.2 of \cite{BHPV}). We claim that in our case the corresponding $S^1$-bundle is a product. The reason is that the rank two vector bundle $\EE$ is topologically trivial because an extension of a trivial bundle by a trivial bundle always splits in the topological category. If $B=\PP^1$ and $X$ is \Kah{}, then the only possibility is $X=(\CC/\Gamma )\times \PP^1$; if $B$ is an elliptic curve and $X$ is \Kah{}, then $X$ is a complex torus or a hyperelliptic surface (see \cite{BHPV} Chapter V).

\subsubsection{Multiplication in the fibers}\label{blowupellipticconstruction}
We now investigate the case where $\Gamma=\ZZ$. Let us see first that it is indeed possible that $Y$ is not geometrically ruled (cf. Proposition \ref{geometricruled}).

Given a decomposable geometrically ruled surface $Y_0=\PP(\EE)$ over $B$, up to tensoring $\EE$ with a line bundle, we can suppose that it decomposes as $\EE=\OO_B\oplus \mathcal{L}$. The decomposition determines two sections $S_1,S_2$ on $Y_0$ without intersection. Pick a finite family of fibers $F_1,\cdots,F_n$ over $b_1,\cdots,b_n\in B$. Choose a positive integer $m\leqslant n$. Denote by $p_1,\cdots,p_m$ the intersection points of $F_1,\cdots,F_m$ with $S_1$ and $p_{m+1},\cdots,p_n$ the intersection points of $F_{m+1},\cdots,F_n$ with $S_2$. Construct a birationally ruled surface $Y$ over $B$ as follows: $Y$ is obtained from $Y_0$
by successive blow-ups at $p_1,\cdots,p_n$ and some infinitely near points. For $i\leq m$ (resp. $i> m$) the successive blow-ups on the fiber over $p_i$ are executed either at the intersection point of two irreducible components of the fiber or at the intersection point of the fiber with (the strict transform of) $S_1$ (resp. $S_2$). We denote by $\pi:Y\rightarrow Y_0$ the contraction map. The singular fibers of $Y\xrightarrow{\phi}B$ are chains of smooth rational curves, i.e. the singular fiber corresponding to $F_i$ is 
\[
\tilde{F_i}=\sum_{k=0}^{l(i)} d_{i,k}F_{i,k};
\]
here $F_{i,0}$ is the strict transform of $F_i$ and the $F_{i,k}$ are all smooth rational curves among which the only non vanishing intersection pairings are $F_{i,k}\cdot F_{i,k+1}=1$. We have $d_{i,0}=1$ and $d_{i,1},\cdots,d_{i,l(i)}\in \NN^+$. Each $F_{i,k}$ has two distinguished points $\{0_{i,k},\infty_{i,k}\}$ which are the two intersection points of $F_{i,k}$ with other components of the fiber or with (strict transforms of) the two sections $S_1,S_2$. Let's say that the one which is closer to $S_1$ is $0_{i,k}$.
\definecolor{ududff}{rgb}{0.30196078431372547,0.30196078431372547,1}
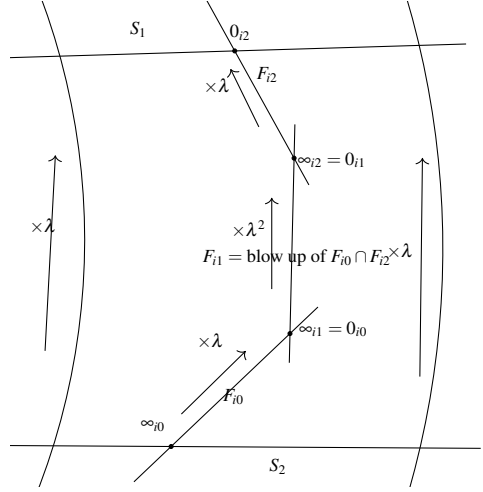
\begin{center}
\begin{tikzpicture}[scale=1.5][line cap=round,line join=round,x=1cm,y=1cm]
\clip(-1.8,-1.8) rectangle (3,2.5);
\draw [line width=0.4pt] (-2,2)-- (2.2416385876815585,2.1271400955600788);
\draw [line width=0.4pt] (-2.0050731618034243,-1.4333651822093394)-- (2.371111506873174,-1.4592597660476623);
\draw [shift={(-7.297028138613542,0.3908467755477395)},line width=0.4pt]  plot[domain=-0.3943144753125969:0.39386861073636076,variable=\t]({1*6.158445205849685*cos(\t r)+0*6.158445205849685*sin(\t r)},{0*6.158445205849685*cos(\t r)+1*6.158445205849685*sin(\t r)});
\draw [shift={(-5.8087863060660085,0.314526681570943)},line width=0.4pt]  plot[domain=-0.30636715261604675:0.34453863177974037,variable=\t]({1*7.844651362997909*cos(\t r)+0*7.844651362997909*sin(\t r)},{0*7.844651362997909*cos(\t r)+1*7.844651362997909*sin(\t r)});
\draw [line width=0.4pt] (-0.19420441022575907,2.762024182327671)-- (0.8485122739501322,0.8771132532404875);
\draw [line width=0.4pt] (0.7241408066385573,1.4141086370680898)-- (0.674726159920817,-0.7003299356409851);
\draw [line width=0.4pt] (0.92872124965597,-0.20570791878831993)-- (-0.7021945896960651,-1.7564147824345133);
\draw [->,line width=0.4pt] (-1.4898923043411383,-0.5945763666089969) -- (-1.3985730493273896,1.1404894786522255);
\draw [->,line width=0.4pt] (0.5191313059613344,-0.046660836526505683) -- (0.5191313059613344,0.7637975517205127);
\draw [->,line width=0.4pt] (0.4049822371941485,1.391617429940034) -- (0.15385428590633937,1.9052882393923696);
\draw [->,line width=0.4pt] (-0.28190205233612664,-1.155399211263577) -- (0.28884329149980315,-0.5960687743043666);
\draw [->,line width=0.4pt] (1.831845596783973,-0.8228745041433684) -- (1.8546754105374101,1.1176596648987882);
\begin{scriptsize}
\draw[color=black] (-0.6601902867384064,2.256178060276532) node {$S_1$};
\draw[color=black] (0.5767553984320192,-1.6149344375173749) node {$S_2$};
\draw[color=black] (0.48068194715664636,1.8557189898054542) node {$F_{i2}$};
\draw[color=black] (0.7409573894392947,0.2305529508089169) node {$F_{i1}=\text{blow up of}\ F_{i0}\cap F_{i2}$};
\draw[color=black] (0.18045241192110614,-1.0024661856368757) node {$F_{i0}$};
\draw[color=black] (-1.5049890693580865,0.51461722067486784) node {$\times \lambda$};
\draw[color=black] (0.31647995614937096,0.49059885785602466) node {$\times \lambda^2$};
\draw[color=black] (0.03634223500804685,1.7515629058452879) node {$\times \lambda$};
\draw[color=black] (-0.0239332072746014,-0.5022366504013368) node {$\times \lambda$};
\draw[color=black] (1.6575358182328558,0.2945254065806523) node {$\times \lambda$};
\draw [fill=black] (0.19101059420271718,2.0656739820146583) circle (0.5pt);
\draw[color=black] (0.2404983189682142,2.2301046090011595) node {$0_{i2}$};
\draw [fill=black] (0.7171410712537292,1.1145919658070615) circle (0.5pt);
\draw[color=black] (1.0529665708487162,1.1) node {$\infty_{i2}=0_{i1}$};
\draw [fill=black] (0.6807759969533673,-0.4414591426694826) circle (0.5pt);
\draw[color=black] (1.0649757522581378,-0.4) node {$\infty_{i1}=0_{i0}$};
\draw [fill=black] (-0.37259479898099074,-1.4430248174923122) circle (0.5pt);
\draw[color=black] (-0.532245375194936,-1.2625120926839835) node {$\infty_{i0}$};
\end{scriptsize}
\end{tikzpicture}\end{center}

By Theorem \ref{Maruyamathm}, the subgroup of $\Aut_B(Y_0)$ fixing $S_1,S_2$ is isomorphic to $\CC^*$; it acts by multiplication in the fibers. The $\CC^*$-action lifts to an action on $Y$ by automorphisms because the blow-ups which we did to obtain $Y$ are all at the fixed points of the $\CC^*$-action. An element $\lambda\in \CC^*$ acts on a component $F_{i,k}$ of the singular fiber $\tilde{F_i}$ by multiplication by $\lambda^{d_{i,k}}$ with fixed points $\{0_{i,k},\infty_{i,k}\}$.

Now for each $\tilde{F_{i}}$, choose a component $F_{i,k_0(i)}$ and consider the open set
\begin{equation*}
U=Y\backslash \left ( S_1\bigcup S_2 \bigcup (\cup_{i}\cup_{k\neq k_0(i)}F_{i,k}) \right ).
\end{equation*}
We remark that $U$ is Zariski open in $Y$. The ruling $\phi$ restricted to $U$ is a surjective holomorphic map onto $B$ with fibers biholomorphic to $\CC^*$. The quotient of $U$ by $\lambda\in \CC^*, \vert \lambda\vert \neq 1$ is a compact complex surface $X$ which admits an elliptic fibration over $B$. If a fiber of $X\rightarrow B$ comes from a non-singular fiber of $\phi:Y\rightarrow B$, i.e.\ the fiber over a point that is not one of the $p_1,\cdots,p_n$, then it is isomorphic to the elliptic curve $\CC^*/<\lambda>$. If a fiber of $X\rightarrow B$ comes from $F_i$, then it is the quotient of $\CC^*=F_{i,k_0(i)}\backslash\{0_{i,k_0(i)},\infty_{i,k_0(i)}\}$ by $<\lambda^{d_{i,k_0(i)}}>$. Hence the only singular fibers of the elliptic fibration $X\rightarrow B$ are multiples of smooth elliptic curves, the multiplicities being the coefficient $d_{i,k_0(i)}$.

Actually the above construction exhausts all possibilities:
\begin{proposition}\label{propruled3}
If the intersection of $U$ with a fiber is $\CC^*$, i.e.\ if $\Gamma=\ZZ$ then $X$ is obtained from the above process. In particular the surface $X$ is an elliptic fibration over $B$ whose singular fibers are all multiples of smooth elliptic curves.
\end{proposition}
\begin{proof}
Recall that by Lemma \ref{ruledsurfacefullimage} $\Gamma\subset \Aut(Y)$. We have also assumed that $\Gamma$ preserves any component of any fiber. Therefore the action of $\Gamma$ descends to an action by automorphisms, after contracting some exceptional curves in a fiber. Continuing the contraction until we get a geometrically ruled surface, we infer that the action of $\Gamma$ on $Y$ comes from a $\CC^*$-action on a geometrically ruled surface $Y_0$; and that $Y$ is obtained from $Y_0$ by blowing-up fixed points (or infinitely near fixed points) of the $\CC^*$-action. This is exactly how the above process works.
\end{proof}

We denote by $X_0$ the quotient of $Y_0\backslash (S_1\cup S_2)$ by $<\lambda>$. It is a principal elliptic fiber bundle over $B$ because the $\CC^*$-action descends to $X_0$ (Proposition V.5.2 of \cite{BHPV}). The surface $X$ is obtained from $X_0$ by logarithmic transformations which replace the fibers over $b_1,\cdots,b_n\in B$ with $d_{1,k_0(1)}E_1,\cdots,d_{n,k_0(n)}E_n$ where $E_i=\CC^*/<\lambda^{d_{i,k_0(i)}}>$ (see \cite{Kod64} for logarithmic transformations).

\section{Invariant rational fibration II: elliptic base}\label{rationalfibrationsectiontwo}
The base curve $B$ is assumed to be an elliptic curve in this section.

We continue to use the notations of section \ref{maruruled}. It is well known that, up to isomorphisms, there are only two indecomposable geometrically ruled surfaces over an elliptic curve $B$: $R_0$ with $e=0$ and $R_1$ with $e=1$ (see \cite{Hart} Chapter 5). Now we describe the whole automorphism group of a geometrically ruled surface $Y$ according to Maruyama \cite{Maru}:

\begin{enumerate}
	\item If $Y=\PP^1\times B$, then $\Aut(Y)=\PGL(2,\CC)\times \Aut(B)$.
	\item If $Y$ is decomposable and $e\neq 0$, then we have an exact sequence of algebraic groups
	\[
	0\rightarrow \Aut_B(Y) \rightarrow \Aut(Y) \rightarrow G \rightarrow 0
	\]
	where $G$ is a finite subgroup of $\Aut(B)$.
	\item If $Y$ is decomposable with $e=0$ and if $Y\neq \PP^1\times B$, then we have an exact sequence of algebraic groups
	\[
	0\rightarrow \Aut_B(Y) \rightarrow \Aut(Y) \rightarrow G \rightarrow 0
	\]
	where $\Aut_B(Y)$ is isomorphic to $\CC^*$ or an extension of $\CC^*$ by $\ZZ/2\ZZ$, and $G$ is a subgroup of $\Aut(B)$
	containing $\Aut^0(B)$. 
	If $s_1$, $s_2$ are the minimal sections, then $Y-(s_1\cup s_2)$ is a principal $\CC^*$-bundle over $B$ which
	has a structure of commutative algebraic group. The algebraic group $Y-(s_1\cup s_2)$ is isomorphic to the connected component of identity $\Aut^0(Y)$.
	\item If $Y=R_0$, then we have an exact sequence of algebraic groups
	\[
	0\rightarrow \CC \rightarrow \Aut(Y) \rightarrow \Aut(B) \rightarrow 0.
	\]
	If $s$ is the minimal section, then $Y-s$ is a principal $\CC$-bundle over $B$ which has the structure of  commutative algebraic group.
	The algebraic group $Y-s$ is isomorphic to the connected component of identity $\Aut^0(Y)$.
	\item If $Y=R_1$, then we have 
	\[
	0\rightarrow \Aut_B(Y) \rightarrow \Aut(Y) \rightarrow \Aut(B) \rightarrow 0
	\]
	where $\Aut_B(Y)$ is isomorphic to $\ZZ/2\ZZ\times \ZZ/2\ZZ$.
\end{enumerate}

We prove the following statement which encapsulates both the case where $\Gamma_B$ is finite and the case where it is infinite:
\begin{theorem}\label{thmruledelliptic}
Let $(Y,\Gamma,U,X)$ be a birational kleinian group on a surface $Y$ which is ruled over an elliptic curve $B$. 
Then up to geometric conjugation and up to taking finite index subgroup, we are in one of the following situations:
\begin{enumerate}
  \item $Y=\PP^1\times B$, $U=D\times B$ where $D\subset \PP^1$ is an invariant component of the domain of discontinuity of a classical Kleinian group. $X$ is a fiber bundle over $B$ or a suspension of $B$.
	\item $Y=R_0$, $U=Y-s$ and $U$ has the structure of an algebraic group. As a complex analytic Lie group $U$ is isomorphic to $\CC^*\times\CC^*$ and the group $\Gamma$ is isomorphic to a lattice in $U$. The quotient $X$ is a complex torus.
	\item $Y$ is a decomposable geomtrically ruled surface with $e=0$, $U=Y-(s_1\cup s_2)$ and $X$ is a complex torus.
	\item As in Case 3) of Theorem \ref{thmruled1}. 	
\end{enumerate}
\end{theorem}
In any case of the above theorem, the open set $U$ is not simply connected. Except in the first case, $U$ is a Zariski open set.

In the rest of this section $(Y,\Gamma,U,X)$ will be a birational Kleinian group satisfying the hypothesis of the theorem. 

\begin{lemma}\label{lemmaellipticbasegeometricallyruled}
Suppose that $\Gamma_B$ is an infinite group. Up to geometric conjugation of birational Kleinian groups, $Y$ is $\PP^1$-fiber bundle, $r(U)=B$ and $\Gamma\subset \Aut(Y)$.
\end{lemma}
\begin{proof}
The open set $r(U)$ is invariant under the infinite group $\Gamma_B$. Suppose by way of contradiction that $r(U)\neq B$. Then according to Lemma \ref{ellipticnondiscrete}, it is a connected component of the complement of a finite union of real subtori, thus a band-like region. The action of $\Gamma_B$ on $r(U)$ is by translations along the direction of the band. Such an action on a band is never cocompact, i.e.\ the union of the images under $\Gamma$ of a compact subset of the band never covers the band. This contradicts the cocompactness of the action of $\Gamma$ on $U$.

By Corollary \ref{regularonacylinder}, the fact $r(U)=B$ implies that $\Gamma\subset \Aut(Y)$. Since $\Gamma_B$ is infinite and $B$ is an elliptic curve, every point of $B$ has an infinite $\Gamma_B$-orbit. If there existed a singular fiber of $r$, then its translations by elements of $\Gamma$ would give rise to an infinite number of singular fibers. Therefore $Y$ is geometrically ruled.
\end{proof}

The ruling $r:Y\rightarrow B$ is identified with the Albanese morphism and is preserved by each element of $\Gamma$. We have an exact sequence
\[
0\rightarrow\Gamma_r\rightarrow\Gamma\rightarrow\Gamma_B\rightarrow 0
\]
where $\Gamma_r$ is the subgroup preserving fiberwise the fibration.

\begin{proposition}\label{gammabinfinite}
Suppose that $\Gamma_B$ is infinite. Then the foliation $\Fol$ is a linear irrational foliation or a suspension of $B$.  
\end{proposition}
\begin{proof}
We apply Theorem \ref{Brunellathm} to the foliation $\Fol$ on $X$ induced by the rational fibration $Y\rightarrow B$. The orbit of a point on $B$ under $\Gamma_B$ is not discrete and its closure contains a real subtorus (cf. Lemma \ref{ellipticnondiscrete}). This implies that no leaf of $\Fol$ is a submanifold of $X$. Thus the foliation is neither a fibration, nor a turbulent foliation, nor an obvious foliation on a Hopf surface. Here the base curve $B$ is elliptic so that $\Fol$ is transversely Euclidean. This rules out obvious foliations on Inoue surfaces (cf. \cite{Zhaobirinoue}). The foliation $\Fol$ cannot be a suspension of $\PP^1$ because the only rational curves in $Y$ are the fibers of the rational fibration. There are only two possibilities left: either $X$ is a complex torus with an irrational linear foliation or $X$ is an infinite suspension of a hyperbolic curve over an elliptic curve. 
\end{proof}

\begin{proof}[Proof of Theorem \ref{thmruledelliptic}]
We already treated the case where $\Gamma_B$ is finite in Theorem \ref{thmruled1}. This is covered by case 1), 2), 4) of Theorem \ref{thmruledelliptic}. From now on we assume that $\Gamma_B$ is infinite. We analyse the situation case by case according to the above description of automorphism group. Since $\Aut_B(Y)$ is infinite we infer that $Y$ is not decomposable with $e\neq 0$. The surface $Y$ is not $R_1$ because $\Aut(R_1)$ is compact and does not contain infinite discrete subgroup.  

Assume that $Y$ is decomposable with $e=0$. The two minimal sections $s_1$ and $s_2$ are disjoint from $U$ because they are invariant under $\Aut(Y)$. The universal cover of the complex analytic Lie group $Y-(s_1\cup s_2)$ is $\CC^2$ because it is commutative. Hence we obtain immediately that $U=Y-(s_1\cup s_2)$ and $X$ is a complex torus. If $Y=R_0$, then the same reasoning tells us that $U=Y-s$ and $X$ is a complex torus.

Assume that $Y=\PP^1\times B$. By Proposition \ref{gammabinfinite} the foliation on $X$ induced by $Y\rightarrow B$ is a linear foliation on a torus or a suspension of $B$ over a hyperbolic curve. If $X$ is a torus then the intersection of $U$ with any fiber of $Y\rightarrow B$ is biholomorphic to $\CC$ or $\CC^*$ because the intersections with fibers cover leaves of the foliation. The conclusion follows immediately in this case. Assume that $X$ is a suspension over a hyperbolic curve. The projection $Y\rightarrow \PP^1$ induces a second regular foliation on $X$ which is transverse to the first one; this second foliation must be the elliptic fibration subjacent to the suspension. Therefore we are in the first case of Theorem \ref{thmruledelliptic}.
\end{proof}

\section{Invariant rational fibration III: rational base}\label{rationalfibrationsectionthree}

In this section we work under the assumption that $B=\PP^1$ and $\Gamma_B$ is infinite. We will investigate the foliation $\Fol$ case by case according to Theorem \ref{Brunellathm} and Corollary \ref{DGcorollary}. 

The situation where $\Fol$ is an obvious foliation on an Inoue surface is treated in our previous paper \cite{Zhaobirinoue} where results stronger than what we need here are proved. We see from the description of Inoue surfaces in Section \ref{Inouesurfaces} that their standard realizations by complex affine Kleinian groups fit in our setting of this section. The main result of \cite{Zhaobirinoue} implies:
\begin{theorem}[\cite{Zhaobirinoue}]\label{inoue}
Let $S$ be an Inoue surface. Then up to geometric conjugation there is only one birational Kleinian group $(Y,\Gamma,U,X)$ such that $X=S$. It corresponds to the standard construction of the Inoue surface by taking the quotient of $\HH\times \CC$ by a group of affine transformations. 
\end{theorem}

\subsection{Turbulent foliations, Hopf surfaces, suspensions of $\PP^1$}
The following lemma will be repeatedly used in this section:
\begin{lemma}\label{elliptictransverse}
Let $E\subset X$ be an elliptic curve which is transverse to the foliaiton $\Fol$. Let $\iota:\CC\rightarrow U$ be the lift of $E$ to its universal cover. Let $G\subset \Gamma$ be the subgroup that preserves $\iota(\CC)$ and let $G_B$ be its induced action on $B$. We have: 
\begin{enumerate}
\item The submanifold $\iota(\CC)\subset U$ is biholomorphic to $\CC$ or $\CC^*$.
\item The projection $r(\iota(\CC))\subset B$ is biholomorphic to $\CC$,$\CC^*$ or $\PP^1$.
\item the group $G_B$ is an infinite abelian subgroup of $\PGL(2,\CC)$. 
\end{enumerate}
Assume now that $r(\iota(\CC))\neq \PP^1$. Then $r\vert_{\iota(\CC)}$ is a covering map. And we have:
\begin{enumerate}[label=(\alph*)]
\item If $G_B$ has rank one then it is a lattice in $\CC^*$ and $r(\iota(\CC))=\CC^*$. In this case $r\circ \iota$ induces an isomorphism of elliptic curves $E\rightarrow \CC^*/G_B$. 
\item If $G_B$ is a group of translations then it is a lattice in $\CC$ and $r(\iota(\CC))=\CC$. In this case $r\circ \iota$ induces an isomorphism of elliptic curves $E\rightarrow \CC/G_B$.
\item If $G_B$ is a group of multiplications of rank two then $r(\iota(\CC))=\CC^*$ and $r\circ\iota$ is an exponential map of the form $z\mapsto c\exp(\lambda z), c,\lambda\in\CC^*$.
\end{enumerate}

\end{lemma}
\begin{proof}
Since $E$ is transverse to $\Fol$, the image $\iota(\CC)$ is transverse to $r$. Therefore $r$ restricted to $\iota(\CC)$ is locally a biholomorphism. An elliptic curve cannot be mapped locally biholomorphically onto an open subset of $\PP^1$. Thus as a covering space of $E$, the submanifold $\iota(\CC)$ is biholomorphic to $\CC$ or $\CC^*$. 

We denote the composition $r\circ \iota$ by $h$ and consider it as a meromorphic function on $\CC$. By Little Picard Theorem it omits at most two values in $\PP^1$. Therefore $h(\CC)$ is biholomorphic to $\CC,\CC^*$ or $\PP^1$.

Assume by way of contradiction that $G_B$ is finite. Up to replacing $G$ with the kernel of $G\rightarrow G_B$ we can assume that $G_B$ is trivial. Then $r$ induces a holomorphic surjection, which is locally biholomorphic, from the elliptic curve $E=\iota(\CC)/G$ to $h(\CC)$. This is impossible.

Assume from now that $r(\iota(\CC))\neq \PP^1$. We identify $\pi_1(E)$ with a lattice in $\CC$. Denote by $\rho$ the composition $\pi_1(E)\rightarrow G\rightarrow G_B$. Up to replacing $G$ with a subgroup of finite index, we can assume that $G_B$ is free abelian of rank one or two. Let $K\subset \pi_1(E)$ be the kernel of $\rho$. It is trivial of cyclic. For a suitable affine coordinate $x$ on $B=\PP^1$, either all elements of $G_B$ are translations, or all of them have the form $x\mapsto \lambda x,\lambda\in \CC^*$. As $r(\iota(\CC))\neq \PP^1$ we will think of $h=r\circ \iota$ as a $\rho$-equivariant holomorphic map on $\CC$. 

Assume by way of contradiction that $G_B$ is generated by $x\mapsto bx$ with $\vert b\vert=1$. Then the kernel $K$ of $\rho$ is cyclic and a generator of $G_B=\pi_1(X)/K$ acts on $\CC/K=\CC^*$ by $z \mapsto az$ for some $a\in \CC^*$ with $\vert a \vert> 1$. Hence $h$ induces a holomorphic map $\bar{h}:\CC^*\rightarrow \CC^*$ such that $\bar{h}(az)=b\bar{h}(z)$ for any $z\in \CC^*$. Consider the closed annulus $A=\{z\in\CC, 1\leq \vert z \vert \leq \vert a \vert\}$. By the maximum modulus principle, $h(A)$ is bounded. Since the iterates of $A$ under $\pi_1(E)$ covers $\CC^*$ and since $\vert b\vert=1$, the equation $\bar{h}(az)=b\bar{h}(z)$ implies that $h$ is bounded, thus constant. This contradicts that $h$ is locally biholomorphic. 

Assume by way of contradiction that $G_B$ has rank two and is generated by two translations $f_1:x\mapsto x+b_1$ and $f_2:x\mapsto x+b_2$ such that $b_1,b_2\in \CC$ are $\RR$-colinear. Let $F_1:z\mapsto z+a_1,F_2:z\mapsto z+a_2$ be two elements of $\pi_1(E)$ such that $\rho(F_i)=f_i,i=1,2$. Then the $\rho$-equivariance of $h$ can be written as $h(z+a_i)=h(z)+b_i,i=1,2$. Let $P\subset \CC$ be the closed parallelogram corresponding to $a_1,a_2$. By the maximum modulus principle, $h(P)$ is bounded. Since the iterates of $P$ under $\pi_1(E)$ covers $\CC$ and since $b_1,b_2$ are on the same real vector line, the equations $h(z+a_i)=h(z)+b_i,i=1,2$ imply that $h$ has values in a strip region. Thus $h$ is constant by Little Picard Theorem. Contradiction.
 
Assume by way of contradiction that $G_B$ has rank one and is generated by $x\mapsto x+b$. Then $K$ is cyclic and a generator of $G_B=\pi_1(X)/K$ acts on $\CC/K=\CC^*$ by $z \mapsto az$ for some $a\in \CC^*$ with $\vert a \vert> 1$. Hence $h$ induces a holomorphic map $\bar{h}$ on $\CC^*$ such that $\bar{h}(az)=\bar{h}(z)+b$ for any $z\in \CC^*$. Considering the closed annulus $A=\{z\in\CC, 1\leq \vert z \vert \leq \vert a \vert\}$, a similar argument as in the previous two paragraphs shows that $\bar{h}$ has values in a strip region. Thus $h$ is constant by Little Picard Theorem, contradiction.

The three previous paragraphs together show that $G_B$ is a lattice in $\CC^*$, a lattice in $\CC^2$ or a group of multiplications of rank two. If $G_B$ is a lattice in $\CC$ then $h$ induces an isomorphism from $E$ to $\CC/G_B$. In particular we have $h(\CC)=\CC$.

Assume that $G_B$ is a group of multiplications generated by $g_i:x\mapsto a_ix, i=1,2$. Here we allow $a_2$ to be $1$ if $G_B$ has rank one. Assume that $f_i:z\mapsto z+b_i,i=1,2$ are two elements of $\pi_1(X)$ such that $\rho(f_i)=g_i,i=1,2$. Denote by $\Lambda$ the lattice $\langle b_1,b_2\rangle$. Then $h$ satisfies $h(z+b_i)=a_ih(z),i=1,2$ and the zeros of $h$ form a $\Lambda$-invariant set. We denote by $\mathcal{Z}\subset \CC$ the set of zeros of $h$. Up to precomposing $h$ with a translation, we can assume that $0\notin \mathcal{Z}$. Remark that $h$ has only finitely many zeros in a fundamental parallelogram of $\Lambda$. Since $\sum_{w\in \Lambda\backslash\{0\}}\frac{1}{w^3}$ is convergent while $\sum_{w\in \Lambda\backslash\{0\}}\frac{1}{w^2}$ is divergent, we obtain that $\sum_{w\in \mathcal{Z}} \frac{1}{w^3}$ is convergent while $\sum_{w\in \mathcal{Z}} \frac{1}{w^2}$ is divergent. By Weierstrass Factorization Theorem, the entire function $h$ can be written as
\[
h(z)=\exp(h^\vee(z))\prod_{w\in \mathcal{Z}}(1-\frac{z}{w})\exp(\frac{z}{w}+\frac{z^2}{2w^2})
\]
where $h^\vee$ is another entire function. Remark that
\[
\prod_{w\in \mathcal{Z}}\frac{w-z-b_i}{w-z}=1, \quad i=1,2.
\] 
because $\mathcal{Z}$ is $\Lambda$-invariant. The equations $h(z+b_i)=a_ih(z),i=1,2$ then imply that
\[
\frac{1}{a_i}\exp\big(h^\vee(z+b_i)-h^\vee(z)\big)=\prod_{w\in \mathcal{Z}}\exp\Big( \frac{b_iz}{w^2}\Big)\prod_{w\in \mathcal{Z}}\exp\Big( \frac{b_i}{w}+\frac{b_i^2}{2w^2} \Big), \quad i=1,2.
\]
The right hand side would not coverge if $\mathcal{Z}\neq \emptyset$. Therefore we conclude that $h$ has no zeros, i.e.\ $r(\iota(\CC))=\CC^*$. And we have $h(z)=\exp(h^\vee(z))$. Then there are two complex numbers $c_1,c_2$ such that $\exp(c_i)=a_i,i=1,2$ and $h^\vee(z+b_i)=h^\vee(z)+c_i,i=1,2$. The arguments in the previous paragraphs show that $\langle c_1, c_2 \rangle$ is a lattice. Then $h(z)=c\exp(\lambda z)$ with $c,\lambda\in \CC^*$ and $\exp(\lambda b_i)=a_i,i=1,2$. In particular if $G_B$ has rank one, i.e.\ if $a_2=1$, then $\vert a_1\vert\neq 1$ and $\langle a_1 \rangle $ is a lattice in $\CC^*$.

\end{proof}

\subsubsection{Examples}\label{examplesturbulent}

The Hirzebruch surface $\cH_n$ is the projectivization of the rank two vector bundle $\OO\oplus \OO(n)$ on $\PP^1$. The subbundles  $\OO$ and $\OO(n)$ determine two sections. The complement of these two sections, denoted by $W_n$, is isomorphic to the principal $\CC^*$-bundle over $\PP^1$ associated with the line bundle $\OO(n)$. The quasi-projective variety $W_0$ is a product $\CC^*\times \PP^1$. The quasi-projective variety $W_1$ is isomorphic to $\CC^2\backslash\{0\}$. For $n\geq 1$, $W_1$ is an unramified cover of degree $n$ of $W_n$; the cover is given in each fiber by $w\mapsto w^n$. 

The automorphism group of $\cH_n$ fits in the following exact sequence (cf. \cite{Maru}):
\[
1\rightarrow H_{n+1}\rightarrow \Aut(\cH_n)\rightarrow \Aut(\PP^1)\rightarrow 1
\]
where $H_{n+1}$ is the group defined in Theorem \ref{Maruyamathm}. The rational fibration $H_{n+1}\rightarrow \PP^1$ can be trivialized over $\PP^1\backslash\{0\}$ and over $\PP^1\backslash\{\infty\}$. The change of coordinates is $(x,y)\mapsto (\frac{1}{x},\frac{y}{x^n})$ where $x$ is the affine coordinate on the base of the fibration while $y$ is an affine coordinate on the fibers. Using coordinates we have
\begin{align*}
\Aut(\cH_n)=&\Bigg\{(x,y)\mapsto \Big(\frac{ax+b}{cx+d},\frac{y+t_0+t_1x+\cdots+t_nx^n}{(cx+d)^n}\Big)\mid
\\
&\begin{pmatrix}
	a&b\\c&d\end{pmatrix}\in \operatorname{GL}(2,\CC), t_0,\cdots,t_n\in\CC\Bigg\}.
\end{align*}
The action around the fiber at $x=\infty$ can be computed by the change of coordinates. In particular for the automorphism $(x,y)\mapsto (a x, b y)$ with $a,b\in\CC^*$, the action on the fiber at infinity is multiplication by $\frac{b}{a^n}$; for the automorphism $(x,y)\mapsto (x+c, b y)$ with $a,b\in\CC^*$, the action on the fiber at infinity is multiplication by $b$.

\begin{example}\label{hopfruledex}
Take $a,b\in\CC^*$ such that $0<\vert b \vert <1$ and $\vert b \vert<\vert a\vert ^n$. The automorphism $f_{n,a,b}:(x,y)\mapsto (a x, b y)$ of $\cH_n$ preserves $W_n$ and acts freely properly discontinuously and cocompactly on $W_n$. The action of $f_{1,a,b}$ on $W_1$ is conjugate to the linear contraction $(z,w)\mapsto (bz,\frac{b}{a} w)$ on $\CC^2\backslash\{0\}$. Let $X$ be the quotient surface $=W_n/<f_{n,a,b}>$ and $\Fol$ be the foliation on $X$ induced by the $\CC^*$-bundle structure of $W_n$. 

Assume that $a$ is not a root of unity. When $n=0$, the quotient is a ruled surface over the elliptic curve $\CC^*/<b>$; the foliation $\Fol$ is a suspension of $\PP^1$. When $n>0$ the surface $X$ is a Hopf surface with two elliptic curves isomorphic to $\CC^*/<b>$ and $\CC^*/<\frac{b}{a^n}>$; the leaves of $\Fol$ come from vector lines on the universal cover $\CC^2\backslash\{0\}$. If $X$ is equipped with an elliptic fibration then $\Fol$ is a turbulent foliation with respect to that elliptic fibration. 

We can modify the above examples by adapting the construction in Section \ref{blowupellipticconstruction}. We blow up $\cH_n$ at points in the fibers over $x=0$ or $x=\infty$ to obtain a new surface $Y$; the points we blow up are intersection points of irreducible components with other irreducible components or with the strict transforms of the two sections of $\cH_n$. The points we blew up are fixed points of the $f_{n,a,b}$. Thus $f_{n,a,b}$ acts also by automorphism on $Y$; its action on an irreducible component of a fiber is by multiplication. Let $U\subset Y$ be a $\CC^*$-fibration over $\PP^1$ which is the complement in $Y$ of the union of strict transforms of the two sections and some components of the singular fibers, as in Section \ref{blowupellipticconstruction}. Then for suitable choices of $a,b$, the action of $f_{n,a,b}$ on $U$ is free, properly discontinuous and cocompact. The quotient is either a Hopf surface or a ruled surface over an elliptic curve. If the quotient is a ruled surface then $U$ is foliated by rational curves of self-intersection $0$; this forces $U$ to be isomorphic to the product $\PP^1\times \CC^*$. Hence the quotient is a ruled surface if and only if the birational Kleinian group is geometrically conjugate to $<f_{0,a,b}>$ on $\cH_0=\PP^1\times \PP^1$.
\end{example}

\begin{example}\label{hopfparabolicex}
Take $0<\vert b \vert <1$ and $c\in \CC^*$. The automorphism $g_{n,b,c}:(x,y)\mapsto (x+c, b y)$ of $\cH_n$ preserves $W_n$; its action on the fiber at infinity is multiplication by $b$. If $n=0$ then the quotient of $W_0$ by $g_{0,b,c}$ is a ruled surface over an elliptic curve equipped with a suspension foliation. The action of $g_{1,b,c}$ on $W_1$ is conjugate to the contraction $(z,w)\mapsto (bz+bcw,bw)$ on $\CC^2\backslash\{0\}$. For $n\geq 1$ the quotient of $W_n$ by $g_{n,b,c}$ is a Hopf surface with only one elliptic curve. It has an unramified cover which is a primary Hopf surface corresponding to the normal form $(z,w)\mapsto (bz+\gamma w,bw)$ where $\gamma$ is some number depending on $n,b,c$. Remark that this construction does not give all primary Hopf surfaces with only one elliptic curve, but only those with linear normal form. The foliation induced by the rational fibration corresponds to the foliation by lines in $\CC^2\backslash\{0\}$. As in Example \ref{hopfruledex} we can modify the above construction by blowing up the invariant fiber. 
\end{example}

\begin{example}\label{ruledsuspensionex}
Consider two automorphisms $f_i:(x,y)\mapsto (x+a_i,y+b_i), i=1,2$ of $\PP^1\times \PP^1$. Suppose that $\ZZ b_1+\ZZ b_2$ is a lattice in $\CC$ and that $a_1,a_2$ are not both zero. Then $f_1$ and $f_2$ preserve $U=\PP^1\times \CC\subset \PP^1\times \PP^1$. The data $(\PP^1\times \PP^1,\langle f_1,f_2\rangle,U)$ give a birational Kleinian group. The quotient $X=U/\langle f_1,f_2\rangle$ is a ruled surface over the elliptic curve $\CC/(\ZZ b_1+\ZZ b_2)$. The foliation $\Fol$ is a suspension. We observe that:
\begin{lemma}\label{turbulentellipticfibration}
There is a genus one fibration on $X$ with respect to which $\Fol$ is turbulent if and only if there exists $\lambda\in \CC^*$ such that $\lambda a_i=b_i, i=1,2$. 
\end{lemma}

\begin{proof}
If there exists $\lambda\in \CC^*$ such that $\lambda a_i=b_i, i=1,2$, then the curves $\{(x,\lambda x +c),x\in \CC\}\subset U, c\in \CC$ are $\langle f_1,f_2\rangle$-invariant and are fibers of the map $(x,y)\mapsto \lambda x-y$. Hence $X$ is isomorphic to $\PP^1\times \big(\CC/(\ZZ b_1+\ZZ b_2)\big)$. The quotient of $\{\infty\}\times\CC\subset U$ is an elliptic curve, both leaf of $\Fol$ and fiber of the trivial genus one fibration.

Conversely assume that $\eta:X\rightarrow \PP^1$ is a genus one fibration with respect to wich $\Fol$ is turbulent. Consider the holomorphic map $\iota:\CC\rightarrow U$ obtained by lifting a generic fiber of $\eta$. By Lemma \ref{elliptictransverse} $\iota(\CC)$ is biholomorphic to $\CC$ or $\CC^*$. Let $G$ be the an infinite subgroup of $\langle f_1,f_2\rangle$ which preserves $\iota(\CC)$. Denote by $q$ the projection $U=\PP^1\times \CC\rightarrow \CC$. The open set $q(\iota(\CC))$ is invariant under a subgroup $G_q\subset \langle b_1,b_2 \rangle$ induced by $G$. The argument in the proof of Lemma \ref{elliptictransverse} shows that $G_q$ is a lattice and that $q(\iota(\CC))=\CC$. In particular we obtain that $G\rightarrow G_q$ is an isomorphism and that $G$ has finite index in $\langle f_1,f_2\rangle$. Also we have that $\iota(\CC)=\CC$ and $q\vert_{\iota(\CC)}$ is a biholomorphism. Thus we can think of $\iota(\CC)$ as the graph of a meromorphic function $h$, i.e.\ $\iota(\CC)=\{(h(y),y), y\in \CC\}$. As $\{\infty\}\times\CC$ is $G$-invariant, the function $h$ is actually holomorphic. Assume from now that $G$ is generated by $f_i:(x,y)\mapsto (x+a_i,y+b_i), i=3,4$. The fact that $\iota(\CC)$ is $G$-invariant can be translated into $h(y+b_i)=h(y)+a_i,i=3,4$. As in the proof of Lemma \ref{elliptictransverse} we obtain that $\langle a_3,a_4 \rangle$ is a lattice. Then $h$ induces an isomorphism $h^\vee$ between the two elliptic curves $\CC/(\ZZ b_3+\ZZ b_4)$ and $\CC/(\ZZ a_3+\ZZ a_4)$ and has the form $x\mapsto \lambda x+c$ for some $\lambda\in \CC^*$ and $c\in \CC$. In particular we have $\lambda a_i=b_i, i=3,4$. Writing $f_3=m_1f_1+m_2f_2$ and $f_4=n_1f_1+n_2f_2$ with $m_i,n_i\in \ZZ$, we obtain 
\[
\begin{pmatrix}
a_3&b_3\\a_4&b_4
\end{pmatrix} =
\begin{pmatrix}
m_1&m_2\\n_1&n_2
\end{pmatrix}
\begin{pmatrix}
a_1&b_1\\a_2&b_2
\end{pmatrix}.
\]
Since $\langle f_3,f_4\rangle$ is isomorphic to $\ZZ^2$, the matrix $\begin{pmatrix} m_1 & m_2 \\n_1 & n_2\end{pmatrix}$ is invertible. We conclude that $(a_1,a_2)$ and $(b_1,b_2)$ are colinear. 

\end{proof}

\end{example}

\begin{example}\label{ruledsuspensionextwo}

Consider two automorphisms $g_1,g_2$ of $\PP^1\times \PP^1$ defined by $g_1:(x,y)\mapsto (a_1x,y+b_1)$ and $g_2:(x,y)\mapsto (a_2x,y+b_2)$ where $a_1,a_2$ are not both equal to $1$. Suppose that $\ZZ b_1+\ZZ b_2$ is a lattice in $\CC$. Then $g_1$ and $g_2$ preserve $U=\PP^1\times \CC\subset \PP^1\times \PP^1$. The data $(\PP^1\times \PP^1,<g_1,g_2>,U)$ give a birational Kleinian group. The quotient $X=U/\langle g_1,g_2\rangle$ is a ruled surface over the elliptic curve $\CC/(\ZZ b_1+\ZZ b_2)$. The foliation $\Fol$ is a suspension. We observe that:
\begin{lemma}\label{turbulentellipticfibrationtwo}
There is a genus one fibration on $X$ with respect to which $\Fol$ is turbulent if and only if there exists $\lambda\in \CC^*$ such that $\exp(\lambda b_i)=a_i,i=1,2$. 
\end{lemma}
\begin{proof}
Assume that $\exp(\lambda b_i)=a_i,i=1,2$ for some $\lambda\in \CC^*$. Then the curves \\ $\{\left(c \exp(\lambda y),y\right),y\in \CC\}\subset U, c\in \CC^*$ are $G$-invariant and are fibers of the map $(x,y)\mapsto \exp(\lambda y)/x$. Hence $X$ is isomorphic to $\PP^1\times\big( \CC/(\ZZ b_1+\ZZ b_2)\big)$. The quotients of $\{\infty\}\times\CC$ and $\{0\}\times\CC$ are both leaves of $\Fol$ and fibers of the trivial genus one fibration.

Conversely assume that $\eta:X\rightarrow \PP^1$ is a genus one fibration with respect to which $\Fol$ is turbulent. Note that the quotients of both $\{\infty\}\times\CC$ and $\{0\}\times\CC$ are necessarily fibers of $\eta$ because otherwise they would be transverse to $\eta$ and be coverings of $\PP^1$. A generic fiber of $\eta$ gives rise to a holomorphic map $\iota:\CC\rightarrow U$. Denote by $G$ the subgroup of $\langle g_1,g_2 \rangle$ that preserves $\iota(\CC)$. The same arguments as in the proof of Lemma \ref{turbulentellipticfibration} show that $\iota(\CC)=\{(h(x),x), x\in \CC\}$ is the graph of a holomorphic function with values in $\CC^*$ and that $G$ is isomorphic to $\ZZ^2$. Suppose that $G$ is generated by $g_i:(x,y)\mapsto (a_ix,y+b_i), i=3,4$. The same arguments as in the proof (last paragraph) of Lemma \ref{elliptictransverse} show that there are two complex numbers $c_3,c_4$ such that $\exp(c_i)=a_i,i=3,4$ and that $\langle c_3, c_4 \rangle$ is a lattice. Moreover $h(y)=c\exp(\lambda y)$ with $c,\lambda\in \CC^*$ and $\exp(\lambda b_i)=a_i,i=3,4$. Writing $f_3=m_1f_1+m_2f_2$ and $f_4=n_1f_1+n_2f_2$ with $m_i,n_i\in \ZZ$, we obtain, similarly as in the proof of Lemma \ref{turbulentellipticfibration}, that there exists $\lambda'\in \CC^*$ such that $\exp(\lambda'b_i)=a_i,i=1,2$.
\end{proof}

\end{example}

\begin{remark}\label{rmkClaudon}
There are no analogues of Example \ref{ruledsuspensionex} for $\cH_n$ with $n\geq 1$. For $n\geq 1$ the action of $(x,y)\mapsto (x+a,y+b)$ on the fiber at infinity is trivial and $(x,y)\mapsto (ax,y+b)$ acts as multiplication by $\frac{1}{a^n}$ on the fiber at infinity.

We can see in another way that there are no such analogues when $n\geq 1$. Let $U_n$ be the complement of the minimal section in $\cH_n$; it is a fibration over $\PP^1$ with fibers biholomorphic to $\CC$. If the quotient of $U_n$ by a $\ZZ^2$-action was a compact surface then this compact surface would be \Kah{} because it has even second Betti number (cf. Theorem IV.3.1 \cite{BHPV}). However $U_n$ is not the universal cover of any compact \Kah{} surface: if the universal cover of a compact \Kah{} surface is non-compact and quasi-projective then it is $\CC^2$ or $\CC\times\PP^1$ (cf. \cite{Claudonnote}, \cite{CHK13}).
\end{remark}

\subsubsection{Characterizations}

\begin{lemma}\label{r(B)oftorus}
Suppose that $\Fol$ is a linear foliation on a torus. Then $r(U)$ is biholomorphic to $\CC$ or $\CC^*$. 
\end{lemma}
\begin{proof}
 
In the universal cover $\CC^2$ of $X$ we take a vector line which is transverse to the foliation. This vector line intersects every leaf. It induces $\iota:\CC\rightarrow U$, a covering map onto its image, which is everywhere transverse to the rational fibration $r:U\rightarrow B$. The composition $r\circ \iota$ can be thought of as a non constant meromorphic function. It omits at most two values in $\PP^1$ by Little Picard Theorem. Remark that $r\circ \iota(\CC)=r(U)$ because $\iota(\CC)$ intersects every leaf. It remains to prove that $r(U)\neq \PP^1$.

A connected component of the preimage in $U$ of a leaf of $\Fol$ is contained in a fiber of $r$. Such a component is a quotient of $\CC$, thus biholomorphic to $\CC$ or $\CC^*$. In any covering space of $X$, the leaves of the foliation are isomorphic to each other. Therefore $U$ is topologically a $\CC$-bundle or $\CC^*$-bundle over $r(U)$. As $U$ is a covering space of the torus $X$, the base $r(U)$ cannot be $\PP^1$.
\end{proof}

\begin{lemma}\label{r(B)ofbidiskquotient}
Suppose that $X$ is a bidisk quotient and that $\Fol$ is one of the natural foliations. Then $U$ is biholomorphic to the bidisk and $r(U)$ is a round disk in $\PP^1$, i.e.\ is up to M\"obius transform the upper half plane.
\end{lemma}
\begin{proof}
Consider the composition $\pi_1(X)\rightarrow \Gamma\rightarrow \Gamma_B\subset \PGL(2,\CC)$. It is a representation of $\pi_1(X)$, an irreducible lattice in the rank two semisimple Lie group $\PSL(2,\RR)\times \PSL(2,\RR)$, into $\PGL(2,\CC)$. Margulis supperrigidity (\cite{Mar91}) says that such a representation, if not trivial, extends to a morphism of Lie groups $\PSL(2,\RR)\times \PSL(2,\RR)\rightarrow \PGL(2,\CC)$. This implies that, up to conjugation in $\PGL(2,\CC)$ and up to a Galois conjugate, $\Gamma_B$ is exactly the projection of the lattice $\pi_1(X)\subset \PSL(2,\RR)\times \PSL(2,\RR)$ into one of the factors $\PSL(2,\RR)$. In particular $\Gamma_B$ is a dense subgroup of $\PSL(2,\RR)$ (up to conjugation). The open set $r(U)$ is $\Gamma_B$-invariant, thus $\PSL(2,\RR)$-invariant. It has to be one of the half planes. Since the three groups $\pi_1(X), \Gamma$ and $\Gamma_B$ are isomorphic, the open set $U$ is the universal cover of $X$, thus biholomorphic to the bidisk.
\end{proof}

\begin{proposition}\label{fullbasecase}
Suppose that $r(U)=B=\PP^1$. Then $\Fol$ is a turbulent foliation, an obvious foliation on a Hopf surface or a suspension of $\PP^1$.
\end{proposition}
\begin{proof}
Inoue surfaces are excluded by Theorem \ref{inoue} because in the standard construction of Inoue surfaces $r(U)$ is the upper half plane. Linear foliations on tori are excluded by Lemma \ref{r(B)oftorus} and bidisk quotients are excluded by Lemma \ref{r(B)ofbidiskquotient}. Suspensions of elliptic curves over elliptic curves are up to finite covering also linear foliations on tori. Fibrations and suspensions of elliptic curves over hyperbolic curves are excluded by Lemma \ref{fibrationcaselemma} and Lemma \ref{r(B)ofellipticsuspension} (the proofs can be read right away and are written there for sake of notations).

By Corollary \ref{regularonacylinder} and Proposition \ref{hasafoliatedprojectivestructure} $\Fol$ has a foliated projective structure. By looking at the list of possible foliations from Theorem \ref{Brunellathm} and Corollary \ref{DGcorollary} we see that the remaining possibilities are those listed in the proposition.
\end{proof}

\begin{proposition}\label{turbulentclassification}
Suppose that $\Fol$ is a turbulent foliation. Then up to a geometric conjugation preserving the rational fibration and up to replacing $\Gamma$ with a finite index subgroup, either we have the normal form of a Hopf surface as in section \ref{Hopfnormalform} or we are in one of the cases described in Examples \ref{hopfruledex}, \ref{ruledsuspensionex} and \ref{ruledsuspensionextwo}. 
\end{proposition}
\begin{proof}
Let $f:X\rightarrow C$ be the genus one fibration subjacent to the turbulent foliation. At least one fiber of $f$ is a leaf of $\Fol$. Let $E$ be such a fiber. Denote by $\pi$ the covering map from $U$ to $X$.  We have an exact sequence $1\rightarrow H\rightarrow \pi_1(X)\rightarrow \pi_1^{orb}(C)\rightarrow 1$ where $H$ is the image of the fundamental group of a regular fiber (see \cite{GurSha85} Theorem 1). If $E_m$ is a multiple fiber of $f$ then the image of $\pi_1(E_m)$ in $\pi_1(X)$ contains $H$ as a subgroup of finite index. Let $G\subset \Gamma$ be the image of the composition $H\rightarrow \pi_1(X)\rightarrow \Gamma$. Then $G$ is a normal subgroup of $\Gamma$ because $H$ is the kernel of $\pi_1(X)\rightarrow \pi_1^{orb}(C)$ and $\pi_1(X)\rightarrow \Gamma$ is surjective. It is the image of $\pi_1(E)$ in $\Gamma$ if $E$ is a non-multiple fiber and is a finite index subgroup of it if $E$ is multiple. Let $E'$ be a fiber of $f$, let $\Omega$ be a connected component of $\pi^{-1}(E')$ and let $G'$ be the subgroup of $\Gamma$ that preserves $\Omega$. Then under the identification of $\Gamma$ as the deck transformation group and as a quotient of $\pi_1(X)$, the group $G'$ is conjugate to the image of $\pi_1(E')\rightarrow \pi_1(X)$. Since $G$ is normal, we have $G'=G$ if $E'$ is a smooth fiber and $G\subset G'$ has finite index if $E'$ is a multiple fiber. Hence any connected component of $\pi^{-1}(E')$ is $G$-invariant. Thus the intermediate Galois cover $U/G$ is a genus one fibration over a possibly non-compact base. Note that $G$ is an infinite quotient of $\ZZ^2$ because if it was finite then $E$, a leaf of $\Fol$, would lift to an elliptic curve contained in the intersection between $U$ and one fiber of the rational fibration $r$, which is impossible. 

Denote by $G_B$ the image of $G$ in $\Gamma_B$. By Lemma \ref{elliptictransverse} $G_B$ is infinite. We claim that $G$ is a subgroup of $\Gamma$ of finite index. Let $\gamma\in\Gamma$ be an arbitrary element. Let $F_E$ be a connected component of $\pi^{-1}(E)$. By the previous paragraph $F_E$ is $G$-invariant. It is a copy of $\CC$ or $\CC^*$ contained in a $G$-invariant fiber $F$ of $r$. The image $\gamma(F_E)$ is also a connected component of $\pi^{-1}(E)$ which is preserved by $G$. This implies in particular that $\gamma(F)$ is a $G$-invariant fiber. As $G_B$ is infinite, there are at most two $G$-invariant fibers of $r$. Up to replacing $\Gamma$ with a subgroup of index two, we can assume that $F$ is $\Gamma$-invariant. The possibly singular fiber $F$ has finitely many irreducible components. Excatly one of them contains $F_E$. Since $F_E$ is biholomorphic to $\CC$ or $\CC^*$, it is the only connected component of the corresponding irreducible componenet of $F$. Again up to taking a subgroup of finite index, we can assume that $F_E$ is $\Gamma$-invariant. However by the previous paragraph $G$ is a subgroup of finite index of the subgroup of $\Gamma$ that preserves $F_E$. We conclude that $\Gamma/G$ is finite. 

Up to replacing $\Gamma$ with a subgroup of finite index, we can assume that $G=\Gamma$ and that $\Gamma,\Gamma_B$ are free abelian groups of rank one or two. Note that $G=\Gamma$ implies that $\pi^{-1}(E)$ is connected and that $f$ is a genus one fiber bundle. We choose coordinates in $B=\PP^1$ so that the fiber $F$ containing $\pi^{-1}(E)$ is over $\infty\in B$. Denote by $R$ the set of fixed points of $G_B$ in $B\cap r(U)$. It contains $\infty$ and has at most two elements. We let the other element be $0$ if it exists. The complement $r(U)\backslash R$ is $\CC$ or $\CC^*$, on which $\Gamma_B$ acts respectively by translations or multiplications. Any point in $r(U)\backslash R$ has an infinite $\Gamma_B$-orbit. Thus by Corollary \ref{regularonacylinder} $r$ has no singular fibers over $r(U)\backslash R$. The fibers over points of $R$ descend to fibers of $f$ in $X$. 

Denote by $S$ the set of points in $C$ over which the fibers of $f$ are leaves of $\Fol$. It has the same cardinality as $R$. The foliation $\Fol$ restricted to $f^{-1}(C\backslash S)$ is everywhere transverse to the genus one bundle $f$. We have 
\[\pi^{-1}\left(f^{-1}(C\backslash S)\right)=r^{-1}(r(U)\backslash R)\cap U\]
and we denote this open set by $M$. 

Let $E'$ be a fiber of $f$ over $C\backslash S$. Let $\Omega$ be its preimage in $M$. It is biholomorphic to $\CC$ if $\Gamma$ has rank two and to $\CC^*$ if $\Gamma$ has rank one. It is everywhere transverse to the rational fibration $r$. The projection $r(\Omega)$ is not $\PP^1$ because it omits $R$ which contains at least $\infty$. Applying Lemma \ref{elliptictransverse}, we obtain that $\Omega\rightarrow r(U)\backslash R$ is a covering map and we are in one of the following situations:
\begin{enumerate}[label=(\Alph*)]
	\item $\Gamma$ is cyclic and $\Gamma_B$ is a lattice in $\CC^*=r(U)\backslash R$. 
	\item $\Gamma$ has rank two and $\Gamma_B$ is a lattice in $\CC=r(U)\backslash R$;
	\item $\Gamma$ has rank two while $\Gamma_B$ is a lattice in $\CC^*=r(U)\backslash R$
	\item $\Gamma$ and $\Gamma_B$ have rank two and $\Gamma_B$ is a group of multiplications in $\CC^*=r(U)\backslash R$.
\end{enumerate}

Assume that we are in case (D). By Theorem \ref{zhaocentralizer}, the group $\Gamma$ is conjugate in $\Jonq$ to a subgroup of $\Aut(\PP^1\times \PP^1)$. By Lemma \ref{birconjtoaut} and Corollary \ref{regularonacylinder}, we can assume without conjugation that $\Gamma$ is generated by $g_i:(x,y)\mapsto (a_i x, \alpha_i(y)), i=1,2$ where $a_i\in \CC^*$ and $\alpha_1,\alpha_2$ are two M\"obius transformations. Since $r^{-1}(\infty)\cap U$ is biholomorphic to $\CC$ and descends to an elliptic curve, we have, up to a change of coordinate, $\alpha_i(y)=y+b_i, i=1,2$ where $b_1,b_2$ generate a lattice in $\CC$. Thus we are in Example \ref{ruledsuspensionextwo}.

Assume that we are in case (A) or (B). Then the covering map $r\vert_\Omega$ is a biholomorphism. Thus a leaf of $\Fol$ transverse to $f$ intersects a fixed fiber of $f$ only once. Thus such a leaf is biholomorphic to $C\backslash S$ and we have $f^{-1}(C\backslash S)=\pi(M)=C\backslash S\times E'$. Recall that $\Gamma=G$ is the image of $\pi_1(E')\rightarrow \pi_1(X)\rightarrow \Gamma$. Therefore any leaf of $\Fol$ in $\pi(M)$ is bihilomorphic to its preimage in $M$. In other words, for any $x\in r(U)\backslash R$, $r^{-1}(x)\cap U$ is biholomorphic to $C\backslash S$, a Riemann surface of finite type. As $r^{-1}(x)$ is $\PP^1$, this is only possible if $C=\PP^1$. Hence $X$ is a genus one bundle over $\PP^1$. Assume that $\#S=2$. Then $\#R=2$, $r(U)=\PP^1$ and we are in Case (A). In this case $U$ is a $\CC^*$-bundle over $\PP^1$. The fact $r(U)=B$ allows us to suppose that $\Gamma\subset \Aut(Y)$ by Corollary \ref{regularonacylinder}. We are thus in Example \ref{hopfruledex}.

Assume that $\#S=\#R=1$ and that we are in Case (A). Then $C\backslash S$ is biholomorphic to $\CC$ and $r(U)=\PP^1\backslash\{0\}$. The fiber $r^{-1}(\infty)\cap U$ is biholomorphic to $\CC^*$ while $M=\CC^*\times \CC$. Therefore $U$ is biholomorphic to $\CC^2\backslash\{0\}$ and $X$ is a Hopf surface. A generator of $\Gamma$ has the form $(x,y)\mapsto (ax,by)$ because it acts by multiplication on $r^{-1}(\infty)\cap U$ and preserves both the horizontal and the vertical foliations on $M$. Therefore this situation corresponds to the normal form of a Hopf surface as in Section \ref{Hopfnormalform}.

Assume that $\#S=\#R=1$ and that we are in Case (B). Then $C\backslash S$ is biholomorphic to $\CC$ and $r(U)=\PP^1$.  We can suppose that $\Gamma\subset \Aut(Y)$ by Corollary \ref{regularonacylinder}. The fiber $r^{-1}(\infty)\cap U$ is biholomorphic to $\CC$ while $M=\CC\times \CC$. Therefore $U$ is a $\CC$-bundle over $\PP^1$, the complement of a section in a ruled surface. Then Remark \ref{rmkClaudon} allows us to conclude that $U=\PP^1\times \CC$ and that we are in Example \ref{ruledsuspensionex}.

Finally let us consider Case (C). In this case $\Omega$ and $r^{-1}(\infty)\cap U$ are biholomorphic to $\CC$. Since $\Gamma=G$ is the image of $\pi_1(E')\rightarrow \pi_1(X)\rightarrow \Gamma$, the set $M$ has a structure of $\CC$-bundle over $C\backslash S$. Denote by $\bar{f}:M\rightarrow C\backslash S$ the fibration map. Consider the map $\eta:M\rightarrow (C\backslash S)\times \CC^*$ defined by $\eta(z)=(\bar{f}(z),r(z))$ where we identified $r(U)\backslash R$ with $\CC^*$. The preimage $\Omega$ of any fiber of $f$ is a fiber of $\bar{f}$. The projection $r:\Omega\rightarrow \CC^*$ is an exponential map by Lemma \ref{elliptictransverse}, i.e.\ an infinite cyclic covering. This implies that $\eta$ is a cyclic covering map. In particular the $\CC$-bundle structure of $M$ is the pull back of the $\CC^*$-bundle structure on $(C\backslash S)\times \CC^*$, thus trivial. In other words $M$ is biholomorphic to $(C\backslash S)\times \CC$. Let $x\in \CC^*$. We have $r^{-1}(x)\cap M=\eta^{-1}\left ((C\backslash S)\times \{x\}\right )$. Thus $r^{-1}(x)\cap M$  
is an infinite cyclic covering (possibly disconnected) of $C\backslash S$. However a punctured hyperbolic curve does not admit cyclic coverings contained in $\PP^1$ (see \cite{Mas65}), so $C$ is $\PP^1$ or an elliptic curve. 

Suppose by way of contradiction that $C$ is an elliptic curve. Let $\Gamma_r$ be the kernel of the projection $\Gamma\rightarrow \Gamma_B$ with generator $\gamma_1$. For $x\in \CC^*$, as $r^{-1}(x)\cap M$ is a cyclic covering of $(C\backslash S)$ where $\Gamma_r$ acts by deck transformations, $r^{-1}(x)\cap M$ is contained in a copy of $\CC^*$ where $\gamma_1$ acts by $y\mapsto ay$ with $a\in \CC^*$. Let $\gamma_2$ be an element of $\Gamma$ whose image in $\Gamma_B$ generates $\Gamma_B$. Then $\Gamma=\ZZ^2$ is generated by $\gamma_1$ and $\gamma_2$. By Theorem \ref{zhaocentralizer} (the information we use here is due to \cite{CD12}), up to conjugation $\gamma_1, \gamma_2$ can be written respectively as $(x,y)\mapsto (x,ay)$ and $(x,y)\mapsto(bx,R(x)y)$ with $a,b\in\CC^*$ and $R\in\CC(x)^*$. By Lemma \ref{birconjtoaut} and Corollary \ref{regularonacylinder}, we can assume these formulas without conjugation at least over $\CC^*=r(U)\backslash R$. Recall that $r^{-1}(\infty)\cap U$ has a component biholomorphic to $\CC$ on which $\Gamma$ acts as a lattice. This gives a contradiction because $\gamma_1$ acts on any irreducible component of $F$ by multiplication. Therefore we obtain that $C=\PP^1$. 

Since $C\backslash S=\PP^1\backslash S$ has an infinite cyclic covering, we have necessarily $\#S=\#R=2$, $C\backslash S\equiv\CC^*$ and that $r^{-1}(x)\cap U$ is biholomorphic to $\CC$ for any $x\notin R$. We have $R=\{0,\infty\}$ and $r(U)=B$. Then $U$ is a $\CC$-bundle over $\PP^1$. The fact that $r(U)=B$ implies that $\Gamma\subset \Aut(Y)$ by Corollary \ref{regularonacylinder}. And since $U$ is the complement of a section of the Hirzebruch surface $Y$, we obtain by Remark \ref{rmkClaudon} that $U=\PP^1\times \CC$ and $Y=\PP^1\times\PP^1$. Hence we are in Example \ref{ruledsuspensionextwo}.
\end{proof}

\begin{proposition}\label{Hopfclassification}
Suppose that $\Fol$ is an obvious foliation on a non-elliptic Hopf surface. Then up to a geometric conjugation preserving the rational fibration and up to replacing $\Gamma$ with a finite index subgroup we are in one of the following situations:
\begin{enumerate}
	\item $X$ contains two elliptic curves and we are in Example \ref{hopfruledex}.
	\item $X$ contains one elliptic curve and we are in Example \ref{hopfparabolicex}.
	\item $U$ is the standard open set $\CC^2\backslash\{0\}$ of $Y=\PP^1\times \PP^1$ and $\Gamma$ is generated by the normal form $(x,y)\mapsto (\alpha x+\gamma y^m, \beta y)$.
\end{enumerate}
\end{proposition}
\begin{proof}
Since the fundamental group of a Hopf surface is virtually cyclic, the open set $U$ is a finite quotient of $\CC^2\backslash\{0\}$. There always exists an immersion $\CC\rightarrow U$ which is transverse to the foliation and intersects each leaf at most once. Consequently $r(U)$ is $\PP^1$ or $\CC$.  

Assume that $r(U)=\PP^1$. Then $U$ is a $\CC^*$-bundle over $r(U)$. By Corollary \ref{regularonacylinder} we have $\Gamma\subset \Aut(Y)$. The complement of $U$ in $Y$ is $\Gamma$-invariant, so it is necessarily the union of two sections of $r$. Hence we are in Example \ref{hopfruledex} or \ref{hopfparabolicex}. 

Assume that $r(U)=\CC$. Then the intersection of $U$ with a fiber of $r$ is $\CC$ except for one which is $\CC^*$. By Corollary \ref{regularonacylinder} $\Gamma$ acts by automorphism on $\CC\times \PP^1$. The complement of $U$ in $\CC\times \PP^1$ is $\Gamma$-invariant, so it is necessarily the union of a point with an algebraic curve isomorphic to $\CC$. Consequently $U=\CC^2\backslash\{0\}\subset \CC^2\subset \CC\times \PP^1$ and $r$ is the projection onto the first coordinate. As $\Gamma$ preserves $U$, a generator of $\Gamma$ has the form $(x,y)\mapsto (ax,B(x)y+R(x))$ where $a\in \CC^*$ and $B,R\in \CC(x)$. As $\Gamma$ acts regularly on $U$, the rational function $B$ has no poles nor zeros while $R$ has no poles. This means that $B$ is a non zero constant and $R$ is a polynomial. Using conjugations of the form $(x,y)\mapsto (x,y+x^n),n\in \NN$, we can reduce the situation to the normal form of a Hopf surface as in Section \ref{Hopfnormalform}.
\end{proof}

\begin{proposition}\label{suspensionrationalcase}
Suppose that $\Fol$ is a suspension of $\PP^1$. Up to a geometric conjugation preserving the rational fibration and up to replacing $\Gamma$ with a finite index subgroup, we have $Y=\PP^1\times \PP^1$ and $\Gamma\subset \PGL(2,\CC)\times \PGL(2,\CC)$. The projection of $\Gamma$ onto the second factor is a classical Kleinian group; we have $U=\PP^1\times D$ where $D\subset \PP^1$ is an invariant component of that classical Kleinian group.
\end{proposition}
\begin{proof}
As $\Fol$ is suspension of $\PP^1$, any covering space of $X$ has a $\PP^1$-bundle structure. In particular $U$ is covered by rational curves with auto-intersection zero which are transverse to $r$. This implies that up to geometric conjugation $Y= \PP^1\times \PP^1$, that the first projection is $r$ and the rational curves in $U$ are fibers of the second projection. The conclusion follows.
\end{proof}

\subsection{Infinite suspensions of elliptic curves}

We denote the covering map $U\rightarrow X$ by $\pi$. We now consider the case where there exists a compact Riemann surface $M$ and an elliptic curve $N$ such that the foliation $\Fol$ on $X$ is an infinite suspension of $N$ over $M$. Recall that $X$ is \Kah{} when $\Fol$ is a suspension. If $M$ is an elliptic curve then by Enriques-Kodaira's classification the \Kah{} surface $X$ is a finite quotient of a torus and we can, by taking a finite index subgroup of $\Gamma$, reduce to the case of linear foliations on tori which will be studied later. We assume in this section that the Riemann surface $M$ is hyperbolic.   

Recall the construction of a suspension. There exists an infinite Galois covering $\bar{M}$ of $M$ with deck transformation group $G$ and a representation $\alpha:G\rightarrow \Aut(N)$ with infinite image such that $X$ is the quotient of $\bar{M}\times N$ by the action of $G$ defined by $g\cdot (m,n)=(g\cdot m, \alpha(g)\cdot n)$. 
We have a fibration $f:X\rightarrow C$ whose fibers are all isomorphic to $N$. The fibration $f$ is everywhere transverse to $\Fol$. Up to replacing $\bar{M}$ with a quotient, we can assume that $\alpha$ is injective. Denote by $\tilde{M}$ the universal cover of $M$. The universal cover of $X$ is $\tilde{M}\times \CC$. The fundamental group of $\bar{M}\times N$ is $\pi_1(\bar{M})\times \pi_1(N)$ and the fundamental group of $X$ fits in the exact sequence
\begin{equation} \label{exactsequencesuspension}
1\rightarrow \pi_1(\bar{M})\times \pi_1(N)\rightarrow \pi_1(X)\rightarrow G  \rightarrow 1.
\end{equation}

We will think of $\pi_1(\bar{M}),\pi_1(N)$ as subgroups of $\pi_1(X)$. Note that the birational Kleinian group $\Gamma$ can be identified with a quotient of $\pi_1(X)$. Thus up to replacing $\Gamma, \pi_1(X)$ with subgroups of finite index, we can assume that $G$ is an abelian group, i.e.\ $\alpha$ sends $G$ injectively to a group of translations on $N$.
 
\begin{lemma}\label{centerGammaone}
The subgroup $\pi_1(N)$ is contained in the center of $\pi_1(X)$.
\end{lemma}
\begin{proof}
The action of an element of $\pi_1(X)$ on the universal cover $\tilde{M}\times \CC$ has the form $(x,y)\mapsto (\beta(x),y+b)$ with $\beta \in \pi_1(M)$ and $b\in \CC$ whereas an element of $\pi_1(N)$ has the form $(x,y)\mapsto (x,y+c), c\in \CC$. 
\end{proof}

Denote respectively by $\Gamma_1$ and $\Gamma_2$ the images of $\pi_1(N)$ and $\pi_1(\bar{M})$ in $\Gamma$. As images of normal subgroups under a surjective homomorphism, $\Gamma_1$ and $\Gamma_2$ are normal subgroups of $\Gamma$. Up to replacing $\Gamma$ with a subgroup of finite index, we can assume that $\Gamma_1$ is a free abelian group of rank one or two. The center of $\Gamma$ contains $\Gamma_1$ by Lemma \ref{centerGammaone}. 
Denote by $\Gamma_3$ the quotient $\Gamma/(\Gamma_1\Gamma_2)$. Then \eqref{exactsequencesuspension} implies that $\Gamma_3$ is isomorphic to a quotient of $G$, thus is abelian. For the subgroups $\Gamma_i\subset \Gamma, i=1,2$ we denote by $\Gamma_{iB}$ the induced subgroup of $\PGL(2,\CC)$. By Lemma \ref{elliptictransverse}, $\Gamma_{1B}$ is infinite. By construction of the suspension, $\Gamma_2$ preserves each leaf of $\pi^*\Fol$ in $U$. In particular $\Gamma_2$ preserves each fiber of the rational fibration $r$. We summarize the above assertions as follows:
\begin{lemma}\label{suspensionGammai}
The center of $\Gamma$ contains $\Gamma_1$. The subgroup $\Gamma_2$ is normal and satisfies $\Gamma_{2B}=\{Id\}$ whereas $\Gamma_{1B}$ is infinite. The quotient $\Gamma_3=\Gamma/(\Gamma_1\Gamma_2)$ is an abelian group.
\end{lemma}

\begin{lemma}\label{suspensionGammaiinfinite}
At least one of $\Gamma_2,\Gamma_3$ is infinite. In other words $\Gamma_1$ is not a subgroup of finite index of $\Gamma$.
\end{lemma}
\begin{proof}
 Assume that $\Gamma_1$ has finite index in $\Gamma$. Then $U$ would be foliated by finite coverings of $M$. These coverings of $M$ induce a pencil of hyperbolic curves on $Y$ preserved by $\Gamma$. This is not possible for an infinite subgroup of $\Bir(Y)$ (see Theorem \ref{strongTitsalternative}). 
\end{proof}

\begin{corollary}\label{existenceofZ2suspension}
There is a subgroup of $\Gamma$ containing $\Gamma_1$ and isomorphic to $\ZZ^n$ for some $n\geq 2$.
\end{corollary}
\begin{proof}
Let $\delta\in \Gamma$ be element of infinite order of $\Gamma_2$ or the preimage of an element of infinite order of $\Gamma_3$. Then $\delta$ and $\Gamma_1$ generate an infinite abelian subgroup of rank at least two.  
\end{proof}

By Lemma \ref{elliptictransverse}, the projection $r(U)$ is $\CC,\CC^*$ or $\PP^1$. We can actually rule out $\PP^1$:
\begin{lemma}\label{r(B)ofellipticsuspension}
We have $r(U)\neq \PP^1$.
\end{lemma}
\begin{proof}
Assume by way of contradiction that $r(U)=\PP^1$. By Corollary \ref{regularonacylinder} we have $\Gamma\subset \Aut(Y)$. Therefore all elements of $\Gamma$ are elliptic birational transformations. Let $\Gamma_4$ be a subgroup isomorphic to $\ZZ^n,n\geq 2$ as in Corollary \ref{existenceofZ2suspension}. Note that $\Gamma_{4B}$ is infinite because it contains $\Gamma_{1B}$. We can conjugate $\Gamma_{4B}$ in $\Jonq$ to a group of one of the first four forms in Theorem \ref{zhaocentralizer}. Then we apply Theorem \ref{bdcentralizer} to $\Gamma$. Note that when we use the descriptions of centralizers in Theorem \ref{bdcentralizer}, we take only those elements preserving the rational fibration $r$. We then deduce that the whole group $\Gamma$ is conjugate in $\Jonq$ to a subgroup of $\PGL(2,\CC)\times \PGL(2,\CC)=\Aut^0(\PP^1\times \PP^1)$. A priori the conjugation may not be geometric. 

Assume by way of contradiction that all elements of $\Gamma_1$, after the above conjugation, have the form $(x,y)\mapsto (A(x),y)$ where $A(x)=x+1$ or $A(x)=ax, a\in \CC^*$. Then the fibers of the genus one fibration $f:X\rightarrow C$ come from the horizontal foliation on $U$ induced by a pencil of raitonal curves on $Y$ transverse to $r$. But then leaves of this horizontal foliation on $U$, which are either all biholomorphic to $\CC$ or all biholomorphic to $\CC^*$,  project via $r$ onto the same open subset $r(U)$. In particular $r(U)\neq \PP^1$.

Therefore there exists an element of $\Gamma_1$ which after conjugation has the form $(x,y)\mapsto (A(x),B(y))$ where both $A,B\in \PGL(2,\CC)$ have infinite order. Then Theorem \ref{bdcentralizer} implies that the whole group $\Gamma$ is conjugate to a group having one of the first four forms of Theorem \ref{zhaocentralizer}. In particular $\Gamma$ is abelian.

If the rank of $\Gamma$ was $\geq 3$, then the action of $\Gamma$ would not be discrete on the invariant fiber $r^{-1}(\infty)$ (or $r^{-1}(0)$). Therefore $\Gamma$ has rank two and its action on an invariant fiber can be identified with a lattice in $\CC$. This means that 1) $\Gamma$ is generated by two elements $g_i:(A_i(x),y+b_i),i=1,2$ where $A_i \in \PGL(2,\CC)$ and $\langle b_1,b_2\rangle$ is a lattice in $\CC$; 2) $U=\PP^1\times \CC\subset Y=\PP^1\times \PP^1$. We are thus in the situation of Example \ref{ruledsuspensionex} or \ref{ruledsuspensionextwo} and $X$ is a ruled surface. Contradiction to the assumption that $M$ is hyperbolic.
\end{proof}

\begin{lemma}\label{suspensionnotvirtuallycyclic}
The group $\Gamma_{B}$ is not virtually cyclic.
\end{lemma}
\begin{proof}
Suppose by way of contradiction that $\Gamma_B$ is virtually cyclic. We know by Lemma \ref{suspensionGammai} that $\Gamma_{1B}$ is infinite. Thus $\Gamma_{1B}$ is a subgroup of finite index of $\Gamma_B$. Up to taking a subgroup of finite index we can assume that $\Gamma_B=\Gamma_{1B}$. By Lemma \ref{elliptictransverse} and Lemma \ref{r(B)ofellipticsuspension} $\Gamma_{1B}$ is a lattice in $\CC^*=r(U)$. Therefore $r(U)/\Gamma_B$ is an elliptic curve and $r$ induces a surjective holomorphic map from $X$ to $r(U)/\Gamma_B$ whose fibers are (unions of) leaves of $\Fol$. This contradicts that $\Fol$ is an infinite suspension.
\end{proof}

Lemma \ref{suspensionnotvirtuallycyclic} allows us to apply Theorem \ref{zhaocentralizer}:
\begin{lemma}\label{suspensioncentralizerlemma}
Up to geometric conjugation realized by elementary transformations, we have that $Y=\PP^1\times \PP^1$ (where $r$ is the projection onto the first factor) and that $\Gamma_1$ is a subgroup of one of the following four groups: $\{(x,y)\mapsto (ax,by)\vert a,b\in \CC^*\}, \{(x,y)\mapsto (ax,y+b)\vert a\in\CC^*,b\in\CC\}, \{(x,y)\mapsto (x+a,by)\vert a\in\CC,b\in\CC^*\}$ or $\{(x,y)\mapsto (x+a,y+b)\vert a,b\in\CC\}$. In the first two cases $r(U)=\CC^*$; in the last two cases $r(U)=\CC$.
\end{lemma}
\begin{proof}
Recall that $\Gamma_1$ is contained in the center of $\Gamma$. The image in $\Aut(B)$ of the centralizer of a \Jonqui twist is virtually cyclic by Theorem \ref{zhaocentralizer1}. Thus by Lemma \ref{suspensionnotvirtuallycyclic} we deduce that $\Gamma_1$ contains no \Jonqui twists and is an elliptic free abelian group. Then there exists a birational map $\phi:Y\dashrightarrow \PP^1\times \PP^1$, composition of elementary transformations, such that $\phi\Gamma_1\phi^{-1}$ has one of the first four forms in Theorem \ref{zhaocentralizer}. It remains to show that the conjugation $\phi$ can be chosen to be a geometric conjugation. The reason is as follows. Up to a coordinate change on $\PP^1=B$ we observe that the elementary transformations in $\phi$ can be made over any point in $B$. As $r(U)$ is a proper subset of $B$ by Lemma \ref{r(B)ofellipticsuspension}, we can do elementary transformations on a fiber over a point outside $r(U)$. Such an elementary transformation has no effects on $U$ or on the dynamics of $\Gamma$ on $U$. In other words elementary transformations outside $U$ realize a geometric conjugation of $\Gamma$. 
\end{proof}

Now we study Cases (a), (b), (c) of Lemma \ref{elliptictransverse} one by one. 
\begin{lemma}\label{casecsuspension}
The group $\Gamma_{1B}$ is not a group of multiplications of rank two.
\end{lemma}
\begin{proof}
Assume by way of contradiction that $\Gamma_{1B}$ is a group of multiplications of rank two. We have $r(U)=\CC^*$ by Lemma \ref{elliptictransverse}. By Lemma \ref{suspensioncentralizerlemma} $\Gamma_{1}$ is a group generated by two elements $g_i:(x,y)\mapsto (a_ix,A_i(y)),i=1,2$ where $a_i\in \CC^*$ and $A_i\in \PGL(2,\CC)$. At least one of the $A_i$ has infinite order because otherwise the action of $\Gamma_{1}$ would be nowhere discrete. Then by Theorem \ref{bdcentralizer} we deduce that $\Gamma$ is a subgroup of $\{(x,y)\mapsto (ax,y+b), a\in \CC^*, b\in \CC\}$ or $\{(x,y)\mapsto (ax,by), a\in \CC^*, b\in \CC^*\}$. Observing that the action is not discrete in any neiborhood of the invariant fiber $r^{-1}(\infty)$ (or $r^{-1}(0)$), we obtain that $U$ is necessarily the Zariski open set $\CC^*\times \CC$ or $\CC^*\times \CC^*$. This implies that $X$ is a torus, contradiction to our hypothesis that $M$ is hyperbolic.
\end{proof}

\begin{lemma}\label{casebsuspension}
If $\Gamma_{1B}$ is a lattice in $\CC$, then $U=\CC\times \Omega$ where $\Omega$ is an invariant component of a classical Kleinian group and $\Gamma\subset\PGL(2,\CC)\times \PGL(2,\CC)$. The genus one fibration on $X$ is induced by the second projection $\PP^1\times \PP^1\rightarrow \PP^1$.
\end{lemma}
\begin{proof}
Suppose that $\Gamma_{1B}$ is a lattice in $\CC$. We have $r(U)=\CC$ by Lemma \ref{elliptictransverse}. By Lemma \ref{suspensioncentralizerlemma} $\Gamma_{1}$ is a group generated by two elements $g_i:(x,y)\mapsto (a_i+x,A_i(y)),i=1,2$ where $a_i\in \CC^*$ and $A_i\in \PGL(2,\CC)$. 

Assume by way of contradiction that one of the $A_i$ has infinite order. Then by Theorem \ref{bdcentralizer} we deduce that $\Gamma$ is a subgroup of $\{(x,y)\mapsto (x+a,y+b), a\in \CC, b\in \CC\}$ or $\{(x,y)\mapsto (x+a,by), a\in \CC, b\in \CC^*\}$. Then similar arguments as in the proof of Lemma \ref{casecsuspension} rule this situation out.

Assume that $A_i=Id,i=1,2$. Then the the genus one fibration on $X$ is simply induced by the second projection $\PP^1\times \PP^1\rightarrow \PP^1$. This implies that $U$ is a product $\CC\times \Omega$ where $\Omega$ is an open subset of $\PP^1$. By Theorem \ref{bdcentralizer}, any element of $\Gamma$ has the form $(x,y)\mapsto (x+R(y),B(y)), R\in \CC(y), B\in \PGL(2,\CC)$. Since $\Gamma$ preserves the rational fibration $r:(x,y)\mapsto x$, the rational function $R$ in the above formula is always constant. In other words $\Gamma$ is a subgroup of $\PGL(2,\CC)\times \PGL(2,\CC)$. Then $\Gamma$ induces via the second projection a classical Kleinian group that preserves $\Omega$. 
 
\end{proof}

\begin{lemma}\label{caseasuspension}
If $\Gamma_{1B}$ is a lattice in $\CC^*$, then $U=\CC^*\times \Omega$ where $\Omega$ is an invariant component of a classical Kleinian group and $\Gamma\subset\PGL(2,\CC)\times \PGL(2,\CC)$. The genus one fibration on $X$ is induced by the second projection $\PP^1\times \PP^1\rightarrow \PP^1$.
\end{lemma}
\begin{proof}
Suppose that $\Gamma_{1B}$ is a lattice in $\CC^*$. We have $r(U)=\CC^*$ by Lemma \ref{elliptictransverse}. By Lemma \ref{suspensioncentralizerlemma} $\Gamma_{1}$ is a subgroup of $\{(x,y)\mapsto (ax,A(y))\vert a\in \CC^*, A\in\PGL(2,\CC)\}$. 

Let $g:(x,y)\mapsto (ax,A(y))$ be the preimage in $\Gamma_1$ of a generator of $\Gamma_{1B}$. Similar arguments as in the proof of Lemma \ref{casebsuspension} show that $A$ has finite order. 

From now on we assume that $A=Id$. We distinguish two cases: either $\Gamma_1$ is cyclic generated by $g$ or $\Gamma_1$ has rank two.

If $\Gamma_1$ is cyclic, then the the genus one fibration on $X$ is simply induced by the second projection $\PP^1\times \PP^1\rightarrow \PP^1$ and we conclude as in the proof of Lemma \ref{casebsuspension}. 

Assume now that $\Gamma_1$ has rank two. Then it is generated by $g$ and another element $h$ of the form $g:(x,y)\mapsto (x,B(y))$ where $B(y)=y+1$ or $B(y)=by,b\in \CC^*$. Applying Theorem \ref{bdcentralizer} to both $g$ and $h$, we see that $\Gamma$ is a subgroup of $\{(x,y)\mapsto (cx,y+d), c\in \CC^*, d\in \CC\}$ or $\{(x,y)\mapsto (cx,dy), a\in \CC^*, d\in \CC^*\}$. Again this is a torus situation.
\end{proof}

Finally the results in this section can be summarized as follows:

\begin{proposition}\label{suspensionprop}
Suppose that $\Fol$ is an infinite suspension of an elliptic curve over a hyperbolic curve. Then up to geometric conjugation and up to taking a subgroup of finite index, $\Gamma$ is a subgroup of $\PGL(2,\CC)\times\PGL(2,\CC)$ acting on $Y=\PP^1\times \PP^1$. The open subset $U\subset Y$ is $\Omega_1\times \Omega_2$ where $\Omega_1$ is $\CC$ or $\CC^*$ and $\Omega_2$ is an invariant component of a classical Kleinian group. The genus one fibration on $X$ is induced by the second projection $\PP^1\times \PP^1\rightarrow \PP^1$.
\end{proposition}

\subsection{Fibrations}\label{jonquieresfibration}
In this subsection we consider the case where the induced foliation $\Fol$ on $X$ is a fibration $f:X\rightarrow C$. We still denote the covering map $U\rightarrow X$ by $\pi$, and by $\Gamma_r$ the kernel of the homomorphism $\Gamma\rightarrow \Gamma_B$. By Corollary \ref{regularonacylinder} $\Gamma$ acts by biholomorphisms on $r^{-1}(r(U))$. Up to replacing $\Gamma$ with a finite index subgroup, we can assume that $\Gamma$ fixes any singular point of any fiber in $r^{-1}(r(U))$. For a point $x\in r(U)$ we denote by $\Gamma_x$ the subgroup of $\Gamma$ that preserves the fiber $r^{-1}(x)$. The group $\Gamma_x$ acts by biholomorphisms on $r^{-1}(x)$. The group $\Gamma_r$ is a normal subgroup of $\Gamma_x$ for all $x$. Recall that we assume that $\Gamma_B$ is infinite throughout Section 9.

\begin{lemma}\label{fibrationcaselemma}
\begin{enumerate}
	\item For any $x\in r(U)$, the intersection of $U$ with $r^{-1}(x)$ has finitely many connected components.
	\item If $\Omega$ is a connected component of $U\cap r^{-1}(x)$, then the subgroup $\Gamma_{\Omega}$ of $\Gamma$ preserving $\Omega$ satisfies 
\[
\Gamma_\Omega\subset \Gamma_r\subset \Gamma_x
\] 
and is a finite index subgroup of $\Gamma_x$.
	\item For any $x\in r(U)$, the fiber $r^{-1}(x)$ is non-singular, i.e.\ $r^{-1}(x)=\PP^1$.
	\item $r(U)$ is a proper subset of $B=\PP^1$.
	
\end{enumerate}
\end{lemma}
\begin{proof}
Let $\Omega$ be a connected component of $U\cap r^{-1}(x)$ for some point $x\in r(U)$. It is a leaf of the foliation $\pi^*\Fol$ on $U$, i.e.\ $\pi(\Omega)$ is a fiber of $f:X\rightarrow C$. We assume that it is not a multiple fiber. The group $\Gamma_{\Omega}$ the subgroup of $\Gamma$ that preserves $\Omega$; it is a subgroup of $\Gamma_x$. The quotient of $\Omega$ by $\Gamma_{\Omega}$ is the fiber $\pi(\Omega)$. The group $\Gamma_{\Omega}$ is the image of the composition $\pi_1(\pi(\Omega))\rightarrow \pi_1(X)\rightarrow \Gamma$. 
Therefore $\Gamma_{\Omega}$ is a normal subgroup of $\Gamma$ and it does not depend on the leaf $\Omega$ as long as $\pi(\Omega)$ is a regular fiber of $f:X\rightarrow C$. In particular $\Gamma_{\Omega}$ preserves fiberwise the rational fibration, i.e.\ $\Gamma_{\Omega}\subset \Gamma_r$. For a leaf $\Omega'$ such that $\pi(\Omega')$ is a multiple fiber, $\Gamma_{\Omega}$ is a subgroup of finite index of $\Gamma_{\Omega'}$.  

Let $x\in r(U)$. Suppose that $r^{-1}(x)$ is a singular fiber of the rational fibration. By our hypothesis at the beginning of the section, $\Gamma_{\Omega}$ fixes all singular points of $r^{-1}(x)$. Thus it is solvable. Consequently $\Omega$ is biholomorphic to $\CC$ or $\CC^*$ and fibers of $f$ are elliptic curves. This implies that the intersection of $U$ with an irreducible component of $r^{-1}(x)$ is either empty or biholomorphic to $\CC$ or $\CC^*$. Thus the first assertion of our lemma is true if $r^{-1}(x)$ is a singular fiber. Let us assume now that $r^{-1}(x)$ is not singular, i.e.\ is isomorphic to $\PP^1$. Then we can think of $\Gamma_{\Omega}$ as a classical Kleinian group with an invariant component $\Omega$. The intersection of $U$ with $r^{-1}(x)$ is a union of connected components of $\Gamma_{\Omega}$. By the above discussion all of them are $\Gamma_{\Omega}$-invariant. Ahlfors's finiteness theorem (see \cite{Ahl64} , \cite{Sul85}) says that $(U\cap r^{-1}(x))/\Gamma_{\Omega}$ is a finite union of compact Riemann surfaces. This implies that $r^{-1}(x)\cap U$ has only finitely many connected components. 

In the first paragraph of the proof we infered $\Gamma_{\Omega}\subset \Gamma_r$; we also have an inclusion $\Gamma_r\subset \Gamma_x$. The above finiteness of connected components implies that $\Gamma_{\Omega}$ has finite index in $\Gamma_x$. From this and the infiniteness of $\Gamma_B$ we deduce that any point $x\in r(U)$ has an infinite orbit under $\Gamma_B$. As $\Gamma$ acts by biholomorphisms on $r^{-1}(r(U))$ by Corollary \ref{regularonacylinder}, we conclude that $r$ has no singular fibers over $r(U)$.

Any element of $\Gamma_B$ has a fixed point in $B=\PP^1$. If $r(U)=B$ then a fixed point $x$ of some element of infinite order of $\Gamma_B$ would give rise to a $\Gamma_x$ such that $\Gamma_r$ has infinite index in $\Gamma_x$, contradicting the second point of the lemma.
\end{proof}

Up to making elemantary transformations on a fiber over a point outside $r(U)$ we can and will assume that $Y=\PP^1\times \PP^1$ and that $r$ is the projection onto the first factor. As a finitely generated subgroup of $\PGL(2,\CC)$, $\Gamma_B$ has a torsion free subgroup of finite index by Selberg's lemma (see \cite{Sellem}). Up to replacing $\Gamma$ with the preimage of such a subgroup of $\Gamma_B$, we can and will assume that \emph{$\Gamma_B$ is torsion free}. Since $\Gamma_r$ has finite index in $\Gamma_x$ by Lemma \ref{fibrationcaselemma}, this is equivalent to say that \emph{for any $x\in B$, we have $\Gamma_x=\Gamma_r$}.

By Proposition \ref{hasafoliatedprojectivestructure}, $\Fol$ is equipped with a foliated $(\PGL(2,\CC),\PP^1)$-structure. Remark that regular fibers of $f$ are isomorphic to each other and have the same $(\PGL(2,\CC),\PP^1)$-structure. This is clear if the fibers are $\PP^1$. If a fiber of $f$ is an elliptic curve, then its $(\PGL(2,\CC),\PP^1)$-structure is induced by one of its infinite covers $\CC$ or $\CC^*$, depending on whether $\Gamma_r$ has rank one or two. If fibers of $f$ have genus $\geq 2$, then the claim follows from Proposition \ref{fibrationprojstructure}. 

We will think of $\Gamma_r$ as a family of classical Kleinian groups parametrized by $r(U)$. Denote still by $\Omega$ a connected component of the preimage of a regular fiber of $f$. The pair $(\Omega,\Gamma_{\Omega})$ gives the developing map and holonomy group of the $(\PGL(2,\CC),\PP^1)$-structure on the corresponding fiber of $f$. Therefore up to conjugation they do not depend on the fiber containing $\Omega$. More precisely there exists a holomorphic map $H:r(U)\rightarrow \PGL(2,\CC)$ such that $H\Gamma_{\Omega}H^{-1}$ is a constant family of classical Kleinian groups. In other words, denoting the field of meromorphic functions on $r(U)$ by $\CC\{x\}$, there are $H\in \PGL(2,\CC\{x\})$ and a subgroup $\Gamma_c\subset \PGL(2,\CC)\subset \PGL(2,\CC\{x\})$ such that $H^{-1}\Gamma_c H$ is $\Gamma_{\Omega}$, a subgroup of $\PGL(2,\CC(x))\subset \PGL(2,\CC\{x\})$. Therefore the conjugation $H$ is actually algebraic, i.e.\ it is an element of $\PGL(2,\CC(x))$. In other words $H$ realizes a geometric conjugation of our birational Kleinian group so that $\Gamma_{\Omega}$ becomes contained in the subgroup $\{\Id\}\times \PGL(2,\CC)$ of $\Aut(\PP^1\times \PP^1)$. Moreover we obtain that $U=r(U)\times \Omega$ by connectedness of $U$.  Thus $\Gamma_r=\Gamma_{\Omega}$. Consequently $f:X\rightarrow C$ is a locally trivial fibration without multiple fibers and $r(U)/\Gamma_B=C$. In summary we have:
\begin{proposition}\label{fibrationproposition}
Up to geometric conjugation and up to taking finite index subgroup, $Y=\PP^1\times \PP^1$, $U=r(U)\times \Omega$ and $\Gamma_r=\Gamma_{\Omega}$ is a subgroup of $\{\Id\}\times \PGL(2,\CC)$.
\end{proposition}
Finally we distinguish three situations according to whether the fibers of $f$ are $\PP^1$,  elliptic curves or hyperbolic Riemann surfaces. If the fibers of $f$ are $\PP^1$, then $U$ is necessarily $r(U)\times \PP^1$. 
\begin{example}
Consider the cyclic group $\Gamma$ generated by
\begin{align*}
&(x,y)\dashmapsto (bx,cx^ky), \quad b,c\in\CC^*, k\in \ZZ^* \ \vert b \vert\neq 1.
\end{align*}
The quotient $\CC^*\times \PP^1/\Gamma$ is a geometrically ruled surface over an elliptic curve; it is a decomposable ruled surface because of the two disjoint sections $\{y=0\},\{y=\infty\}$. 
\end{example}

\begin{proposition}\label{ellipticbundles}
Suppose that the fibers of $f$ have genus $1$. Then up to geometric conjugation and up to finite index subgroup we are in one of the two following situations
\begin{enumerate}
	\item $Y=\PP^1\times \PP^1$, $U=D_1\times \CC^*$ and $D_1$ is an invariant component of the classical Kleinian group $\Gamma_B$. $\Gamma$ is a subgroup of 
	\[J^+:=\{(x,y)\dashrightarrow (\eta(x),R(x)y)\vert \eta\in\PGL(2,\CC), R\in\CC(x)^*\}=\CC(x)^*\rtimes\PGL(2,\CC).\]
And $\Gamma_r\subset \{\Id\}\times \PGL(2,\CC)$ is a cyclic central subgroup of $\Gamma$.
	\item $Y=\PP^1\times \PP^1$, $U=D_1\times \CC$ and $D_1$ is an invariant component of the classical Kleinian group $\Gamma_B$. $\Gamma$ is a subgroup of $J^*:=\{(x,y)\dashrightarrow (\eta(x),y+R(x))\vert \eta\in\PGL(2,\CC), R\in\CC(x)\}=\CC(x)\rtimes\PGL(2,\CC)$. $\Gamma_r\subset \{\Id\}\times \PGL(2,\CC)$ is a central subgroup of $\Gamma$ isomorphic to $\ZZ^2$.
\end{enumerate}
\end{proposition}
\begin{proof}
By the exact sequence $1\rightarrow \Gamma_r\rightarrow \Gamma\rightarrow \Gamma_B\rightarrow 1$ the group $\Gamma$ acts by outer automorphisms on the abelian group $\Gamma_r$. Any such outer automorphism is induced by an automorphism of an elliptic curve fixing the origin. Therefore up to replacing $\Gamma$ with a finite index subgroup, we can assume that $\Gamma_r$ is central in $\Gamma$. This is equivalent to say that we replace $X$ with a finite unramified cover which is a principal elliptic bundle (cf. \cite{BHPV} V.5). 

It remains to show that $\Gamma$ is contained in $J^+$ or $J^*$. This is because the centralizer of $\Gamma_r$ is respectively $J^+$ or $J^*$ by Thereom \ref{bdcentralizer}.
\end{proof}

Let us see two examples where the group $\Gamma$ is not conjugate to a subgroup of\\ $\PGL(2,\CC)\times \PGL(2,\CC)$. The first is given by the standard construction of primary Kodaira surfaces (see Paragraph \ref{primarykodaira}). In the second example $\Gamma$ is not conjugate to a group of automorphisms of a projective surface:
\begin{example}\label{kodairasurfaceex}
Consider the group $\Gamma$ generated by the following two elements
\begin{align*}
&(x,y)\mapsto (x,ay)\\
&(x,y)\dashmapsto (bx,cx^ky), \quad a,b,c\in\CC^*, k\in \ZZ^* \ \vert a\vert,\vert b \vert\neq 1.
\end{align*}
By lifting the deck transformations to $\CC^2$ we see that the quotient $\CC^*\times \CC^*/\Gamma$ is a primary Kodaira surface. Note $\Gamma$ is included in the toric subgroup. The birational transformation $(x,y)\mapsto (bx,cx^ky)$ has linear degree growth and is not birationally conjugate to an automorphism of a projective surface.
\end{example}

\begin{proposition}\label{highergenusfibration}
Suppose that the fibers of $f$ have genus $\geq 2$. Then up to geometric conjugation and up to finite index subgroup we have: $Y=\PP^1\times \PP^1$, $\Gamma=\Gamma_B\times \Gamma_r$ and $U=D_1\times D_2$ where $D_1,D_2$ are respectively invariant components of the classical Kleinian groups $\Gamma_B,\Gamma_r$. 
\end{proposition}
\begin{proof}
For $x\in r(U)$, we denote by $\PP^1_x$ the fiber of $r$ over $x$ and by $G_x$ the restriction of $\Gamma_r$ on $\PP^1_x$. Since $\Gamma_r\subset \{\Id\}\times \PGL(2,\CC)$ by Proposition \ref{fibrationproposition}, the classical Kleinian groups $G_x$ are the same if we identify the different $\PP^1_x$.

Let $\gamma$ be an element of $\Gamma$. We write it as $\gamma:(x,y)\dashmapsto (\gamma_B(x),\rho(x)(y))$ where $\gamma_B$ is the image of $\gamma$ in $\Gamma_B$ and $\rho$ is a rational map from $B$ to $\PGL(2,\CC)$. Since $\Gamma_r$ is a normal subgroup, we have $\gamma\Gamma_r\gamma^{-1}=\Gamma_r$. Then we have 
\[
\rho(x)G_x\rho(x)^{-1}=G_{\gamma_B(x)}=G_x
\]
This means that $\rho(x)$ is in the normalizer of $G_x$ in $\PGL(2,\CC)$. As the fibers of $f$ are hyperbolic, $G_x$ is a non-elementary Kleinian group and its normalizer is a discrete subgroup of $\PGL(2,\CC)$ (see \cite{Mas88} Chapter V E.10). Being a holomorphic map into this normalizer, the map $\rho$ is constant and consequently $\gamma\in \PGL(2,\CC)\times \PGL(2,\CC)$

Therefore $\Gamma$ is a subgroup of $\PGL(2,\CC)\times \PGL(2,\CC)$. Let $\Gamma_2$ be its projection onto the second factor. Since the fibers of the locally trivial fibration $f:X\rightarrow C$ have genus $\geq 2$, there exists a finite unramifed cover $\bar{C}$ of $C$ such that the pullback fibration over $\bar{C}$ is simply a product (see \cite{BHPV} V.6). In particular a finite unramified cover of $X$ is a product of two curves. This implies that $\Gamma_2/\Gamma_r$ is a finite group. Hence up to replacing $\Gamma$ with a finite index subgroup, we have $\Gamma_2=\Gamma_r$ and $\Gamma$ is isomorphic to $\Gamma_B\times \Gamma_r$. 
\end{proof}

\subsection{Complex tori}

\begin{proposition}\label{toriprop}
Assume that $X$ is a complex torus. Then up to geometric conjugation the birational Kleinian group $(Y,\Gamma,U,X)$ satisfies that 
\begin{itemize}
	\item $Y=\PP^1\times \PP^1$. $U$ is one of the three Zariski open sets: $\CC^2$, $\CC\times\CC^*$ or $\CC^*\times \CC^*$.
	\item $\Gamma$ can be identified with a lattice in $U$ with its natural structure of algebraic group, i.e.\ the elements of $\Gamma$ have respectively the form $(x,y)\mapsto(x+a,y+b)$, $(x,y)\mapsto(x+a,by)$ or $(x,y)\mapsto(ax,by)$.
\end{itemize}
\end{proposition}
\begin{proof}
The fundamental group of a complex torus is isomorphic to $\ZZ^4$. After replacing $\Gamma$ with a subgroup of finite index we can and will assume that $\Gamma$ and $\Gamma_B$ are free abelian groups. Remark that the only regular foliations on complex tori are turbulent foliations and linear foliations (which can be considered as suspensions in some cases). We know by Proposition \ref{turbulentclassification} that $\Fol$ is linear. We also know by Lemma \ref{r(B)oftorus} that $r(U)$ is $\CC$ or $\CC^*$.

A linear foliation on the torus $X$ is induced by a linear map from $\CC^2$ to $\CC$ equivariant under $\pi_1(X)$, the action on $\CC$ being by translations. The image of $\pi_1(X)$ in $\Aut(\CC)$ is a free abelian group of rank $2,3$ or $4$, depending on whether the leaves of $\Fol$ are elliptic curves, $\CC^*$ or $\CC$. The open set $U$ being a quotient of $\bC^2$ with $r$ induced by the above linear map, is a $\CC$ or $\CC^*$ bundle over $\CC$ or $\CC^*$. It is biholomorphic to $\CC^2$, $\CC\times\CC^*$ or $\CC^*\times \CC^*$. In other words the complement $r^{-1}(r(U))\backslash U$ is one or two holomorphic sections of $r$ over $r(U)$. To conclude we need to prove that these sections are algebraic. 

We assume first that $\Gamma_B$ is free of rank $\geq 2$. In this case the descriptions of $\Gamma$ and $U$ given in the statement follow from Theorem \ref{zhaocentralizer}. We only need to observe that the conjugation that we need in Theorem \ref{zhaocentralizer} is a geometric conjugation of birational Kleinian groups. This is because the conjugation is a sequence of elementary transformations. If we let the elementary transformations to be done outside $U$, then the conjugation is geometric. Remark that the hypothesis that the rank of $\Gamma_B$ is $\geq 2$ is always satisfied when $\Fol$ is irrational. This is because if $\Fol$ is irrational then the action of $\Gamma_B$ on $r(U)$ is not discrete.

Consider now the case where $\Gamma_B$ is cyclic. In this case $r(U)$ is necessarily $\CC^*$, $\Fol$ is an elliptic fibration and $\Gamma_r$, the subgroup of $\Gamma$ preserving fiberwise the rational fibration, is a free abelian group of rank $1$ or $2$. If $\Gamma_r$ has rank $1$ then $r^{-1}(x)\cap U=\CC^*$ for any $x\in r(U)$ and $\Gamma_r$ acts by multiplications in the fibers; if $\Gamma_r$ has rank $2$ then $r^{-1}(x)\cap U=\CC$ for any $x\in r(U)$ and $\Gamma_r$ acts by translations in the fibers.

First consider the case where $\Gamma_r$ has rank $1$. By Theorem \ref{zhaocentralizer} up to conjugation $\Gamma_r$ is generated by an element $\gamma_1:(x,y)\mapsto (x,ay)$ and $\Gamma$ is generated by $\gamma_1$ and $\gamma_2:(x,y)\mapsto (bx,R(x)y)$ with $R\in \CC(x)^*$; being a sequence of elementary transformations the conjugation can be done outside $U$ so that it is a geometric conjugation of birational Kleinian groups. Therefore $U$ is the standard Zariski open set $\CC^*\times \CC^*$. As $\gamma_2$ acts regularly over $r(U)=\CC^*$, the rational function $R$ has neither zeros nor poles over $\CC^*$. Thus $R(x)=x^k$ for some $k\in\ZZ$. If $k\neq 0$ then the quotient $\CC^*\times \CC^*$ would be a primary Kodaira surface (cf. Example \ref{kodairasurfaceex}). Hence the conclusion.

Finally consider case where $\Gamma_r$ has rank $2$. By Theorem \ref{zhaocentralizer} up to conjugation $\Gamma_r$ is generated by two elements $\gamma_j:(x,y)\mapsto (x,y+a_j), j=1,2$ and $\Gamma$ is generated by $\gamma_1, \gamma_2$ and $\gamma_3:(x,y)\mapsto (bx,y+R(x))$ with $R\in \CC(x)$. Again the conjugation can be chosen to be geometric conjugation. Therefore $U$ is the standard Zariski open set $\CC^*\times \CC$. As $\gamma_3$ acts in a regular way over $r(U)=\CC^*$, the rational function $R$ has no poles over $\CC^*$. Thus $R(x)=\frac{Q(x)}{x^k}$ for some $k\in\ZZ$ and $Q\in \CC[x]$. The iterates of $\gamma_3$ can be written as 
\[\gamma_3^n:(x,y)\mapsto (b^nx,y+S(x)) \quad \text{where}\quad S(x)=\frac{1}{x^k}\sum_{j=0}^{n-1}\frac{Q(bx)}{b^k}.\]
Thus $\gamma_3$ has bounded degree growth and is an elliptic element. By conjugating $\Gamma$ with elements of the form $(x,y)\mapsto (x,y+\frac{P(x)}{x^k})$ with $P\in \CC[x]$, we can, without changing $\gamma_1,\gamma_2$, put $\gamma_3$ in the form $(x,y)\mapsto (bx,y)$. Hence the conclusion.
\end{proof}

\subsection{Bidisk quotients}
Finally we consider the case where $\Fol$ is a transversely hyperbolic foliation with dense leaves. By Proposition \ref{fullbasecase} $r(U)$ is a proper subset of $\PP^1$. Since every leaf of $\Fol$ is dense in $X$, any $\Gamma_B$-orbit in $r(U)$ is infinite. This implies by Corollary \ref{regularonacylinder} that $r^{-1}(r(U))$ is biholomorphic to $r(U)\times \PP^1$. Then by Proposition \ref{hasafoliatedprojectivestructure} $\Fol$ is equipped with a foliated $(\PGL(2,\CC),\PP^1)$-structure. By Corollary \ref{DGcorollary} $X$ is a bidisk quotient. 

Consider the composition $\pi_1(X)\rightarrow \Gamma\rightarrow \Gamma_B\subset \PGL(2,\CC)$. As in Lemma \ref{r(B)ofbidiskquotient}, Margulis supperrigidity (\cite{Mar91}) implies that $\pi_1(X)\rightarrow \Gamma_B$ is, up to conjugation in $\PGL(2,\CC)$ and up to a Galois conjugate, the projection of the lattice into one of the factors $\PSL(2,\RR)$. Moreover $r(U)$ is a round disk and $U$ is biholomorphic to the bidisk. 

By Proposition \ref{bidiskprojstructure} the foliated $(\PGL(2,\CC),\PP^1)$-structure on $\Fol$ is constant along the foliation. Every leaf of $\Fol$ can be identified with the unit disk as $(\PGL(2,\CC),\PP^1)$-manifold. Consequently $r^{-1}(x)\cap U$ is a round disk for any $x\in r(U)$, and there exists a holomorphic map $h:\bD\rightarrow \PGL(2,\CC)$ and a $\pi_1(X)$-equivariant biholomorphism $H:\bD\times \PP^1\rightarrow \bD\times \PP^1$ defined by $(x,y)\mapsto (x,h(x)(y))$ which sends $U$ to $\bD\times\bD$. It remains to show that $h$ is in fact an algebraic map.

\begin{proposition}\label{bidiskrigidity}
If $\Fol$ is a transversely hyperbolic foliation with dense leaves, then up to geometric conjugation we are in the standard example of bidisk quotient.
\end{proposition}

\begin{proof}
We would like to show that $H$ is the restriction of a birational transformation preserving the rational fibration, i.e.\ an element of the \Jonqui{} group. Then it realizes the desired geometric conjugation of birational Kleinian groups.

Our group $\Gamma$ is isomorphic to a cocompact irreducible lattice in $\PSL(2,\RR)\times \PSL(2,\RR)$. We will use two properties of irreducible lattices:
\begin{enumerate}
	\item $\Gamma$ has subgroups isomorphic to $\ZZ^2$ because $\PSL(2,\RR)\times \PSL(2,\RR)$ has real rank two (see \cite{PraRag72}).
	\item Elements $\Gamma_B$, are either elliptic or loxodromic elements of $\PSL(2,\RR)$ because every element of a cocompact lattice is semisimple (see \cite{BH62}, \cite{MT62}).
\end{enumerate}
Let $\gamma_1,\gamma_2\in\Gamma$ such that $<\gamma_1,\gamma_2>$ a subgroup isomorphic to $\ZZ^2$. Theorem \ref{zhaocentralizer} together with the second property above allow us to write $\gamma_1,\gamma_2$, up to conjugation in $\Jonq$, as $\gamma_i:(x,y)\mapsto (a_ix,b_iy), i=1,2$ or $\gamma_i:(x,y)\mapsto (a_ix,b_i+y), i=1,2$. Here the latter case is impossible because the holomorphic conjugation $H$ is supposed to make $\gamma_i, i=1,2$ also loxodromic in the $y$ coordinate. Hence we can write them as $\gamma_i:(x,y)\mapsto (a_ix,b_iy), i=1,2$. 

Denote by $\partial U$ the boundary of $U$ in $\PP^1\times \PP^1$ and by $\boundh(U)$ the boundary of $U$ in $\bD\times \PP^1$. The two horizontal disks $S_1=\{y=0,x\in\bD\},S_2=\{y=\infty,x\in\bD\}$ are invariant under $\gamma_i, i=1,2$. We want to show that they are contained in $\boundh(U)$. Since $\Gamma_B$ is isomorphic to $\Gamma$, the subgroup of $\CC^*$ generated by $a_1,a_2$ is isomorphic to $\ZZ^2$, thus dense. Therefore there is a non-trivial sequence of birational transformations $(\delta_n)$ in $<\gamma_1,\gamma_2>$ such that $\delta_{nB}$ tends to the identity in $\PGL(2,\CC)$. Up to extracting a subsequence, $\delta_n\vert_{U}$ converges to a holomorphic map from $U$ to $U\cap \partial U$. By discontinuity of the action of $\Gamma$ on $U$, the limit map has values in $\partial U$; as $\delta_{nB}$ tends to the identity, the image of the limit map is in $\boundh(U)$. Since each $\delta_n$ has the form $(x,y)\mapsto (c_nx,d_ny)$, the limit disk is either $S_1$ or $S_2$. By considering the sequence $(\delta_n^{-1})$, we infer that $S_1$ and $S_2$ are both contained in $\boundh(U)$.

Consider the conjugates $H\circ\delta_n \circ H^{-1}$. The sequence $H\circ\delta_n \circ H^{-1}\vert _{\bD\times\bD}$ converges to a horizontal disk as well. Therefore we can assume that $H(S_i)=S_i,i=1,2$. In other words the holomorphic conjugation map $H$ has the form $(x,y)\mapsto (x,A(x)y)$ where $A$ is a holomorphic function $\bD\rightarrow \CC^*$. To conclude we need to show that $A$ is algebraic. It suffices to exhibit other disks in $\boundh(U)$ defined by algebraic functions. We can conjugate the $\gamma_1,\gamma_2$ in $\Gamma$ to obtain $\gamma_3,\gamma_4\in\Gamma$ such that $<\gamma_3,\gamma_4>$ is isomorphic to $\ZZ^2$ but $\gamma_3,\gamma_4$ do not commute with $\gamma_1,\gamma_2$. We apply the above discussion to $\gamma_3,\gamma_4$ and obtain other algebraic disks in $\boundh(U)$. The proof is finished. 
\end{proof}

\section{Classification and proof}\label{synthesis}

\subsection{Classification}\label{longlistsection}

\begin{maintheorem}\label{verylongthm}
Let $(Y,U,\Gamma,X)$ be a birational Kleinian group in dimension $2$. Suppose that if $\Gamma$ is virtually cyclic or if $X$ is a class VII surface then $\Gamma$ contains no loxodromic elements. Then up to geometric conjugation and up to taking a finite index subgroup, we are in one of the cases in the following table:

 \begin{longtable}[c]{|P{.02\textwidth}|P{.24\textwidth} | P{.24\textwidth} | P{.20\textwidth} | P{.12\textwidth} |}

 \hline
 \multicolumn{5}{| c |}{Birational Kleinian groups in dimension two}\\
 \hline
  &$Y$ & $U$ & $\Gamma$  & $X$ \\
 \hline
 \endfirsthead

 \hline
 \hline
  &$Y$ & $U$ & $\Gamma$  & $X$ \\
 \hline
 \endhead

 \hline
 \endfoot

 \hline
 \hline\hline
 \endlastfoot

 1& $B\times\PP^1$ where $B$ is compact Riemann surface & $B\times D_1$ where $D_1$ is an invariant component of a classical Kleinian group $\Gamma_1$ & $\{\Id\}\times \Gamma_1\subset \Aut(B)\times \Aut(\PP^1)$ & $B\times (D_1/\Gamma_1)$  \\
\hline

2& $\PP(\EE)$ where $\EE$ is an extension of $\OO_B$ by $\OO_B$ and $B$ is a compact Riemann surface & the complement of a section of the ruling & isomorphic to $\ZZ^2$ & a principal elliptic bundle  \\
\hline

3& a blow-up of a decomposable ruled surface & a Zariski open subset whose intersection with each fiber of the ruling is $\CC^*$ & isomorphic to $\ZZ$ & an elliptic fibration with only type $mI_0$ singular fibers  \\
\hline

4& $B\times\PP^1$ where $B$ is an elliptic curve & $B\times D_1$ where $D_1$ is an invariant component of a classical Kleinian group $\Gamma_1$ & $\exists \rho:\Gamma_1\rightarrow \Aut(B), \forall \gamma\in\Gamma, \exists \gamma_1\in\Gamma_1, \gamma=(\rho(\gamma_1),\gamma_1)\in \Aut(B)\times \Aut(\PP^1)$ & $B$-bundle over $D_1/\Gamma_1$  \\
\hline

5& an indecomposable geometrically ruled surface over an elliptic curve with a section of zero self-intersection & the complement of the section of zero self-intersection & isomorphic to $\ZZ^2$ & complex torus  \\
\hline

6& a decomposable geometrically ruled surface over an elliptic curve & the complement of two disjoint sections & isomorphic to $\ZZ$ & complex torus  \\
\hline

7& $\PP^2$ & $\CC^2$ & isomorphic to $\ZZ^4$ & complex torus  \\
\hline

8& $\PP^2$ & $\CC\times \CC^*$ & isomorphic to $\ZZ^3$ & complex torus  \\
\hline

9& $\PP^2$ & $\CC^*\times \CC^*$ & isomorphic to $\ZZ^2$ & complex torus  \\
\hline

10& $\PP^2$ & $\CC^2$ & a group of affine transformations, an extension of $\ZZ^2$ by $\ZZ^2$ & primary Kodaira surface  \\
\hline

11& $\PP^2$ & $\CC^2\backslash\{0\}$ & isomorphic to $\ZZ$ & Hopf surface  \\
\hline

12& $\PP^2$ & $\HH\times \CC$ & a solvable group of affine transformations  & Inoue surface  \\
\hline

13& $\PP^2$ & $\bB^2$ & a cocompact lattice in $\PU(1,2)$  & ball quotient  \\
\hline

14& $\PP^1\times \PP^1$ & $\bD^1\times \bD^1$ & an irreducible cocompact lattice in $\PSL(2,\RR)\times \PSL(2,\RR)$  & bidisk quotient  \\
\hline

15& a Hirzebruch surface blown up at at most two fibers & a Zariski open subset whose intersection with each fiber of the ruling is $\CC^*$ & a cyclic group generated by $(x,y)\mapsto (ax,by)$  & Hopf surface  \\
\hline

16& a Hirzebruch surface blown up at at most one fiber & a Zariski open subset whose intersection with each fiber of the ruling is $\CC^*$ & a cyclic group generated by $(x,y)\mapsto (x+a,by)$  & Hopf surface  \\
\hline

17& $\PP^1\times \PP^1$ & $D_1\times D_2$ where $D_i$ is an invariant component of a classical Kleinian group $\Gamma_i$ & $\Gamma_1\times \Gamma_2\subset \PGL(2,\CC)\times \PGL(2,\CC)$  & $(D_1/\Gamma_1)\times (D_2/\Gamma_2)$  \\
\hline

18& $\PP^1\times \PP^1$ & $D_1\times \CC^*$ where $D_1$ is an invariant component of a classical Kleinian group $\Gamma_1$ & isomorphic to a central extension of $\Gamma_1$ by $\ZZ$, any element has the form $(x,y)\dashrightarrow (\gamma_1(x),R(x)y)$ where $\gamma_1\in\Gamma_1$ and $R\in \CC(x)^*$  & principal elliptic bundle over $D_1/\Gamma_1$ \\
\hline

19& $\PP^1\times \PP^1$ & $D_1\times \CC$ where $D_1$ is an invariant component of a classical Kleinian group $\Gamma_1$ & isomorphic to a central extension of $\Gamma_1$ by $\ZZ^2$, any element has the form $(x,y)\dashrightarrow (\gamma_1(x),y+R(x))$ where $\gamma_1\in\Gamma_1$ and $R\in \CC(x)$  & principal elliptic bundle over $D_1/\Gamma_1$ \\
\hline

20& $\PP^1\times \PP^1$ & $D_1\times \PP^1$ where $D_1$ is an invariant component of a classical Kleinian group $\Gamma_1$ & isomorphic to $\Gamma_1$, preserves the projection of $\PP^1\times \PP^1$ onto the first factor & geometrically ruled surface over $D_1/\Gamma_1$ \\
\hline

\end{longtable}

\end{maintheorem}
The surface $Y$ is rational except the first six cases. The group $\Gamma$ is a group of automorphisms except in the last three cases. In cases 18), 19), 20), $\Gamma$ may or may not be conjugate to a group of automorphisms of some projective surface. 

We explain how Theorem \ref{verylongthm} implies Theorem \ref{themainthm}. Under the hypothesis that $U$ is simply connected and $X$ is \Kah{}, the group $\Gamma$ is isomorphic to the \Kah{} group $\pi_1(X)$. By Hodge Theorem the abelianization of a \Kah{} group has even rank, thus is not cyclic. Hence the hypothesis of Theorem \ref{themainthm} implies the hypothesis of Theorem \ref{verylongthm}. It suffices thus to remark that in the list given by Theorem \ref{verylongthm} the only cases where $U$ is simply connected and $X$ is \Kah{} are the seven cases listed in Theorem \ref{themainthm}.

When $X$ is an Inoue surface, we proved in our previous paper \cite{Zhaobirinoue} that there is no cyclic birational Kleinian group having quotient surface $X$; in other words the two hypothesis in Theorem \ref{verylongthm} can be dropped. The proof uses particular geometric features of Inoue surfaces as well as ergodic properties of loxodromic birational transformations. Known Class VII surfaces have rather small fundamental groups, typically cyclic.

\subsection{Proof}

Let $(Y,\Gamma,U,X)$ be a birational Kleinian group in dimension two. By Theorem \ref{nonnegthm} $Y$ has Kodaira dimension $-\infty$. Thus it is either rational or birational to a ruled surface over a non-rational curve. 

Suppose that $Y$ is not rational. Then there is a rational fibration $r:Y\rightarrow B$ onto a curve $B$ of genus $\geq 1$. Every birational transformation of $Y$ preserves the fibration $r$. If $\Gamma_B$ is a finite group then all possibilities are classified by Theorem \ref{thmruled1}. This gives the first three cases in Theorem \ref{verylongthm}. If $\Gamma_B$ is infinite then $B$ is an elliptic curve and all possibilities are classified by Theorem \ref{thmruledelliptic}. This gives cases 1), 4), 5) and 6) in Theorem \ref{verylongthm}.

From now on we assume that $Y$ is a rational surface. We apply Theorem \ref{strongTitsalternative} to $\Gamma$. We will first deal with the cases where a rational fibration is preserved by $\Gamma$ up to conjugation, then deal with other cases of Theorem \ref{strongTitsalternative}.

\subsubsection{Rational surfaces I: invariant rational fibrations}\label{rationalsurfaceinvariantrationalfibration}

Assume that up to conjugation the group $\Gamma$ preserves a rational fibration, i.e.\ $\Gamma$ preserves a pencil of rational curves. This hypothesis always holds in Cases 2), 3) of Theorem \ref{strongTitsalternative} and can hold for some groups in case 1). Since the action of $\Gamma$ on $U$ is free, the set of base points of the invariant pencil does not intersect $U$. We can blow up the base points to get an invariant fibration. Thus we have:

\begin{lemma}\label{withoutsingularpoints2}
If $\Gamma\subset \Bir(Y)$ is conjugate in $\Bir(Y)$ to a group preserving a fibration, then the conjugation can be chosen to be geometric. 
\end{lemma}

Therefore we can and will assume that $\Gamma$ preserves a rational fibration $r:Y\rightarrow B$ where $B=\PP^1$. As before we denote by $\Gamma_B$ the image of $\Gamma$ in $\Aut(B)$. By Proposition \ref{withoutsingularpoints2} $r$ induces on $X$ a regular holomorphic foliation $\Fol$. If $\Gamma_B$ is a finite group then all possibilities are classified by Theorem \ref{thmruled1}. This gives the first three cases in Theorem \ref{verylongthm}. In this case $\Fol$ is a fibration. 

Assume now that $\Gamma_B$ is infinite. We distinguish two cases according to whether $r(U)=B$ or $r(U)$ is a proper subset of $B$.

\proofstep{First case}
Suppose that $r(U)=B$. Then by Proposition \ref{fullbasecase} the foliation $\Fol$ is a turbulent foliation, an obvious foliation on a Hopf surface or a suspension of $\PP^1$. 

Assume that $\Fol$ is a turbulent foliation. Then by Proposition \ref{turbulentclassification}, up to taking a subgroup of finite index of $\Gamma$, we are in Examples \ref{hopfruledex}, \ref{ruledsuspensionex} and \ref{ruledsuspensionextwo}. In the situation of these three examples $X$ is a Hopf surface or a geometrically ruled surface over an elliptic curve. If $X$ is a Hopf surface then this is Case 15) of Theorem \ref{verylongthm}. If $X$ is a geometrically ruled surface over an elliptic curve, then the situations described in Examples \ref{hopfruledex}, \ref{ruledsuspensionex} and \ref{ruledsuspensionextwo} are subcases of Case 20) in Theorem \ref{verylongthm}.

Assume that $\Fol$ is an obvious foliation on a non-elliptic Hopf surface. Then Proposition \ref{Hopfclassification} classifies all possibilities. There are three cases in Proposition \ref{Hopfclassification}. The last case does not satisfy $r(U)=B$. The first two cases satisfy $r(U)=B$ and they correspond to Cases 15) and 16) of Theorem \ref{verylongthm}.

Assume that $\Fol$ is a suspension of $\PP^1$. Then Proposition \ref{suspensionrationalcase} says that we are in Case 20) of Theorem \ref{verylongthm}

\proofstep{Second case}
Assume now that $r(U)\neq B$. If $X$ is a Hopf surface then $\Fol$ is a turbulent foliation, an elliptic fibration or an obvious foliation. All possibilities are classified in Propositions \ref{turbulentclassification}, \ref{Hopfclassification} and \ref{ellipticbundles}. Apart from the normal form of a Hopf surface, only the last case of Proposition \ref{Hopfclassification} satisfies $r(U)\neq B$ and it corresponds to Case 11) of Theorem \ref{verylongthm}.

If $\Fol$ is an obvious foliation on an Inoue surface then Theorem \ref{inoue} says that we are in Case 12) of Theorem \ref{verylongthm}.  

If $\Fol$ is an irrational linear foliation on a complex torus, then Proposition \ref{toriprop} says that we are in Cases 7), 8) or 9). If $\Fol$ is a transversely hyperbolic foliation with dense leaves, then we are in Case 14) by Proposition \ref{bidiskrigidity}.

If $\Fol$ is a suspension of an elliptic curve over an elliptic curve then a finite covering of $X$ is a torus where $\Fol$ becomes linear. If $\Fol$ is an infinite suspension of an elliptic curve over a hyperbolic curve then Proposition \ref{suspensionprop} says that we are in some refined subcases of Cases 18) or 19). 

Assume that $\Fol$ is a fibration. If the genus of a fiber is $\geq 2$ then Proposition \ref{highergenusfibration} says that we are in Case 17) of Theorem \ref{verylongthm}. If the fibers are elliptic curves then Proposition \ref{ellipticbundles} says that we are in Cases 18) and 19). If the fibers are $\PP^1$ then Proposition \ref{fibrationproposition} says that we are in Case 20).

\subsubsection{Rational surfaces II: other cases in strong Tits alternative}
Now we consider Cases 1), 4), 6) and 7) of Theorem \ref{strongTitsalternative} for $\Gamma$. Case 5) of Theorem \ref{strongTitsalternative} is excluded by our hypothesis in Theorem \ref{verylongthm}. Note that we do not need to consider Case 8) of Theorem \ref{strongTitsalternative} because being a quotient of the fundamental group of a compact manifold, $\Gamma$ is finitely generated.

\proofstep{Genus one fibration}
Let us consider the fourth case of Theorem \ref{strongTitsalternative}. The elements of $\Gamma$ are Halphen twists. There is a rational surface $Y'$ and a birational map $\phi:Y\rightarrow Y'$ such that $\Gamma'=\phi\Gamma\phi^{-1}$ acts by automorphisms and preserves a genus one fibration $\varphi:Y'\rightarrow C$. By Lemma \ref{birconjtoaut} or Lemma \ref{withoutsingularpoints2} we can choose $Y'$ and $\phi$ so that the conjugation $\phi$ is a geometric conjugation of birational Kleinian groups. For any Halphen twist $f$ preserving $\varphi$, the induced action of $f$ on $C$ has finite order (see \cite{CF03} Proposition 3.6). This is impossible by Lemma \ref{ellipticfibrationlem}.

\proofstep{Elliptic group}
Let us consider the first case of Theorem \ref{strongTitsalternative}. In this case there is a projective surface $Y'$ and a birational map $\phi:Y\rightarrow Y'$ such that $\Gamma'=\phi\Gamma\phi^{-1}$ is a group of automorphisms of $Y'$ and a finite index subgroup of $\Gamma'$ is in $\Aut^0(Y')$. By Lemma \ref{birconjtoaut} we can choose $Y'$ and $\phi$ so that the conjugation $\phi$ is a geometric conjugation of birational Kleinian groups. Thus we can and will assume that $\Gamma \subset \Aut^0(Y)$.

Assume that $Y$ is not $\PP^2$. Then there is a birational morphism $\varphi:Y\rightarrow Y'$ to a Hirzebruch surface $Y'$ and we have $\varphi\circ \Aut^0(Y) \circ \varphi^{-1}\subset \Aut^0(Y')$. Since the $\Aut^0$ of a Hirzebruch surface preserves a rational fibration, $\Gamma$ preserves a rational fibration on $Y$. Then all possibilities are already classified in \ref{rationalsurfaceinvariantrationalfibration}.

In the remaining case $Y=\PP^2$ and $\Gamma\subset \PGL_3(\CC)=\Aut(\PP^2)$. This is exactly the case of complex projective Kleinian groups. The complete classification is given by Theorem \ref{classificationofcplxprojkleinian}. It corresponds to Cases 7)--13) in Theorem \ref{verylongthm}.

\proofstep{Toric subgroup}
Proposition \ref{notorickleinian} says that, under the hypothesis that $\Gamma$ is not cyclic, if $\Gamma$ is conjugate in $\Bir(Y)$ to a subgroup of the toric subgroup then it contains no loxodromic elements. Thus Case 6) of Theorem \ref{strongTitsalternative} is impossible.

\proofstep{Non-elementary subgroup}
Theorem \ref{kleinianareelementary} says that $\Gamma$ is an elementary subgroup of $\Bir(Y)$ under the hypothesis that $X$ is not of class VII. Thus Case 7) of Theorem \ref{strongTitsalternative} is ruled out.

\paragraph{Acknowledgement}
I would like to thank warmly Serge Cantat for introducing me to this subject and for his constant support. I would like to thank Benoît Claudon, Laura DeMarco, Bertrand Deroin, Philippe Eyssidieux, Adolfo Guillot, Bruno Klingler, Frank Loray, Pierre Py, Fr\'ed\'eric Touzet, Junyi Xie for discussions and comments. I would like to thank the anonymous referees for many helpful comments.

\nocite{}

\bibliographystyle{alpha}

\bibliography{biblio}

\end{document}